\documentclass{article}
\usepackage{amsmath,mathrsfs}
\usepackage{graphicx}
\usepackage[top=21.5mm, bottom=22mm, left=21.5mm, right=21.5mm]{geometry}
\usepackage{amsfonts}
\usepackage{mathtools}
\usepackage{subcaption}
\usepackage{amssymb}
\usepackage{amsthm}
\usepackage{bbm,mathspec}
\usepackage{comment}
\usepackage{chngpage}
\usepackage[backref]{hyperref}
\hypersetup{
	colorlinks=true,
	linkcolor=blue,
	citecolor=RubineRed,
	filecolor=black,      
	urlcolor=black,
}
\usepackage{enumerate}
\usepackage{makeidx}
\usepackage{thmtools}
\usepackage{thm-restate}
\usepackage[dvipsnames]{xcolor}
\usepackage{tikz}
\usetikzlibrary{calc}
\usetikzlibrary{math}
\usetikzlibrary{patterns}
\usetikzlibrary{decorations.pathreplacing}

\makeindex

\makeatletter
\renewenvironment{proof}[1][\proofname] {\par\pushQED{\qed}\normalfont\topsep6\p@\@plus6\p@\relax\trivlist\item[\hskip\labelsep\bfseries#1\@addpunct{.}]\ignorespaces}{\popQED\endtrivlist\@endpefalse}
\makeatother

\newtheorem{proposition}{Proposition}[section]

\newtheorem{lemma}[proposition]{Lemma}
\newtheorem{theorem}[proposition]{Theorem}

\newtheorem{question}[proposition]{Question}

\theoremstyle{definition}
\newtheorem{definition}[proposition]{Definition}

\newtheorem{remark}[proposition]{Remark}

\newtheorem*{remark*}{Remark}
\newtheorem*{theorem*}{Theorem}
\newtheorem{claim}[proposition]{Claim}

\usepackage[shortlabels]{enumitem}
\usepackage[capitalise]{cleveref}

\AddToHook{cmd/appendix/before}{%
  \crefalias{section}{appendix}%
  \crefalias{subsection}{appendix}% Only needed if you have subsections in your appendix
}

\usepackage{url}

\usepackage[numbers,sort]{natbib}

\newcommand{\FF}{\mathcal{F}}
\newcommand{\GG}{\mathcal{G}}
\newcommand{\HH}{\mathcal{H}}

\newcommand{\II}{\mathcal{I}}
\newcommand{\zpz}{\mathbb{Z}/p\mathbb{Z}}

\newcommand{\mb}{\mathbb}
\newcommand{\mc}{\mathcal}

\newcommand{\floor}[1]{{\lfloor #1 \rfloor}}
\newcommand{\ceil}[1]{{\lceil #1 \rceil}}
\newcommand{\modp}[1]{\textnormal{ (mod }#1\textnormal{)}}

\DeclareMathOperator{\spann}{span}

\title{\vspace{-0.8cm} Sunflowers and Ramsey problems for restricted intersections}
\author{Barnabás Janzer\thanks{Department of Mathematics, ETH Zürich, Switzerland. Research supported in part by SNSF grant
200021-228014. Email: \texttt{\{barnabas.janzer,~zhihan.jin,~benjamin.sudakov\}@math.ethz.ch}} \and Zhihan Jin\footnotemark[1] \and Benny Sudakov\footnotemark[1] \and Kewen Wu\thanks{Computing and Mathematical Sciences Department, Caltech, CA, USA. Email: \texttt{shlw\textunderscore kevin@hotmail.com}.}}
\date{\vspace{-21pt}}

\begin{document}

\maketitle
\begin{abstract}
Extremal problems on set systems with restricted intersections have been an important part of combinatorics in the last 70 years. In this paper, we study the following Ramsey version of these problems. Given a set $L\subseteq \{0,\dots,k-1\}$ and a family $\mathcal{F}$ of $k$-element sets which does not contain a sunflower with $m$ petals whose kernel size is in $L$, how large a subfamily of $\mathcal{F}$ can we find in which no pair has intersection size in $L$? We give matching upper and lower bounds, determining the dependence on $m$ for all $k$ and $L$. This problem also finds applications in quantum computing.

As an application of our techniques, we also obtain a variant of Füredi’s celebrated semilattice lemma, which is a key tool in the powerful delta-system method. We prove that one cannot remove the double-exponential dependency on the uniformity in Füredi's result, however, we provide an alternative with significantly better, single-exponential dependency on the parameters, which is still strong enough for most applications of the delta-system method.
\end{abstract}

\section{Introduction} \label{sec:intro}
A large variety of problems and results in extremal combinatorics deal with estimates on the sizes of families of sets with some restrictions on the intersections of their members. In particular, given $n \ge k \ge 1$ and $L \subseteq [0,k-1]$, one of the central problems in extremal set theory is to determine the largest size of a set system $\mc{F} \subseteq \binom{[n]}{k}$ where any two distinct sets $F,F' \in \mc{F}$ satisfy $|F \cap F'| \in L$.\footnote{Throughout this paper, $[a,b]$ stands for the set $\{a,a+1,\dots,b\}$, $[a]$ denotes the set $[1,a]$, and given a set $A$, $\binom{A}{k}$ is the set of $k$-element subsets of $A$.} 
Such set systems are often referred to as {\em set systems with restricted intersections} or \emph{$(n,k,L)$-systems}.
In the vast body of literature in this area, two celebrated results are the Ray-Chaudhuri--Wilson theorem~\cite{RW75} and the Frankl--Wilson theorem~\cite{FW81}.
The former states that if $L$ contains $s$ elements, then $|\mc{F}| \le \binom{n}{s}$; the latter states that if the elements of $L$ correspond to $s$ residues modulo some prime $p$ while $k$ corresponds to another residue, then $|\mc{F}| \le \binom{n}{s}$.
Both bounds are tight in general and have found various applications, for example, in explicit constructions of Ramsey graphs, combinatorial geometry, coding theory, the chromatic numbers of Euclidean spaces, and more.
We refer the reader to the survey by Frankl and Tokushige~\cite{FT16} and the book by Babai and Frankl~\cite{babai2022linear} for a detailed discussion on set systems with restricted intersections.

In this paper, for convenience, instead of $|F \cap F'| \in L$, we will mostly talk about the opposite condition that $|F\cap F'| \notin L$. Note that this is equivalent to replacing the set $L$ by its complement $[0,k-1]\setminus L$.
We say that a set system $\mc{F}$ {\em avoids intersections of size in $L$} if no two sets $F,F'\in\mc{F}$ satisfy $|F\cap F'| \in L$.

One way to study pairwise intersections in $\mc F$ is to build an auxiliary graph $G_\mc{F}$ whose vertices are the sets in $\mc F$ and two sets $F,F' \in \mc{F}$ are connected if and only if $|F\cap F'|\in L$. 
Then $\mc F$ being an $(n,k,L)$-system is equivalent to $G_{\mc{F}}$ being a clique; and $\mc F$ being a set system that avoids intersections of size in $L$ is equivalent to $G_{\mc{F}}$ being an independent set. Thus, we will say that sets $F_1,\dots,F_t$ form an {\em $L$-clique} of size $t$ if $|F_i\cap F_j| \in L$ for all $1 \le i < j \le t$.

One of the most well-known results in discrete mathematics is {\em Ramsey's theorem}, which asserts that in any graph, the clique number and independence number cannot both be small, see \cite{CFS15} for a survey of the rich history of this area. Therefore it is natural to study the following Ramsey-type questions  for $L$-cliques in set systems: \emph{if a set system $\mc{F} \subseteq \binom{[n]}{k}$ contains no $L$-clique of size $m$, how large a subfamily are we guaranteed to find that avoids intersections of size in $L$}?

Among the various ways to form $L$-cliques, one particularly interesting structure is given by the so-called {\em sunflowers} (also called {\em delta-systems}~\cite{DEF78}).
A sunflower (with $t$ \emph{petals}) is a collection of distinct sets $F_1,\dots,F_t$ where $F_i\cap F_j=\bigcap_{\ell=1}^t F_\ell$ for every $1 \le i < j \le t$.
Here, the set $\bigcap_{\ell=1}^t F_\ell$ is called the {\em kernel} of this sunflower.
The study of sunflowers has received a lot of attention in the last 60 years. 
Notably, the sunflower conjecture, proposed by Erd\H{o}s and Rado~\cite{ER60}, asks to determine the minimum number of $k$-element sets that guarantees the existence of a sunflower with $t$ petals, with a conjectured answer of $C^k$ for some $C$ depending only on $t$.
This conjecture has applications in combinatorial number theory, extremal graph theory and also computational complexity theory, see, e.g., \cite{rao2023sunflowers}.
Despite significant efforts, a solution to this conjecture remains elusive. 
A recent breakthrough in this direction was obtained by Alweiss, Lovett, Wu, and Zhang~\cite{ALWZ21}.
Their ideas have soon found wide applications not only within combinatorics but also in theoretical computer science and probability theory~\cite{fractionalKahnKalai,PP24KahnKalai,KZ24}.

While sunflowers might at first appear like natural but very specific instances of $L$-cliques, in some cases they are essentially equivalent notions. Indeed, Deza~\cite{Deza74} showed that for every singleton set $L$, any $L$-clique of at least $k^2-k+2$ sets must be a sunflower.

We say that a sunflower is an {\em $L$-sunflower} if the size of its kernel is in $L$.
Our question above regarding $L$-cliques has a natural counterpart with respect to $L$-sunflowers.
In this paper, we focus on the following two questions, where we think of $k$ as a small (fixed) number compared to $m$.
\begin{question}\label{question: no clique or sunflower}
    Let $k \ge 1$ and $\mc{F}$ be a set system of $k$-element sets.
    \begin{itemize}
        \item If $\mc{F}$ contains no $L$-clique of size $m$, how large a subfamily are we guaranteed to find that avoids intersections of size in $L$?
        \item If $\mc{F}$ contains no $L$-sunflower with $m$ petals, how large a subfamily are we guaranteed to find that avoids intersections of size in $L$?
    \end{itemize}
\end{question}

\noindent
We note that not having an $L$-clique of size $m$ is a stronger restriction than not having an $L$-sunflower with $m$ petals. Thus, a lower bound for the sunflower problem also implies the same lower bound for the clique question.

As mentioned above, we can define an auxiliary graph $G_\FF$ on vertex set $\FF$ by joining two vertices if their intersection has size in $L$. Then our assumptions correspond to bounding the clique number of this auxiliary graph, and we want to find a large independent set. Using the well-known bounds for off-diagonal Ramsey numbers \cite{erdos1935combinatorial}, an $N$-vertex graph with no clique of size $m$ has an independent set of size  $N^{1/m-o(1)}$.
Interestingly, the answers to both our questions turn out to be linear in $|\mc{F}|$, much better than the bound for off-diagonal Ramsey numbers.
In particular, the following theorem shows that if a set system $\mc{F} \subseteq \binom{[n]}{k}$ contains no $L$-sunflower with $m$ petals (in particular, if it contains no $L$-clique of size $m$), then there exists a subfamily of size $\Omega_k(m^{-k} |\mc{F}|)$ that avoids intersections of size in $L$. (Throughout this paper, we use $O_{k}(\cdot)$, $\Omega_{k}(\cdot)$ and $\Theta_{k}(\cdot)$ to hide multiplicative factors depending only on $k$.)

\begin{theorem}\label{theorem: forbidding sunflowers}
    Let $k \ge 1$ and $m \ge 2$ be integers, and let $L\subseteq [0,k-1]$. Suppose that a finite set system $\mc{F}$ of $k$-element sets contains no $L$-sunflower with $m$ petals (in particular, this holds if $\FF$ has no $L$-clique of size $m$).  Then, writing $a =\min\{i \ge 0: i \notin L\}$ and $b= \min\{i \ge a: i \in L \textnormal{ or } i = k\}$, there exists a subfamily $\FF'\subseteq \FF$ of size 
    $$
    |\FF'|=\Omega_k\!\left(m^{-k+b-a}\cdot |\mc{F}|\right)
    $$
    that avoids intersections of size in $L$.
    Moreover, this bound is tight for $L$-sunflowers for all $L \subseteq [0,k-1]$.
\end{theorem}

Note that the numbers $a$ and $b$ are defined so that $L$ is formed by
\(
\{0,1,\dots,a-1\}\cup\{b\}\) and arbitrary elements larger than $b$,
except if $L$ is of the form $\{0,\dots,a-1\}$, in which case $b$ is defined to be $k$.
The theorem above answers Question \ref{question: no clique or sunflower} for $L$-sunflowers completely (in terms of the dependency of $m$) and gives a nontrivial lower bound for all $L$-cliques.
In general, the question for $L$-cliques can be more delicate and we will discuss it in \cref{sec:intro_modular}.
Since the size of $\mc{F}'$ depends linearly on $|\mc{F}|$, we define the following notation.
\begin{definition}
    Let $k \ge 1, m \ge 2$ be integers, and let $L\subseteq [0,k-1]$.
    Define $r(k,L,m)$ (respectively, $r_{\text{sf}}(k,L,m)$) to be the supremum over all $\alpha\ge 0$ for which whenever a finite set system $\mc{F}$ consisting of $k$-element sets contains no $L$-clique of size $m$ (respectively, no $L$-sunflower with $m$ petals), then there exists a subfamily $\mc{F}'\subseteq \mc{F}$ with $|\mc{F}'|\ge \alpha |\mc{F}|$ that avoids intersections of size in $L$.
\end{definition}
In this language, \cref{theorem: forbidding sunflowers} tells that $$r(k,L,m) \ge r_{\textnormal{sf}}(k,L,m) = \Theta_k(m^{-k+b-a}).$$

\cref{theorem: forbidding sunflowers} gives an interesting connection between two famous problems: the Erd\H{o}s--S\'os problem and the Duke--Erd\H{o}s problem.
The Erd\H{o}s--S\'os problem~\cite{E75} asks for the largest size of a family $\mc{F} \subseteq \binom{[n]}{k}$ that avoids intersection size $\ell$, and the Duke--Erd\H{o}s problem~\cite{DE77} asks for the largest size of a family $\mc{F} \subseteq \binom{[n]}{k}$ with no sunflower with $m$ petals whose kernel has size $\ell$.
Note that the latter problem is a generalisation of the former by taking $m=2$. We will show that~\cref{theorem: forbidding sunflowers} enables us to give a connection for general $m$.
Both problems have received a lot of attention.
Among them, the first significant progress on the Erd\H{o}s--S\'os problem was made by Frankl and F\"uredi~\cite{FF85} in 1985, who determined the correct order of magnitude $\Theta_k\big(n^{\max(\ell,k-\ell-1)}\big)$; see \cite{FF85,EKL24,KZ24} for more refined results.
Recently, Brada{\v{c}}, Buci{\'{c}}, and Sudakov~\cite{BBS23} managed to show that for the Duke--Erd\H{o}s problem the answer is $\Theta_k\big(n^{\max(\ell,k-\ell-1)}m^{\min(k-\ell,\ell+1)}\big)$; see also \cite{KN24} for more refined results when $k \ge 2\ell+3$.

Using \cref{theorem: forbidding sunflowers} for $L=\{\ell\}$, it is easy to see that when $\ell \in [k-1]$, then the answer for the $m$-petal Duke--Erdős problem is at most a factor of $O_k(m^{k-\ell})$ times bigger than the answer for the Erdős--Sós problem.
Indeed, if $\mc{F}\subseteq \binom{[n]}{k}$ contains no sunflower with $m$ petals and kernel size $\ell$, then by \cref{theorem: forbidding sunflowers}, there is a subfamily $\mc{F}'\subseteq \mc{F}$ with $|\mc{F}'|=\Omega_k(m^{-k+\ell} |\mc{F}|)$ that avoids intersections of size $\ell$.
In particular, the result of Frankl and F\"uredi~\cite{FF85} mentioned above implies  an upper bound of $O_k(n^{\max(\ell,k-\ell-1)} m^{k-\ell})$ for the $m$-petal problem if $\ell \ge 1$.
We note that in the proof Brada{\v{c}}, Buci{\'{c}}, and Sudakov~\cite{BBS23} of the general bound $\Theta_k(n^{\max(\ell,k-\ell-1)}m^{\min(k-\ell,\ell+1)})$, the ``balanced'' case $\ell=(k-1)/2$ plays an important role, since all other cases can be reduced to the balanced case.
Notably, the above connection recovers their result when $\ell \ge (k-1)/2$, which includes the ``balanced'' case.\smallskip

To obtain our lower bound in~\cref{theorem: forbidding sunflowers}, we use a technique that we call \emph{colour certificates}. 
To demonstrate how colourings can be useful, consider a family $\mc{F} \subseteq \binom{[n]}{k}$. We want to find a large subfamily $\mc{F}'\subseteq \mc{F}$ with the property that for all $A$ with $|A|\in L$ we have some special element $\varphi(A)\in [n]\setminus A$ such that for all $F\in\FF'$,
\begin{equation}\label{equation: special element}
   \textnormal{if $A\subseteq F$ then $\varphi(A)\in F$.}
\end{equation}
Clearly, this implies that no two sets in $\FF'$ can have intersection precisely $A$.
The idea of constructing such a $\varphi$ has already been used by F\"uredi~\cite{F83}. Notice that since $A$ is not the kernel of any $m$-petal sunflower, $\{F\setminus A:F\in\FF,A\subseteq F\}$ has a vertex cover of size at most $mk$, which we denote as $\psi(A)$.
We want to choose the special elements $\varphi(A)$ from $\psi(A)$ in such a way that there are many sets $F$ satisfying \eqref{equation: special element} for all $A\subseteq F$ of size in $L$. To achieve this, we introduce a colouring of the ground set $[n]$. If the number of colours $M$ that we use is small, yet every $\psi(A)$ is rainbow, then we are able to encode all the information about $\psi(A)$ by the colours and reduce our problem to a simple design problem on the subsets $\binom{[M]}{k}$. Then, we can take $\FF'$ to be the collection of all sets whose vertex colours form a $k$-set belonging to the chosen design. See \cref{section: colour certificate} for more details.

\subsection{Forbidding \texorpdfstring{$L$}{L}-cliques in the modular setting}\label{sec:intro_modular}

For the problem of forbidding $L$-cliques in general, it seems hard to provide a complete characterization similar to \cref{theorem: forbidding sunflowers}.
An important and well-studied problem for set systems with restricted intersections is when the allowed intersection sizes (i.e., the elements of $L$) correspond to some residues modulo a prime $p$. 
As mentioned before, a celebrated theorem of Frankl and Wilson~\cite{FW81} states that if $p$ is a prime and $L\subseteq \zpz$ with $k\textrm{ mod }p\not \in L$ and $|L|=s\leq k$, then whenever $\FF\subseteq \binom{[n]}{k}$ is such that every distinct pair has intersection size in $L$ modulo $p$, then $|\FF|\leq \binom{n}{s}$. 

It is therefore natural to study whether our general lower bound $r(k,L,m)=\Omega_k(m^{-k})$ can be similarly improved when $L$ satisfies some modularity conditions. In this direction, we obtain the following result. With a slight abuse of notation, when $L\subseteq \zpz$ for some prime $p$, we will write $r(k,L,m)$ to mean $r(k,L',m)$ for $L'=\{\ell\in [0,k-1]:\ell\!\!\mod p\in L\}$. 

\begin{theorem}\label{theorem: modular clique}
    Let $p$ be a prime, and let $k \ge p$ and $m \ge 2$ be integers.
    Let $L\subseteq \zpz$ have size $p-1$.
    Then
        $r(k,L,m)=\Omega_k\left(m^{-\lceil k/p\rceil-1}\right)$ and $r(k,L,m)=O_k\left(m^{-\floor{k/p}+1}\right).
    $
    Moreover, when $L=(\zpz)\setminus\{0\}$ and $k$ is a multiple of $p$, we have $$
        r(k,L,m) = \Theta_k\left(m^{-k/p}\right).
    $$
\end{theorem}
Surprisingly, as we shall see in the proof, the lower bounds above do not require $p$ to be a prime.
In addition, the residue missing from $L$ can be any residue, in contrast with the setting of the Frankl--Wilson theorem, where $k$ needs to have a different residue modulo $p$ from all $\ell \in L$.
This is because in our problem (Question~\ref{question: no clique or sunflower}), one can add some fixed number of new ``dummy'' elements to each set in our set system to increase the uniformity (hence changing the value of $k$ mod $p$) while keeping the intersections, whereas this operation is not possible in the setting of the Frankl--Wilson theorem since the cardinality of the underlying universe is fixed to be $n$. Furthermore, \cref{theorem: modular clique} demonstrates a difference in behaviour between the sunflower problem (resolved by \cref{theorem: forbidding sunflowers}) and the clique question, which might be surprising in light of Deza's result stating that any large $\{\ell\}$-clique must be a sunflower.

Although \cref{theorem: modular clique} is a very natural case of the modular problem, where $L$ corresponds to all but one possible residues, 
this result was also motivated by the problem that appeared in the context of quantum computing, and has some interesting consequences there.
We will discuss this in \cref{sec:intro_fermionic}.

The next question is to consider the case where $L$ corresponds to fewer than $p-1$ residues, for example, only one residue. We are not able to fully resolve this case, but we have some good upper bounds. For example, in \cref{sec: modular things}, we show that if $L$ consists of a singleton element in $\zpz$, then $r(k,L,m)=O_k\big(m^{-(p-1)k/p+p^2} \big)$ (see~\cref{theorem: modular clique upper bound}). This suggests that when $L$ corresponds to precisely one residue, the exponent of $m$ might be $\frac{(p-1)k}{p}$ plus some additive term depending only on $p$.
More generally, almost the same argument works to provide upper bounds $O_k\big(m^{-\frac{(p-1)k}{ps}+p^2} \big)$ if $L\subseteq \zpz$ consists of $s$ residues.

\subsection{An efficient delta-system lemma}\label{sec: intro furedi}
The methods we developed to prove \cref{theorem: forbidding sunflowers} can be used to prove a variant of Füredi’s celebrated semilattice lemma~\cite{F83}, which has significantly better, single-exponential dependence on the parameters. 
From now on, for simplicity, we refer to Füredi's result~\cite{F83} as \emph{F\"uredi's delta-system lemma}.

F\"uredi's delta-system lemma is a fundamental ingredient in the powerful delta-system method, see, e.g.,~\cite{MV16,Furedi91}.
The idea behind the delta-system method is to make use of sunflowers to find certain structures in complicated set systems. For example, a very useful property of sunflowers is given by the observation that if $m>k$, the $k$-element sets $A_1,\dots,A_m$ form a sunflower with kernel $K$, and $B$ is any other $k$-element set, then the intersection $B\cap K$ is realised by the intersection with some petal: $B\cap K=B\cap A_i$. Similarly, if we have two large sunflowers, then their kernels' intersection is realised as the intersection of two petals. These observations enable us to find sets in our family satisfying certain intersection properties by studying sunflowers and their kernels.
This approach was initiated by Deza, Erd\H os, and Frankl~\cite{DEF78} and Frankl~\cite{F78}.
Later, F\"uredi~\cite{F83} proved his important structural result mentioned above, which turned out to be a very powerful tool in extremal combinatorics, for example, in the study of the Erd\H{o}s--S\'os problem~\cite{FF85} and in Chv\'atal's problem on simplices~\cite{FF87}.
See also \cite[Section 4.3]{MV16} for applications in Tur\'an problems for expansions of trees. Very recently, Jiang and Longbrake~\cite{jiang2024number} developed a supersaturation variant of the delta-system method, and used it to count the number of $H$-free hypergraphs for a large family of degenerate $k$-graphs $H$.

In order to describe the statement of F\"uredi's delta-system lemma, we need some additional notation.
A set system $\mc{F}\subseteq \binom{[n]}{k}$ is called {\em $k$-partite} if there exists a partition $[n]=V_1\cup\dots\cup V_k$ such that every $F \in \mc{F}$ intersects every part in precisely one element.
Let $\pi:[n]\to[k]$ be the projection such that $a\in V_{\pi(a)}$ for all $a\in [n]$, and we write $\pi(A)=\{\pi(a): a\in A\}$ when $A\subseteq [n]$.
For every $F \in \mc{F}$, define $I(F,\mc{F}) := \{F \cap F': F' \in \mc{F}\setminus \{F\}\}$.
With slight abuse of notation, $\pi(I(F,\mc{F}))$ is the projected set system $\{\pi(J):J\in I(F,\mc{F})\}$ on $[k]$.
In addition, we say a set system $\mc{I}$ is {\em intersection-closed} if $I\cap I' \in \mc{I}$ for all $I,I' \in \mc{I}$.

\begin{theorem}[F\"uredi's delta-system lemma~\cite{F83}]\label{lemma: furedi}
    Suppose $\mc{F}\subseteq \binom{[n]}{k}$ and let $m \ge 2$. 
    Then, writing $\alpha_{k,m}=(2^k km)^{-2^k}$, there exists a $k$-partite subfamily $\mc{F}'\subseteq \mc{F}$ with $|\mc{F}'| \ge \alpha_{k,m}|\mc{F}|$ such that 
    \begin{enumerate}[(1)]
        \item for every $F_1, F_2 \in \mc{F}'$ distinct, there is a sunflower in $\mc{F}'$ with $m$ petals whose kernel is $F_1\cap F_2$;\label{item:Furedimain}
        \item for every $F \in \mc{F}'$, $I(F,\mc{F}')$ is intersection-closed (if $m \ge k + 1$);\label{item:Furediintersection-closed}
        \item there exists a set system $\mc{I} \subseteq 2^{[k]}\setminus \{[k]\}$ such that $\pi(I(F,\mc{F}'))=\mc{I}$ for every $F \in \mc{F}'$.\label{item:Furedisame}
    \end{enumerate}
\end{theorem}

Note that this result immediately gives some very weak but positive lower bound for the numbers $r_{\text{sf}}(k,L,m)$ (and hence also $r(k,L,m)$) introduced in the previous section. Indeed, if we apply F\"uredi's delta-system lemma for a family $\mc{F}$ of $k$-element sets that contains no $L$-sunflowers with $m$ petals, we acquire a subfamily $\mc{F}'$ such that $|\mc{F}'|=\Omega_k\big(m^{-2^k}|\mc{F}|\big)$ and for any distinct $F_1,F_2\in \mc{F}'$, $F_1\cap F_2$ is the kernel of a sunflower in $\mc{F}'$ with $m$ petals.
This means that any distinct $F_1,F_2 \in \mc{F}'$ must have $|F_1\cap F_2|\notin L$, i.e., $\mc{F}'$ avoids intersections of size in $L$.
This immediately proves $r_{\text{sf}}(k,L,m) =\Omega_k\big(m^{-2^k}\,\big)$, a significantly weaker statement than \cref{theorem: forbidding sunflowers}.
Indeed, this was our starting point to look for a better dependency in F\"uredi's delta-system lemma.

Although the delta-system lemma has been extensively used over the years, it seems that there have not been any improvements since F\"uredi's original paper~\cite{F83} on the best possible value of $\alpha_{k,m}$ for which the conclusion of this statement holds. 
In his paper F\"uredi comments that the proof gives an extremely small value of $\alpha_{k,m}$ and the only upper bound he has is $(m-1)^{-k}$. Using algebraic tools, we give a construction that, rather surprisingly, shows that the double-exponential dependence on $k$ is in fact necessary for the full strength of the delta-system lemma.

\begin{theorem}\label{thm:doubleexponential}
    In \cref{lemma: furedi}, we must take $\alpha_{k,m}\leq k!\cdot2^{-\binom{k}{\ceil{k/2}}}$ for all $m\ge 2$ and $k \ge 3$.
\end{theorem}

\cref{thm:doubleexponential} shows that in its classical, strong form, one cannot hope for efficient bounds on the constant in Füredi's delta-system lemma. Thus, it is natural to look for variants of Füredi's delta-system lemma for which better (say, single-exponential) bounds on the constant $\alpha_{k,m}$ hold, but the conclusion is still strong enough for applications. Our methods give the following statement, which satisfies these goals for many applications, as discussed below.
\begin{theorem}\label{lemma: weak furedi}
    Suppose $\mc{F}\subseteq \binom{[n]}{k}$ and let $m \ge 2$.
    Then, writing $\beta_{k,m}=(25\cdot2^k km)^{-k}$, there exists a subfamily $\mc{F}'\subseteq \mc{F}$ with $|\mc{F}'| \ge \beta_{k,m}|\mc{F}|$ such that 
    \begin{enumerate}[(a)]
        \item for every $F_1,F_2 \in \mc{F}'$ distinct, there is a sunflower in $\mc{F}$ with $m$ petals whose kernel is $F_1\cap F_2$;\label{item:weakfuredimain}

        \item in fact, for every $F \in \mc{F}'$, there exists an intersection-closed set system $\mathcal{I}_F$ such that $I(F,\mc{F}')\subseteq \mathcal{I}_F\subseteq 2^F \setminus \{F\}$ and every element of $\mathcal{I}_F$ is the kernel of a sunflower with $m$ petals in $\FF$.\label{item:weakfurediintersectionclosed}
    \end{enumerate}
\end{theorem}
There are three main differences between \cref{lemma: weak furedi} and \cref{lemma: furedi}. 
Firstly, the main advantage is that we acquire a significantly better dependency on $m$ and $k$.
In particular, the dependency on $k$ is single exponential and that of $m$ is $m^k$. Note that $m^k$ is necessary by considering a $k$-partite set system $\mc{F}$ where each part has size $m-1$ and every possible edge is included. 
Typically, in applications, our stronger bound gives a better dependency in statements which hold for `large enough $n$'.
Secondly, the sunflower we could guarantee is in $\mc{F}$ but not necessarily in $\mc{F}'$.
While this is a weaker conclusion, this distinction is in fact not important in many cases. In fact, one can replace Füredi's delta-system lemma by \cref{lemma: weak furedi} in some of the most well-known applications of the delta-system method, for example, in the proofs of Frankl and Füredi when studying set systems with restricted intersections~\cite{FF85} (i.e., the Erdős--Sós problem), and Chv\'atal's problem on simplices~\cite{FF87}. In \cref{section:weakfuredi}, we demonstrate two further applications of \cref{lemma: weak furedi}. However, the stronger property of the sunflower being contained in $\FF'$ (guaranteed by Füredi's result) can sometimes be useful if we need a recursive embedding argument, for example, in proofs for the Tur\'an numbers of expansions of trees (see~\cite{MV16}).
The last difference is that we cannot guarantee \ref{item:Furedisame}, i.e., that all intersection-patterns $I(F,\mathcal{F}')$ are isomorphic via a projection map. Again, in most applications, this property is convenient but not actually used. However, it can occasionally be necessary for some applications -- for example, once again, in the context of some Tur\'an problems for expansions (see \cite[Section 4.3]{MV16}).

Our proof of \cref{lemma: weak furedi} is in fact very simple and uses a special case of our colour certificate technique that we introduced to prove \cref{theorem: forbidding sunflowers}.

\subsection{Connections to quantum computing}\label{sec:intro_fermionic}

Curiously, the problem of forbidding $L$-cliques in the modular setting (see \cref{sec:intro_modular}) appears naturally in quantum computing.
To describe the connection and application, we give a brief introduction to necessary quantum terminologies.

In quantum physics, a \emph{quantum state} embodies the knowledge of a quantum system, and an \emph{observable} is a physical property or quantity that can be measured.
One important question in quantum computing is \emph{shadow tomography} \cite{aaronson2018shadow,buadescu2021improved}, which aims to learn various properties (i.e., expectation values under observables) of unknown quantum states.
Given many copies of an unknown quantum state $\rho$ and a fixed set of observables $O_1,\ldots,O_s$, a shadow tomography algorithm aims to learn, up to an $\varepsilon$ error, the expectation values of $\rho$ under the measurement of each observable $O_i$, $i\in[s]$.
Due to the cost of obtaining and storing copies of quantum states and the limitation of physical implementation, a recent work by King, Gosset, Kothari, and Babbush~\cite{king2025triply} put forth the concept of being \emph{triply efficient} for shadow tomography algorithms, which asks for:
\begin{itemize}
    \item \emph{Sample Efficiency.} The algorithm should not use too many copies of $\rho$.
    \item \emph{Time Efficiency.} The algorithm should run fast.
    \item \emph{Measurement Efficiency.} The algorithm should only perform measurements on constantly many copies of $\rho$ each time.
\end{itemize}
We refer interested readers to \cite{king2025triply} for formal definition and more literature on shadow tomography.

Observables of particular interest here are \emph{local fermionic observables}, which arise from many-body physics and quantum chemistry such as electronic structure of molecules.
Indeed, $k$-local fermionic observables capture important physical properties even for $k$ as small as $1$ or $2$.
See \cite[Section I]{king2025triply} and the references within.
In light of the motivations above, one of the main results of~\cite{king2025triply} is a triply efficient shadow tomography algorithm for $O(1)$-local fermionic observables. 

\begin{theorem}[{\cite[Theorem 6]{king2025triply}}]\label{thm:fermionic_prior}
    For $k$-local fermionic observables on $n$ qubits, there is a triply efficient shadow tomography algorithm that uses $O_k\left((1/\varepsilon)^{r_k}\cdot\log(n)\right)$ total copies of the unknown quantum state, where $r_1=4$, $r_2=18$, $r_3=110$, and $r_k=O\left((2k)^{k+1}\right)$ in general.
\end{theorem}

While they anticipate that their algorithm could find applications in quantum simulations of chemistry and fermionic physics, their fast growing $r_k$ renders the bound impractical for $k\ge2$. The authors made considerable attempts in their paper and proposed various conjectures to get better bounds on $r_k$.

Using \cref{theorem: modular clique}, we exponentially improve their estimates on $r_k$. Moreover, our bound is significantly better even for small values of $k$ such as $k=2$ or $k=3$, making practical implementations of the algorithm feasible.

\begin{theorem}\label{thm:fermionic_new}
In \Cref{thm:fermionic_prior}, we can set $r_k=2(k+1)$ for all $k\ge1$.
\end{theorem}

To give some details, $k$-local fermionic observables on $n$ qubits have a one-to-one correspondence to $2k$-uniform sets over $[2n]$ (see \cite[Section I]{king2025triply}).
By some nontrivial preprocessing and reductions (see \cite[Section IV]{king2025triply}), it appears that one needs to upper bound the fractional chromatic number $\chi$ of the graph $G_{\mc F}$ for some $\mc F\subseteq\binom{[2n]}{2k}$ and the intersection sizes $L=\{1,3,5,\dots,2k-1\}$.
It is additionally guaranteed that the clique number of $G_{\mc F}$ is at most $m=O(1/\varepsilon^2)$.
Their analysis \cite[Lemma 2]{king2025triply} yields an upper bound of $\chi\le m^{O((2k)^{k+1})}$, which leads to the exponential $r_k$ in \Cref{thm:fermionic_prior} in the end.

In our approach, \cref{theorem: modular clique} gives a large independent set in $G_{\mc F}$ of roughly $m^{-k}\approx(1/\varepsilon)^{-2k}$ fraction, and our proofs in fact also show that $G_{\mc F}$ has fractional chromatic number $\chi$ at most roughly $(1/\varepsilon)^{2k}$.
Putting this back into their reductions gives the final bound of $r_k=2(k+1)$ claimed in \Cref{thm:fermionic_new}.
We provide more details of this analysis in \cref{section: fractional chromatic}.

\section{Colour certificates}\label{section: colour certificate}
In this section, we prove the main ingredient for the proofs of the lower bounds in \cref{theorem: forbidding sunflowers,lemma: weak furedi}.
As mentioned before, given a family $\mc{F} \subseteq \binom{[n]}{k}$, we want to find a ``large'' subfamily $\mc{F}'\subseteq \mc{F}$ such that for each $A \subseteq [n]$ that is not the kernel of any ``large'' sunflower, we enforce that all $F \in \mc{F}'$ containing $A$ must contain some special vertex $\varphi(A)$.
The advantage is that in $\mc{F}'$, no two sets can have intersection precisely $A$.
The idea of constructing such a $\varphi$ has been exploited by F\"uredi~\cite{F83}.

To construct $\mc{F}'$, we use a new approach which we call ``colour certificates''.
Notice that $A \subseteq [n]$ is not the kernel of any ``large'' sunflower if and only if $\mc{F}(A):=\{F\setminus A: A\subseteq F \in \mc{F}\}$ has no ``large'' matching, i.e., a ``large'' collection of disjoint sets.
Hence, if $A$ is not the kernel of any ``large'' sunflower, then $\mc{F}(A)$ has a ``small'' vertex cover, which we denote as $\psi(A)$.
In spirit, our goal is to colour all the elements $[n]$ by $M$ colours such that $M$ is ``small'' and every $\psi(A)$ is rainbow, i.e., no two elements in $\psi(A)$ share the same colour.
Then, we are able to encode all the information of $\psi(A)$ by the colours and reduce our problem to a simple design problem in $\binom{[M]}{k}$.
The formal statement is given by the following lemma.
\begin{lemma}\label{lemma: main}
    Let $n\ge k \ge 1$, $m \ge 2$ and $\ell \in [0,k-1]$ be integers.
    Suppose $\mc F \subseteq \binom{[n]}{k}$.
    Then, there exists a subfamily $\mc F' \subseteq \mc F$ with 
    $$
    |\mc F'| \ge \left(\binom{k}{\ell} \big(8e\cdot 2^kkm\big)^{k-\ell}\right)^{-1}\cdot  |\mc F| 
    $$
    such that the following holds.
    For all $A \subseteq [n]$ satisfying $|A|\in[\ell,k-1]$ and such that $A$ is not the kernel of any sunflower with $m$ petals {in $\mc{F}$}, there exists an element $\varphi(A)\in [n]\setminus A$ such that every $F \in \mc{F}'$ containing $A$ must also contain $\varphi(A)$.
\end{lemma}
\begin{proof}
    As mentioned above, we first define $\psi: \binom{[n]}{\ell}\cup\dots\cup\binom{[n]}{k-1} \to 2^{[n]}$ as follows.
    For every $A \subseteq [n]$ with $\ell \le |A| < k$, let $\{F_1,\dots,F_t\}\subseteq \mc{F}$ be a sunflower in $\mc{F}$ with kernel $A$ that obtains the maximum number of petals.
    Then, set 
    $$
    \psi(A):=\begin{cases}
    \bigcup_i (F_i\setminus A) & t<m,\\
    \emptyset & t\ge m.
    \end{cases}
    $$
    In the case $t<m$, by the maximality of $t$, every $F \in \mc{F}$ containing $A$ satisfies that $F \setminus A$ and $\psi(A)$ intersect.
    In addition, $|\psi(A)| \le (k-|A|)(m-1)< (k-\ell)m$.

    Next, we sample a colouring $\chi: [n] \to [M]$ uniformly at random with $M:=4\cdot 2^k k(k-\ell) m$.
    We say that $F \in \mc{F}$ is {\em good} if both of the following conditions hold:
    \begin{enumerate}[(1)]
        \item $F$ is rainbow, i.e., $|\chi(F)|=k$;\label{item:rainbow}
        \item for every $A \subset F$ with $|A|\in[\ell,k-1]$ and every $v \in \psi(A) \setminus F$, the colour $\chi(v)$ is distinct from all the colours $\chi(u)$ with $u\in F$.\label{item:psicolours}
    \end{enumerate}
    For any fixed $F$, the probability that $F$ violates \ref{item:rainbow} is at most $\binom{k}{2}/M<1/4$ and the probability that $F$ violates \ref{item:psicolours} is at most $2^k\cdot k\cdot (k-\ell)m/M\le 1/4$.
    This means the probability for $F$ to be good is at least $1/2$.
    Hence, the expected number of good sets $F \in \mc F$ is at least $|\mc{F}| / 2$.
    
    Therefore, there exists a colouring $\chi: [n]\to [M]$ and a subfamily $\mc F_1$ consisting of good $F \in \mc{F}$ such that $|\mc{F}_1| \ge |\mc{F}|/2$.
    For every $F \in \mc F_1$, since $F$ is rainbow, its colour palette $\chi(F)$ is in $\binom{[M]}{k}$.
    We claim that there is a partition of all the possible colour palettes $\binom{[M]}{k}$ into at most $K:=\binom{k}{\ell}\binom{M-\ell}{k-\ell}$ classes such that inside each class, any two distinct colour palettes $P_1,P_2$ satisfy $|P_1\cap P_2| < \ell$.
    Indeed, we can build an auxiliary graph $G$ with vertex set $\binom{[M]}{k}$ where two vertices $P_1,P_2$ are connected precisely when $|P_1\cap P_2|\ge \ell$.
    Since the maximum degree of $G$ is at most $\binom{k}{\ell}\big(\binom{M-\ell}{k-\ell}-1\big)<\binom{k}{\ell}\binom{M-\ell}{k-\ell}$, we know the chromatic number of $G$ is at most $\binom{k}{\ell}\binom{M-\ell}{k-\ell}$.
    The claim follows since we can take the colour classes in the corresponding proper vertex-colouring of $G$ to be our $K$ classes of $\binom{[M]}{k}$.

    By the pigeonhole principle, among these $K$ classes of the colour palettes, there exists a class $\mc{P} \subseteq \binom{[M]}{k}$ such that at least a $1/K$ fraction of the members $F \in \mc{F}_1$ have $\chi(F) \in \mc{P}$.
    In other words, $\mc{F}':=\{F \in \mc{F}_1: \chi(F) \in \mc{P}\}$ has size at least $|\mc{F}_1|/K$.
    We complete the proof by showing that this subfamily $\FF'$ satisfies the required conditions.
    
    Indeed, using the standard inequality $\binom{a}{b} \le (ea/b)^b$, we see that 
    $$
    K\le \binom{k}{\ell}\left(\frac{eM}{k-\ell}\right)^{k-\ell}
        = \binom{k}{\ell} (4e \cdot 2^k k m)^{k-\ell},
    $$
    which implies that $|\mc{F}'| \ge |\mc{F}|/(2K) \ge |\mc{F}|/\left(\binom{k}{\ell} (8e \cdot 2^k k m)^{k-\ell}\right)$.
    
    In addition, let $A \subseteq [n]$ be such that $\ell \le |A| < k$ and $A$ is not the kernel of any sunflower in $\mc{F}$ with $m$ petals.
    We need to find $\varphi(A) \in [n]\setminus A$ such that any $F \in \mc{F}'$ containing $A$ must also contain $\varphi(A)$.
    To this end, let $F_1,\dots,F_s$ be all the sets in $\mc{F}'$ that contain $A$.
    We may assume $s \ge 1$ as otherwise we can set $\varphi(A)$ arbitrarily.
    Let $\psi'(A) := \psi(A) \cap (F_1\cup\dots\cup F_s)$ be the ``remaining elements'' among $\psi(A)$.
    We note that $\psi'(A)$ is rainbow, i.e., $|\chi(\psi'(A))|=|\psi'(A)|$.
    To see this, note that if $u,v \in \psi'(A)$ and there is some $i$ with $u,v\in F_i$, then \ref{item:rainbow} guarantees that $u,v$ have different colours.
    Otherwise, say $F_i$ contains $u$ but not $v$. Then \ref{item:psicolours} guarantees that $u,v$ have different colours. Thus, $\psi'(A)$ is indeed rainbow.
    In addition, since $F_1$ is good, we know that $F_1$, and hence also $A$, must be rainbow, i.e., $|\chi(A)|=|A|$.
    This means $|\chi(F_i)\cap \chi(F_j)| \ge |A| = \ell$ for all $i,j$.
    According to our choice of $\mc{P}$, $\chi(F_i)$ and $\chi(F_j)$ must be the same colour palette, thus, $\chi(F_1)=\chi(F_2)=\dots=\chi(F_s)$.
    Recall that $\psi(A)$ and $F_1\setminus A$ intersect.
    Let $\varphi(A) \in \psi(A)\cap (F_1\setminus A)$ be arbitrary, and we claim that $\varphi(A) \in F_i$ for all $1 \le i \le s$.
    Indeed, if $\varphi(A) \notin F_i$, then \ref{item:psicolours} on $F_i$ implies $\chi(\varphi(A)) \notin \chi(F_i)$, but this contradicts that $\chi(F_1)=\chi(F_i)$.
    This proves the claim and completes the proof.
\end{proof}

\section{Forbidding \texorpdfstring{$L$}{L}-sunflowers}\label{sec: forbidding sunflowers}

In this section, we prove \cref{theorem: forbidding sunflowers}.
We start with the lower bound using \cref{lemma: main}.

\begin{proof}[Proof of the lower bound in \cref{theorem: forbidding sunflowers}]
    We first consider all $k$ and all $L$ such that $0 \notin L$.
    The case $L = \emptyset$ is trivially true by taking $\mc{F}'=\mc{F}$, so we may assume $L \neq \emptyset$.
    In this case, $a = 0$ and $b \in \{1,2,\dots,k-1\}$.
    Now, suppose a set system $\mc{F}$ consists of $k$-element sets and contains no $L$-sunflower with $m$ petals.
    By \cref{lemma: main} with $\ell=b$, there exists $\mc{F'}\subseteq \mc{F}$ with $|\mc{F}'|=\Omega_k(m^{-k+b}|\mc{F}|)=\Omega_k(m^{-k+b-a}|\mc{F}|)$ such that for any $A\subseteq [n]$ with $|A| \in L$, there exists an element $\varphi(A) \in [n]\setminus A$ such that every $F\in \mc{F}'$ containing $A$ also contains $\varphi(A)$.
    We show that $\mc{F}'$ avoids intersections of size in $L$.
    Indeed, suppose $F_1,F_2 \in \mc{F}'$ satisfy $|F_1\cap F_2| \in L$.
    Putting $A:=F_1\cap F_2$, the condition on $\mc{F}'$ guarantees $\varphi(A) \in F_1$ and $\varphi(A) \in F_2$, contradicting that $F_1\cap F_2 = A$.
    This proves the lower bound in the case $0\not\in L$.

    Now, suppose $0 \in L$, i.e., $a \ge 1$.
    Recall that $\{0,1,\dots,a-1\} \subseteq L$ and $a \notin L$.
    Suppose that a (non-empty) set system $\mc{F}$ consists of $k$-element sets and contains no $L$-sunflower with $m$ petals.
    In particular, there is no matching (i.e., $\{0\}$-sunflower) with $m$ sets.
    Suppose $F_1,\dots,F_t$ form a maximal matching in $\mc{F}$.
    We know $t < m$ and any $F \in \mc{F}$ must intersect some $F_i$, $1 \le i \le t$.
    By the pigeonhole principle, for some $u^{(1)} \in F_1\cup \dots\cup F_t$, the family $\mc{F}^{(1)}:=\{F\setminus u^{(1)}: F\in\FF, u^{(1)} \in F\}$ has cardinality $|\mc{F}^{(1)}| \ge |\mc{F}|/(kt)\ge |\mc{F}|/(km)$.
    In addition, every set $F \in \mc{F}^{(1)}$ corresponds to a set $F\cup \{u^{(1)}\} \in \mc{F}$.
    This means that $\mc{F}^{(1)}$ contains no $L^{(1)}$-sunflower with $m$ petals where $L^{(1)}:=\{\ell -1: \ell \in L\setminus\{0\}\}$.
    We iterate this $a$ times to acquire $\mc{F}^{(1)},\dots,\mc{F}^{(a)}$ with $L^{(1)},\dots,L^{(a)}$ and distinct $u^{(1)},\dots,u^{(a)}$ such that 
    \begin{itemize}
        \item $|\mc{F}^{(a)}| \ge \frac{1}{km}|\mc{F}^{(a-1)}|\ge\dots\ge \frac{1}{(km)^a} |\mc{F}|$;
        \item every set $F \in \mc{F}^{(a)}$ corresponds to a set $F\cup \{u^{(1)},\dots,u^{(a)}\} \in \mc{F}$;
        \item $\mc{F}^{(a)}$ contains no $L^{(a)}$-sunflower with $m$ petals where $L^{(a)}:=\{\ell: \ell + a \in L\} \subseteq [0,k-a-1]$.
    \end{itemize}
    If $a=k$, then $\mc{F}^{(a)}$ contains precisely the empty set, which means $|\mc{F}| \le (km)^k$.
    In this case, $b=k$ and taking $\FF'$ to be a singleton containing an arbitrary set in $\mc{F}$ gives $|\FF'|=\Omega_k(m^{-k}|\FF|)=\Omega_k(m^{-k+b-a}|\FF|)$.
    If $a < k$, then we know that $0 \notin L^{(a)}$ and $\min\{i \ge 0: i \in L^{(a)} \text{ or } i = k-a\} = b-a$.
    By the case $0 \notin L$ above, we know that $\mc{F}^{(a)}$ contains a subfamily $\mc{F}'^{(a)}$ with $|\mc{F}'^{(a)}|=\Omega_k(m^{-(k-a)+b-a} |\mc{F}^{(a)}|)$ that avoids intersections of size in $L^{(a)}$.
    Take $\mc{F}' := \{F \cup \{u^{(1)},\dots,u^{(a)}\}: F \in \mc{F}'^{(a)}\}$.
    Then 
    $$
    |\mc{F}'|=|\mc{F}'^{(a)}|=\Omega_k(m^{-(k-a)+b-a} |\mc{F}^{(a)}|)=\Omega_k(m^{-k+b-a} |\mc{F}|).
    $$
    In addition, let $F\cup \{u^{(1)},\dots,u^{(a)}\}, F'\cup \{u^{(1)},\dots,u^{(a)}\} \in \mc{F}'$ be any two distinct sets.
    Using that $|F \cap F'|\notin L^{(a)}$, we know $|(F\cup \{u^{(1)},\dots,u^{(a)}\})\cap (F'\cup \{u^{(1)},\dots,u^{(a)}\})| = a+|F\cap F'| \notin L$.
    This shows that $\mc{F}'$ avoids intersections of size in $L$ and completes the proof of the lower bound.
\end{proof}

Next, we prove the upper bound in \cref{theorem: forbidding sunflowers} (i.e., the tightness of the lower bound).
We will show how to reduce the general case to the following result, where $0 \notin L$ and $1 \in L$.

\begin{theorem} \label{theorem: forbidding sunflowers with one but zero}
    Let $k \ge 2$ and $m \ge 2$ be integers and let $L\subseteq [0,k-1]$ such that $0 \notin L$ and $1 \in L$.
    Then, $r_{\textnormal{sf}}(k,L,m) = O_k(m^{-({k-1})})$.
\end{theorem}

Assuming \cref{theorem: forbidding sunflowers with one but zero}, we complete the proof of \cref{theorem: forbidding sunflowers}.

\begin{proof}[Proof of the upper bound in \cref{theorem: forbidding sunflowers}] 
    Recall that we write $a =\min\{i \ge 0: i \notin L\}$ and $b= \min\{i \ge a: i \in L \textnormal{ or } i = k\}$.
    Observe that if $L=\emptyset$ and thus $a=0,b=k$, then $-k+b-a=0$ and so trivially $|\FF'|=O_k(m^{-k+b-a}|\FF|)$. Furthermore, if $m\leq f(k)$ for some arbitrary fixed function $f$, then $|\FF'|\le|\FF|=O_k(m^{-k+b-a}|\FF|)$ clearly holds. Thus, from now on, we may assume that $L\not =\emptyset$ and $m$ is sufficiently large (compared to $k$).
    
    First, assume $L$ is of the form $L=[0,a-1]$, so $b=k$. Let $\FF=\binom{[m-1]}{k}$, so $|\FF|=\Theta_k(m^k)$. Clearly, $\FF$ contains no $L$-sunflowers with $m$ petals (and in fact no $m$-petal sunflower of any kernel size). By the Erdős--Ko--Rado theorem~\cite{EKR61}, if $m$ is large enough and $\FF'\subseteq\FF$ avoids intersections of size in $L$, then 
    $$
    |\FF'|\leq \binom{m-1-a}{k-a}\leq m^{k-a}=O_k(m^{-a}|\FF|)=O_k(m^{-k+b-a}|\FF|).
    $$ 
    
    Next, assume $a=0$, so $b\ge1$ and $L=\{b\}\cup(L\cap[b+1,k-1])$.  Let $k'=k-(b-1)$ and $L'=\{\ell-(b-1):\ell\in L\}$. By \cref{theorem: forbidding sunflowers with one but zero}, we can construct some (non-empty) finite set system $\GG$ of $k'$-element sets such that $\GG$ contains no $L'$-sunflower with $m$ petals, and whenever $\GG'\subseteq \GG$ avoids intersections of size in $L'$, it satisfies $|\GG'|= O_{k'}(m^{-(k'-1)}|\GG|)=O_k(m^{-(k-b)}|\GG|)$. 
    Now construct $\FF$ from $\GG$ by enlarging the ground set by $b-1$ new ``dummy'' elements, and adding these $b-1$ elements to each set in $\GG$. In this way, we have $|\FF|=|\GG|$ and $\FF$ contains no $L$-sunflowers with $m$ petals.
    In addition, whenever $\FF'\subseteq \FF$ avoids intersections in $L$, we can remove those dummy elements and obtain $\GG'\subseteq\GG$ that avoids intersections in $L'$. Hence $|\FF'|=|\GG'|= O_k(m^{-(k-b)}|\GG|)=O_k(m^{-(k-b)}|\FF|)$, as desired.
    
    Finally, assume $a\ge1$ and $L\not =[0,a-1]$, so $a+1\le b\le k-1$ and $L=[0,a-1]\cup \{b\}\cup (L\cap [b+1,k-1])$. Let $k'=k-a$ and $L'=\{\ell-a:\ell\in L, \ell\geq a\}$. As $0\not \in L'$, by the paragraph above, we can find some (non-empty) finite set system $\GG$ of $k'$-element sets such that $\GG$ contains no $L'$-sunflower with $m$ petals, and whenever $\GG'\subseteq \GG$ avoids intersections of size in $L'$, it satisfies 
    $$
    |\GG'|= O_{k'}(m^{-(k-a)+(b-a)}|\GG|)=O_k(m^{-k+b}|\GG|).
    $$
    Let $D$ be a set of size $m-1$ consisting of new elements. For each $X\in\binom{D}{a}$, take a copy $\GG_X$ of our set system $\GG$ over some ground set $V_X$ in such a way that $V_X\cap V_Y=\emptyset$ if $X\not =Y$, and also $V_X\cap D=\emptyset$ for all $X$. Let
    $$
    \FF_X=\{X\cup A:A\in\GG_X\}
    \quad\text{and}\quad
    \FF=\bigcup_{X\in\binom{D}{a}}\FF_X.
    $$
    Note that $\FF$ is a family of $k$-element sets over ground set $D\cup\bigcup_{X\in\binom{D}{a}}V_X$.
    
    First notice that $\FF$ contains no $L$-sunflower with $m$ petals. 
    To see this, assume, for a contradiction, that there is such an $m$-petal sunflower formed by $P_1,\dots,P_m\in\FF$ with kernel size $\ell\in L$.
    Then, either $P_1\cap D,\dots,P_m\cap D$ forms an $m$-petal sunflower or all the sets $P_i\cap D$ coincide.
    But the former does not hold because $|D| < m$, so it must be that $P_1\cap D=\cdots=P_m\cap D=X$ for some $X\in\binom{D}{a}$. 
    This means $P_1\setminus X,\dots,P_m\setminus X$ form an $m$-petal sunflower of kernel size $\ell-a\in L'$ in $\GG_X$, giving a contradiction.
    
    It remains to show that if $\FF'\subseteq \FF$ avoids intersections of size in $L$, then $|\FF'|= O_k(m^{-k+b-a}|\FF|)$. Thus, let $\FF'$ be a subfamily avoiding intersections of size in $L$. 
    Note that for every distinct $X_1,X_2\in \binom{D}{a}$ and $F_1\in\FF_{X_1},F_2\in\FF_{X_2}$, we have $F_1\cap F_2=X_1\cap X_2$ and $|X_1\cap X_2|\in[0,a-1]\subseteq L$.
    Therefore there is some set $X\in \binom{D}{a}$ such that $\FF'\subseteq \FF_X$. Moreover, $\{F\setminus X:F\in\FF'\}$ must be a subfamily of $\GG_X$ that avoids intersections in $L'$, hence 
    $$
    |\FF'|= O_k\left(m^{-k+b}|\GG|\right)=O_k\left(m^{-k+b}\cdot\frac{|\FF|}{\binom{m-1}a}\right)=O_k\left(m^{-k+b-a}|\FF|\right).
    $$
    This completes the proof of the tightness in \cref{theorem: forbidding sunflowers}.
\end{proof}

The full proof of \cref{theorem: forbidding sunflowers with one but zero} is rather technical and deferred to \cref{sec: appendix}. Here we prove the special case $L=\{1\}$ as it already illustrates many of the ideas. (Note that, by the reduction above, even this special case is sufficient to deal with the cases $L=\{\ell\}$ for each $\ell$.) 
%We note that the $L=\{1\}$ case is sufficient for dealing with all cases where $L$ is a singleton by the reduction above.
For simplicity of notation, we will show that $r_{\textnormal{sf}}(k,\{1\},m+1) = O_k(m^{-(k-1)})$. We will make use of the following graph.
\begin{lemma}\label{lemma: digraph}
    For any positive integers $k,m,n$ with $k\geq 2$, there exists a graph $G$ such that the vertex set of $G$ can be partitioned into $V_1,\dots,V_k$ and the following properties hold.
    \begin{enumerate}
        \item Each edge of $G$ goes between $V_i$ and $V_{i+1}$ for some $i\in [k-1]$.
        \item We have $|V_i|=n^{k-i-1}m^{i}$ for all $i\in [k-1]$, and $|V_k|=m^{k-1}$.
        \item For all $i\in[k-1]$, each vertex in $V_i$ has $m$ neighbours in $V_{i+1}$. \label{prop:out}
        \item Each vertex in $V_i$ has $n$ neighbours in $V_{i-1}$ $\forall i\in [2,k-1]$, and each vertex in $V_k$ has $m$ neighbours in $V_{k-1}$. \label{prop:in}
        \item For all $i,j\in[k]$ with $i<j$ and all $u \in V_i$, $v \in V_j$, there is at most one path with $j-i$ edges that starts from $u$, goes through parts $V_{i+1},V_{i+2},\dots,V_{j-1}$, and ends at $v$. \label{prop:paths}
    \end{enumerate}
\end{lemma}
\begin{proof}
    For each $i\in[k-1]$, let $V_i=[m]^{i-1}\times [n]^{k-i-1}\times [m]\times\{i\}$, and let $V_k=[m]^{k-1}\times \{k\}$. 
    We define the graph $G$ on vertex set $V_1\cup\dots\cup V_k$ as follows. 
    For each $i\in [k-1]$, the neighbourhood of $(a_1,\dots,a_{k-1},i)\in V_i$ in $V_{i+1}$ is given by
    $$\{(a_1,\dots, a_{i-1},x,a_{i+1},a_{i+2},\dots,a_{k-1},i+1):x\in[m]\}.$$
    In other words, we replace the $i$th coordinate by an arbitrary element of $[m]$, and increase the last coordinate by $1$. 
    It is straightforward to check that all the conditions are satisfied.
\end{proof}

\begin{proof}[Proof of \cref{theorem: forbidding sunflowers with one but zero} for $L=\{1\}$]
    As mentioned above, we prove $r_{\textnormal{sf}}(k,\{1\},m+1)= O_k(m^{-(k-1)})$ for $k \ge 2,m\ge 1$.
    Let $n$ be a sufficiently large integer compared to $m$, and let $G$ be the graph from~\cref{lemma: digraph} for the parameters $k,m,n$.
    Let $\mc{F}$ be the collection of all $k$-sets of vertices in $G$ which form paths from $V_1$ to $V_k$ of length $k-1$.
    In other words, these are the (vertex sets of) paths that start from $V_1$, go through $V_2,V_3,\dots,V_{k-1}$, and end at $V_k$.
    Clearly, 
    $$
    |\mc{F}|=|V_1|m^{k-1}=n^{k-2}m^k.
    $$
    Also, Properties~\ref{prop:out} and~\ref{prop:in} easily guarantee that there is no sunflower with $m+1$ petals whose kernel has size 1.
    Now, assume that $\FF'\subseteq \FF$ avoids intersections of size 1, i.e., no $F,F' \in \mc{F}'$ satisfies $|F\cap F'|=1$.
    Our goal is to show that $|\FF'|\le (1+o(1)) m^{-(k-1)}|\FF|=(1+o(1)) |V_1|$, where the $o(1)$ term denotes a quantity that converges to $0$ as $n\to\infty$ (with $k,m$ fixed).

    Given an edge $e$ of $G$ between $V_i$ and $V_{i+1}$, we say that $e$ is \emph{light} if $e$ is contained in at most $n^{i-\frac{3}{2}}$ sets in $\FF'$, and \emph{heavy} otherwise.     
    Observe that the number of sets in $\FF'$ that contain a light edge is at most $$
        \sum_{i=1}^{k-1}e_G(V_i,V_{i+1})n^{i-\frac{3}{2}}
        =\sum_{i=1}^{k-1}n^{k-i-1}m^{i+1}n^{i-\frac{3}{2}}
        <kn^{k-\frac{5}{2}}m^k=kn^{-1/2}|\FF|=o\left(m^{-(k-1)}|\FF|\right).
    $$
    Here, $e_G(V_i,V_{i+1})$ stands for the number of edges of $G$ between $V_i$ and $V_{i+1}$.
    Hence, it suffices to show that the number of sets in $\FF'$ formed by paths with $k-1$ heavy edges is small.
    \begin{claim}
        If $n$ is sufficiently large, then for each $i \in [k-1]$, every vertex in $V_i$ is incident to at most one heavy edge to $V_{i+1}$.
    \end{claim}
    \begin{proof}
        Suppose for contradiction that there is some vertex $u\in V_i$ and distinct vertices $v,v'\in V_{i+1}$ such that $uv$ and $uv'$ are both heavy. 
        Pick any $F\in \FF'$ containing the edge $uv$. 
        Whenever $F' \in\FF'$ contains the edge $uv'$, we have $u\in F\cap F'$, thus, by the assumption on $\mc{F}'$, we must have $|F \cap F'| \ge 2$.
        In addition, Property~\ref{prop:paths} guarantees that those $j \in [k]$ with $F\cap V_j=F'\cap V_j$ form an interval.
        So, we must have $i > 1$ and $F\cap V_{i-1}=F'\cap V_{i-1}$.
        Using Properties~\ref{prop:out} and~\ref{prop:in}, the number of such sets $F'$ (agreeing with $F$ on $V_{i-1}$ and $V_i$) is at most $n^{i-2}\cdot m^{k-(i+1)}$, so the edge $uv'$ is contained in at most $n^{i-2}\cdot m^{k-(i+1)}<n^{i-3/2}$ sets in $\FF'$ for $n$ sufficiently large.
        This contradicts our assumption that $uv'$ is heavy and finishes the proof of the claim.
    \end{proof}
    Hence, the number of paths of length $k-1$ (starting at $V_1$ and ending at $V_k$) formed by heavy edges is at most $|V_1|=m^{-(k-1)}|\FF|$. 
    Therefore $|\FF'|\leq o(m^{-(k-1)}|\FF|)+m^{-(k-1)}|\FF|=(1+o(1))m^{-(k-1)}|\FF|$.
\end{proof}

\section{A new variant of the delta-system lemma}\label{section:weakfuredi}

In this section, we discuss F\"uredi's delta-system lemma and our variant (\cref{lemma: weak furedi}) which has a slightly weaker conclusion but a significantly better quantitative dependency. We start by proving \cref{lemma: weak furedi} in \cref{section: weak Furedi proof}, as well as giving a brief discussion about typical ways the delta-system method can be applied. Then, in \cref{section: furedi double exponential}, we show that F\"uredi's delta-system lemma requires a double-exponential dependency on $k$, demonstrating the necessity of sacrificing some properties guaranteed by its conclusion if we aim for stronger bounds. Finally, we discuss two new applications of F\"uredi's delta-system lemma in \cref{section: excluding one residue,section: Frankl--Wilson} (including an application for the modular clique problem, \cref{theorem: excluding one residue}). We prove these results via our \cref{lemma: weak furedi} to demonstrate that for most applications \cref{lemma: weak furedi} is sufficient, allowing us to achieve a better dependency on the uniformity $k$.

\subsection{Our new variant}\label{section: weak Furedi proof}

Let us start by proving \cref{lemma: weak furedi} using \cref{lemma: main}.

\begin{proof}[Proof of \cref{lemma: weak furedi}]
    We apply \cref{lemma: main} for $\mc{F}$ with $\ell=0$ to acquire $\mc F' \subseteq \mc F$ with $|\mc F'| \ge (8e\cdot 2^kkm)^{-k}\cdot |\mc F|$ such that the following holds.
    For any $A \subseteq [n]$ such that $0\le |A| < k$ and $A$ is not the kernel of any sunflower in $\mc{F}$ with $m$ petals, there exists an element $\varphi(A)\in [n]\setminus A$ such that any $F \in \mc{F}'$ containing $A$ must also contain $\varphi(A)$.
    
    We claim this is the desired $\mc{F}'$.
    Indeed, for each $F \in \mc{F}'$, define $$
        \II_F:=\left\{F\cap F_1\cap\dots\cap F_t: t \ge 1,  F_1,\dots,F_t \in \mc{F}'\setminus \{F\} \right\} .
    $$
    Clearly, $I(F,\mc{F}') \subseteq \II_F \subseteq 2^{F}\setminus\{F\}$ and $\II_F$ is intersection-closed.
    In addition, suppose $F \in \mc{F}'$ and $A \in \II_F$ are such that no sunflower in $\mc{F}$ with $m$ petals has kernel $A$.
    By the definition of $\II_F$, $A=F\cap F_1\cap \cdots\cap F_t$ for some $t \ge 1$ and $F_1,\dots,F_t \in \mc{F}'\setminus \{F\}$.
    The choice of $\mc{F}'$ guarantees that $\varphi(A) \in F,F_1,\dots,F_t$, contradicting that $A=F\cap F_1\cap \cdots \cap F_t$.
    This shows that \ref{item:weakfurediintersectionclosed} (and hence \ref{item:weakfuredimain}) in \cref{lemma: weak furedi} hold and completes the proof.
\end{proof}

It turns out that \cref{lemma: weak furedi} is sufficient to reprove many applications of F\"uredi's delta-system lemma, especially for problems on set systems with restricted intersections.
For example, as mentioned in the Introduction, two of the successes of F\"uredi's delta-system lemma are the proofs for the Erdős--S\'os problem (on forbidding one intersection size) by Frankl and F\"uredi~\cite{FF85}, and Chv\'atal's problem on simplices by the same authors~\cite{FF87}.
It can easily be checked that in both proofs, F\"uredi's delta-system lemma can be replaced by \cref{lemma: weak furedi} to achieve a single-exponential dependency for the size of the ground set in terms of $k$ (in their original proof, the dependency is double-exponential).
We omit the details as the required modifications are quite minor.
Instead, as mentioned above, in \cref{section: excluding one residue,section: Frankl--Wilson} we will provide two new applications of the delta-system lemma on certain set systems with restricted intersections and prove them via \cref{lemma: weak furedi}.

In recent years, significant progress has been made on the two well-known problems mentioned above, particularly in certain regimes, through the use of the 
junta method~\cite{KL21,EKL24}  and the spread approximation method~\cite{KZ24} (for the Erdős–Sós problem).
These results allow one to take the size $n$ of the ground set to be much smaller, linear in $k$, while getting precise bounds (which were proved by Frankl and Füredi when $n$ is double-exponential $k$).
Despite being successful in the range where $k$ grows ``rapidly'' with $n$, these methods have some further constraints on the third parameter (the forbidden intersection size $\ell$ in the Erd\H{o}s--S\'os problem and the dimension $d$ of the simplices in the other) and require (long) problem-specific proofs, whereas the arguments using the delta-system lemma are usually much shorter.
Additionally, these methods seem to require the extremal constructions to be juntas, families that are defined based on only few vertices in the ground set.
For certain problems, for example, the Erd\H{o}s--S\'os problem for $\ell \ge (k-1)/2$, the extremal example is conjectured to be some modified design~\cite{FF85} (which is far away from juntas), and there have been no improvements since the paper of Frankl and F\"uredi~\cite{FF85} (which used the delta-system method). We note that in this range for the Erdős--Sós problem, only the order of magnitude of the answer is known, determining the precise or asymptotical answer is open even when $n$ is much larger than $k$.

Continuing with our discussion of our delta-system lemma, in \cref{section: excluding one residue,section: Frankl--Wilson} we will demonstrate that the very useful technique of considering \emph{ranks} of intersection-closed set systems also works with our variant proved above. If $\II\subseteq 2^{X}\setminus \{X\}$ is a set system, by its rank we mean the size of the smallest subset of $X$ not contained in any member of $\II$ (i.e., not \emph{covered} by $\II$). The main observation making ranks useful is that if the rank of each $\II_F$ is at most $r$, then $|\FF'|\leq \binom{n}{r}$, since different sets $F$ must give different uncovered $r$-sets. This allows us to bound complicated extremal problems by simply considering the structures of intersection-closed set systems satisfying certain size-restrictions for their members.

\subsection{Double-exponential dependency in F\"uredi's delta-system lemma}\label{section: furedi double exponential}
In this section, we prove \cref{thm:doubleexponential}, showing that F\"uredi's delta-system lemma (\cref{lemma: furedi}) requires a $2^{-\Omega\left(2^{k}/\sqrt{k}\right)}$ dependency on $k$. Indeed, we will essentially show that Property \ref{item:Furedisame} forces this double-exponential behaviour.

The idea is to define $\mc{F}=\bigcup_{\mc{I}} \mc{F}_{\mc{I}}$ where for ``many'' intersection-closed systems $\mc{I}\subseteq 2^{[k]}$, we construct a $k$-partite set system $\mc{F}_{\mc{I}}$ (each  $\FF_\II$ on a different set of vertices) such that $\pi(I(F,\mc{F}_{\mc{I}}))=\mc{I}$ for every $F \in \mc{F}_{\mc{I}}$ (recalling \ref{item:Furediintersection-closed} and \ref{item:Furedisame} in \cref{lemma: furedi}).
Furthermore, we will make sure that this property holds robustly, i.e., if we delete some members from $\FF_\II$, then as long as a not-too-small fraction of sets remain in $\mc{F}_{\mc{I}}$, $\pi(I(F,\mc{F}_{\mc{I}}))=\mc{I}$ still holds for some $F$ that remains.
Then, if we keep a not-too-small fraction of sets in $\mc{F}=\bigcup_{\mc{I}} \mc{F}_{\mc{I}}$, some $\mc{F}_{\mc{I}}$ and $\mc{F}_{\mc{I}'}$ will witness distinct projections $\mc{I},\mc{I}'\subseteq 2^{[k]}$, contradicting \ref{item:Furedisame} in \cref{lemma: furedi}.
To make this idea work, we will consider one simple class of intersection-closed systems: those containing every set of size less than $\ell$ and some of the sets of size $\ell$; there are $2^{\binom{k}{\ell}}$ such choices.
Clearly, the best choices of $\ell$ are $\floor{k/2}$ and $\ceil{k/2}$, but for notational convenience and readability, we use $\ell$ instead of the concrete numbers.
In addition, we will index $\mc{F}_\mc{I}$ with $\mc{I}\subseteq \binom{[k]}{\ell}$ (instead of the intersection-closed set system $\II'$ obtained by adding all sets of size less than $\ell$ to this $\mc{I}$).

Formally, we will need the following statement. The only difference from the discussion above is that for simplicity, we will only focus on sunflowers of kernel size $\ell$ and not prove that all sets of size at most $\ell-1$ appear as kernels of large sunflowers -- this property is not used in the proof of \cref{thm:doubleexponential}, but with some extra work one can also deduce that from our construction.

\begin{lemma}\label{lemma: construction for each intersection-closed family}
    For all integers $k > \ell \ge 2$ and every $q \ge 2$, there exists $s=s(k,\ell,q) > 0$ such that the following holds.
    Suppose $\mc{I} \subseteq \binom{[k]}{\ell}$.    
    Then, there exists a $k$-partite set system $\mc{F}$ with parts $X^{(1)},\dots,X^{(k)}$ such that 
    \begin{enumerate}[(a)]
        \item $|\mc{F}|=s$; \label{item: property cardinality}
        \item $\II\subseteq\pi(I(F,\mc{F}))\subseteq \mc{I}\cup\left(\binom{[k]}{0}\cup\binom{[k]}{1}\cup\dots\cup\binom{[k]}{\ell-1}\right)$ for all $F\in\FF$ ($\pi$ is the projection with respect to the partition $X^{(1)},\dots,X^{(k)}$); \label{item: property projection}
        \item for every $I \in \mc{I}$, the family $\mc{F}$ can be partitioned into sunflowers with at least $q$ petals each such that every sunflower has a kernel of the form $\{x^{(i)}: i \in I\}$ with $x^{(i)} \in X^{(i)}$ for all $i \in I$ (i.e., the kernels under projection $\pi$ are all precisely $I$).
        \label{item: property decomposition sunflowers}
    \end{enumerate}
\end{lemma}    

Before we prove \cref{lemma: construction for each intersection-closed family}, we explain how it implies \cref{thm:doubleexponential}.

\begin{proof}[Proof of~\cref{thm:doubleexponential}]
    Write $\ell = \ceil{k/2} \in [2,k-1]$, and fix $q$ to be a sufficiently large integer in terms of $k$, specified later.
    For each $\mc{I} \subseteq \binom{[k]}{\ell}$ (and the chosen value of $q$), let $\mc{F}_{\mc{I}}$ be the family provided by \cref{lemma: construction for each intersection-closed family} with $k$ parts $X^{(1)}_\mc{I},\dots,X^{(k)}_\mc{I}$.
    We will assume that $\bigcup_{i=1}^k X^{(i)}_\mc{I}$ is disjoint from $\bigcup_{i=1}^k X^{(i)}_{\mc{I}'}$ for every distinct $\mc{I},\mc{I'}$.
    This means no two sets from distinct $\mc{F}_{\mc{I}}$ will contribute to intersections we are interested in, i.e., those of size precisely $\ell$.
    In addition, note that \cref{lemma: construction for each intersection-closed family} guarantees the existence of some integer $s$ that $|\mc{F}_{\mc{I}}| = s$ for all $\mc{I}$.

    Now, take $\FF$ to be the union of all set systems $\FF_\mc{I}$ with $\mc{I}\subseteq \binom{[k]}{\ell}$, so $|\FF|=s\cdot2^{\binom{k}{\ell}}$.
    Assume that $\FF'\subseteq \FF$ satisfies the conclusion of \cref{lemma: furedi} with $|\FF'|\geq \alpha_{k,m}|\FF|$ (our goal is to show that we must have $\alpha_{k,m} \leq k!\cdot2^{-\binom{k}{\ell}}=k!\cdot 2^{-\binom{k}{\ceil{k/2}}}$).
    In particular, $\FF'$ is $k$-partite, say with partite classes $Y^{(1)},\dots,Y^{(k)}$, and there exists $\mc{I}^\star\subseteq 2^{[k]}\setminus \{[k]\}$ such that $\pi(I(F,\FF'))=\mc{I}^\star$ for all $F\in \FF'$.
    Recall that $I(F,\FF') = \{F\cap F': F'\in \FF'\setminus \{F\}\}$ and $\pi$ denotes projection with respect to the partition $Y^{(1)},\dots,Y^{(k)}$. 

    For each $\mc{I}$, we partition $\mc{F}_\mc{I}\cap \mc{F}'$ based on the correspondence between the classes $X^{(1)}_\mc{I},\dots,X^{(k)}_\mc{I}$ and $Y^{(1)},\dots,Y^{(k)}$ as follows.
    For each permutation $\sigma: [k] \to [k]$, define
    $$ 
        \mc{F}_{\mc{I}}'(\sigma)
        = \left\{ \{x^{(1)},\dots,x^{(k)}\} \in \mc{F}_{\mc{I}} \cap \mc{F}': x^{(i)}\in X^{(i)}_{\mc{I}} \cap Y^{(\sigma(i))} \textnormal{ for all } i \in [k] \right\}. 
    $$       
    Note that $\bigcup_\mc{I}\bigcup_\sigma \FF_\mc{I}'(\sigma)=\FF'$.
    By the pigeonhole principle, there exists some $\sigma$ with $|\bigcup_\mc{I}\FF_\mc{I}'(\sigma)|\geq|\FF'|/k!$. 
    By relabelling $Y^{(1)},\dots,Y^{(k)}$, we may assume that $\sigma=\mathrm{id}$ is the identity permutation. Hence, writing $\FF_\II'':=\mc{F}_\mc{I}'(\mathrm{id})$, we have $|\bigcup_\II \FF_\II''|\geq |\FF'|/k!$ and
    $$ 
        \mc{F}_{\mc{I}}'' =\mc{F}_\mc{I}'(\mathrm{id})       = \left\{ \{x^{(1)},\dots,x^{(k)}\} \in \mc{F}_{\mc{I}} \cap \mc{F}': x^{(i)}\in X^{(i)}_{\mc{I}} \cap Y^{(i)} \textnormal{ for all } i \in [k] \right\}.
    $$           

    Let $\mc{I}^\star_0 = \mc{I}^\star\cap \binom{[k]}{\ell}$. 
    We claim that whenever $\II$ is such that $\mc{F}_\mc{I}''\neq \emptyset$, then $\mc{I}^\star_0 \subseteq \mc{I}$, i.e., the intersection pattern of $\mc{F}'$ is a subfamily of that of $\mc{F}_\mc{I}$.
    % We claim that $\mc{I}^\star_0 \subseteq \mc{I}$ for all $\mc{I}\subseteq \binom{[k]}{k/2}$ with $\mc{F}_{\mc{I}}'(\sigma)\neq \emptyset$.
    To see this, fix some $F \in \mc{F}_\mc{I}''$ and take any $A \in \mc{I}_0^\star$.
    Recalling that $\pi(I(F,\mc{F}'))=\mc{I}^\star$, there exists some $F' \in \mc{F}'$ such that $\pi(F\cap F')=A$.   
    This $F'$ must be in $\mc{F}_\mc{I}$ as $A$ is non-empty.
    By the definition of $\mc{F}_\mc{I}''$, the projection of $F\cap F'$ with respect to $X^{(1)}_\mc{I},\dots,X^{(k)}_\mc{I}$ is precisely $A$, so \ref{item: property projection} of \cref{lemma: construction for each intersection-closed family} implies $A \in \mc{I}$.
    The claim then follows by considering all $A \in \mc{I}^\star_0$.    
    % we know $\mc{I}^\star=\pi({I}(F,\mc{F}'))$, so for any $A \in \mc{I}^\star_0$, there exists some $F' \in \mc{F}'$ such that $\pi(F\cap F')=A$.
    % By the assumption $F\in\mc{F}_\II'(\sigma)$, for any vertex $x \in F \cap F'$ and any $i \in [k]$, we have $x \in Y_i$ if and only if  $x \in X_{\sigma(i)}^\mc{I}$.
    % By considering the projection of this $A$ onto $X^{(1)}_\mc{I},\dots,X^{(k)}_\mc{I}$, we see that $A \in \mc{I}$.
    % This, together with the fact that $\pi(F \cap F')=A$, shows that $\{j: F\cap F'\cap X_j^\II\not =\emptyset\}=\sigma(A)$ and hence $\sigma(A) \in \mc{I}$ by our choice of the families $\FF_\II$ (see \cref{lemma: construction for each intersection-closed family}).
    % The claim follows by considering all $A \in \mc{I}^\star_0$.    

    Moreover, for every $\mc{I}\subseteq \binom{[k]}{\ell}$ with $\mc{F}_\mc{I}''\neq \emptyset$, as long as $\mc{I} \neq \mc{I}_0^\star$, there exists $I \in \mc{I}\setminus \mc{I}^\star_0$.
    According to \ref{item: property decomposition sunflowers} of \cref{lemma: construction for each intersection-closed family}, $\mc{F}_\mc{I}$ can be partitioned into sunflowers with $q$ petals such that the kernel of each of these sunflowers intersects precisely the parts $X^{(i)}_\II$, $i \in I$.
    However, since $\mc{F}_\mc{I}''\subseteq \mc{F}'$ and $I\not\in\II_0^\star$, each of these sunflowers have at most one petal in $\mc{F}_\mc{I}''$ (once again using the definition of $\FF_\II''$).
    In particular, $|\mc{F}_\mc{I}''| \le |\mc{F}_\mc{I}|/q=s/q$.
    In total, the cardinality of $\bigcup_\mc{I}\FF_\mc{I}''$ is 
    \[
        \Big|\bigcup_\mc{I}\FF_\mc{I}''\Big| = \sum_{\mc{I}}|\mc{F}_\mc{I}''| = |\mc{F}_{\mc{I}^\star_0}''| + \sum_{\mc{I}\neq \mc{I}^\star_0} |\mc{F}_\mc{I}''|
        < s + 2^{\binom{k}{\ell}} \cdot \frac{s}{q}.
    \]
    Recall that $\big|\bigcup_\mc{I}\FF_\mc{I}''\big| \ge|\mc{F}'|/k! \ge\alpha_{k,m} \cdot |\mc{F}|/k!=\alpha_{k,m} \cdot s\cdot 2^{\binom{k}{\ell}}/k!$.
    We acquire $\alpha_{k,m} < k!2^{-\binom{k}{\ell}} + k!/q$.
    Since this must hold for all $q$, by taking $q\to\infty$, we get $\alpha_{k,m} \le k!\cdot 2^{-\binom{k}{\ell}}=k!\cdot 2^{-\binom{k}{\ceil{k/2}}}$, as desired.
\end{proof}

Thus, it is left to construct the families $\FF_\II$ as in \cref{lemma: construction for each intersection-closed family}. Our construction for $\mc{F}_{\mc{I}}$ is based on linear spaces satisfying certain properties, and will require the following technical lemma.
%To make our notations clear, we will use superscripts to indicate which part the vectors (or vertices) belong to, and use subscripts to indicate vectors in the same part.

\begin{lemma}\label{lem:lemmafordoubleexponential}
    For all integers $k>\ell\geq 2$, there exist positive integers $p_0, d$ such that whenever $p\geq p_0$ is a prime and $\mathcal{I}\subseteq \binom{[k]}{\ell}$, then there exist $d$-dimensional subspaces $V^{(1)},\dots,V^{(k)}$ of $\mathbb{F}_p^{\ell d}$ such that the following properties hold.
    \begin{enumerate}[(i)]
       \item Whenever $I\in\binom{[k]}{\ell} \setminus \mc{I}$, then $\dim\sum_{i\in I} 
       V^{(i)}=\ell d$, i.e., $\mathbb{F}_p^{\ell d}\label{item: basis}=\bigoplus_{i\in I} V^{(i)}$.
        \item Whenever $I\in \mc{I}$, then $\dim\sum_{i\in I} V^{(i)}=\ell d-1$.\label{item: not basis}
        \item Whenever $I\in\binom{[k]}{\ell+1}$, then $\dim\sum_{i\in I} V^{(i)}=\ell d$, i.e., $\mathbb{F}_p^{\ell d}=\sum_{i\in I} V^{(i)}$. \label{item:span}       
    \end{enumerate}
\end{lemma}

Once again, before proving \cref{lem:lemmafordoubleexponential}, we use it construct the families in \cref{lemma: construction for each intersection-closed family}.

\begin{proof}[Proof of \cref{lemma: construction for each intersection-closed family}]
    Fix some prime $p$ that is sufficiently large in terms of $k$ and $q$, and let $d \in \mb{N}$ and the $d$-dimensional subspaces $V^{(1)},\dots,V^{(k)}\subset \mb{F}_p^{\ell d}$ be as in \cref{lem:lemmafordoubleexponential} for $\mc{I}$. 
    For each $i\in[k]$, pick a basis $v_1^{(i)},v_2^{(i)},\dots,v_d^{(i)}$ of $V^{(i)}$. 
    Let $X^{(1)},\dots,X^{(k)}$ be $k$ disjoint copies of $(\mb{F}_{p}^{\ell d})^d$, thus, every member $x^{(i)} \in X^{(i)}$ is a $d$-tuple of vectors in $\mb{F}_{p}^{\ell d}$. We write $x^{(i)}_j\in \mathbb{F}_p^{\ell d}$ for the $j$th vector in this $d$-tuple $x^{(i)}$.
    %.   Observe that every member $x^{(i)} \in X^{(i)}$ can be viewed as a matrix in $\mb{F}_{p}^{kd/2\times d}$.  For convenience, we write $x^{(i)}_j$ for the $j$th column of $x^{(i)}$.
    
    Now, we construct a $k$-partite set system $\mc{F}$ on the $k$ parts $X^{(1)},\dots, X^{(k)}$ as follows. Let $x^{(1)}\in X^{(1)},\dots,x^{(k)}\in X^{(k)}$ be $k$ elements of our ground set (so each $x^{(i)}$ is a $d$-tuple of vectors).
    Informally, these $k$ elements form a member in our family $\FF$ precisely when the $kd$ vectors $x^{(i)}_j$ appearing in the $d$-tuples satisfy the same linear dependency relations as the vectors $v^{(i)}_j$. 
    That is, formally, the vertices $x^{(1)},\dots,x^{(k)}$ form a set in $\FF$ if and only if there is an invertible $\ell d$-by-$\ell d$ matrix $A$ such that $x^{(i)}_j= Av^{(i)}_j$ for all $i\in[k], j\in[d]$.
    We will call this $A$ the {\em defining matrix} of the set $\{x^{(1)},\dots,x^{(k)}\}$.
    Since the vectors $(v_j^{(i)})_{i\in[k],j\in[d]}$ span $\mathbb{F}_p^{\ell d}$, there is a one-to-one correspondence between the sets in $\mc{F}$ and the defining matrices $A$.
    Since the number of invertible $\ell d$-by-$\ell d$ matrices in $\mb{F}_p$ is $\prod_{i=1}^{\ell d} (p^{\ell d}-p^{i-1})$, by taking $s:=\prod_{i=1}^{\ell d}(p^{\ell d}-p^{i-1})$, we have $|\mc{F}|=s$.
    Then, \ref{item: property cardinality} holds.
    
    Suppose $I\in\binom{[k]}{\ell} \setminus \mc{I}$ and consider $\{x^{(i)}\}_{i\in I}$ with $x^{(i)}\in X^{(i)}$, we show that $\{x^{(i)}\}_{i\in I}$ cannot appear as the intersection of two sets in $\FF$.
    Indeed, by \ref{item: basis} of \cref{lem:lemmafordoubleexponential}, the vectors $\{v_j^{(i)}\}_{i\in I, j\in [d]}$ span the entire space $\mathbb{F}_p^{\ell d}$ and therefore
    the defining matrix $A$ of any set in $\FF$ containing $\{x^{(i)}\}_{i\in I}$ is determined by the conditions $\big\{x_j^{(i)}=Av_j^{(i)}\big\}_{i\in I, j\in [d]}$. Hence the $\ell$-element set $\{x^{(i)}\}_{i\in I}$ is contained in at most one member of $\FF$. 
    Similarly, using~\ref{item:span}, we can show that no set of size greater than $\ell$ appears as the intersection of distinct sets in $\mc{F}$.
    So, $\pi(I(F,\mc{F})) \subseteq \mc{I}\cup\left(\binom{[k]}{0}\cup\binom{[k]}{1}\cup\dots\cup\binom{[k]}{\ell-1}\right)$ holds for all $F \in \mc{F}$; this is the second half of \ref{item: property projection}.
    The first half of \ref{item: property projection} follows directly from \ref{item: property decomposition sunflowers}.
    
    Hence, it suffices to show \ref{item: property decomposition sunflowers}.
    To this end, we fix $I\in\mc I$ in the rest of the proof.
    Let $\{x^{(i)}\}_{i\in I}$ be arbitrary where each $x^{(i)}\in X^{(i)}$ and $\{x^{(i)}\}_{i\in I}\subseteq F$ for some $F\in\FF$.
    We claim that there are precisely $p^{\ell d}-p^{\ell d-1}$ sets $E\in\FF$ that contain $\{x^{(i)}\}_{i\in I}$.
    Since we already know that no two sets in $\FF$ intersect in more than $\ell$ elements, these sets $E$ form a sunflower of kernel $\{x^{(i)}\}_{i\in I}$ and with $p^{\ell d}-p^{\ell d-1}\ge q$ petals by taking $p$ large enough.
    In particular, this proves \ref{item: property decomposition sunflowers} by listing all possible $\{x^{(i)}\}_{i\in I}$ and the corresponding sunflowers.

    We are left to prove the above claim.
    Recall that we fix $I \in \mc{I}$ and
    $\{x^{(i)}\}_{i\in I}$ where each $x^{(i)}\in X^{(i)}$ and $\{x^{(i)}\}_{i\in I}\subseteq F$ for some $F \in \mc{F}$.
    Then, \ref{item: not basis} of \cref{lem:lemmafordoubleexponential} implies that the vectors $(v^{(i)}_j)_{i\in I,j\in[d]}$ span an $(\ell d-1)$-dimensional subspace of $\mb{F}_p^{\ell d}$.
    Hence, using that $\{x^{(i)}\}_{i\in I}\subseteq F$ for some set $F \in \mc{F}$, we know that the vectors $(x^{(i)}_j)_{i\in I,j\in[d]}$ also span an $(\ell d-1)$-dimensional subspace of $\mb{F}_p^{\ell d}$ (since there is an invertible matrix $A$ such that $x^{(i)}_j=A v^{(i)}_j$ for all $i\in I,j\in[d]$.)
     Fix any $i' \in [k]\setminus I$. 
    By \ref{item:span} of \cref{lem:lemmafordoubleexponential}, there exists $j' \in [d]$ such that %$\rank(V,v^{(i')}_{j'})=kd/2$, where $(V,v^{(i')}_{j'})$ stands for the $kd/2$-by-$(kd/2+1)$ matrix obtained by adding the column vector $v^{(i')}_{j'}$.
    $(v^{(i)}_j)_{i\in I,j\in[d]}$ together with $v^{(i')}_{j'}$ span the entire space $\mathbb{F}_p^{\ell d}$.
    
    Consider any $E \in \mc{F}$ such that $E$ contains $\{x^{(i)}\}_{i \in I}$. Thus, writing $E=\{y^{(i)}\}_{i\in [k]}$ (where $y^{(i)}\in X^{(i)}$ for all $i$), we have $y^{(i)}=x^{(i)}$ for all $i\in I$.
    We know there exists some invertible matrix $A \in \mb{F}_p^{\ell d\times \ell d}$ such that
    %$(X,E^{(i')}_{j'})=A(V,v^{(i')}_{j'})$, where $E_{j'}^{(i')}$ stands for the $j'$th column of the vertex (matrix) in $E \cap V^{(i')}$.
    $x^{(i)}_j=Av^{(i)}_j$ for all $i\in I,j\in [d]$, and $y^{(i')}_{j'}=Av^{(i')}_{j'}$.
    Moreover, by the dimension considerations above, given $(x^{(i)})_{i\in I}$, the matrix $A$ (and hence the set $E$) is uniquely determined by the vector $y^{(i')}_{j'}=Av^{(i')}_{j'}$. By the invertibility assumption on $A$, we must have $y^{(i')}_{j'}\in\mb{F}_p^{\ell d}\setminus \spann\{x^{(i)}_j: i\in I,j\in [d]\}$, and each such $y^{(i')}$ gives a unique invertible matrix $A$ (and corresponding set $E$) satisfying the equalities mentioned above.
    Hence, precisely $p^{\ell d}-p^{\ell d-1}$ sets in $\mc{F}$ contain $\{x^{(i)}\}_{i \in I}$ (one set for each choice of $y^{(i')}_{j'}$).
    This shows the claim, and hence, completes the proof of this lemma.
\end{proof}

Finally, we prove the technical linear algebraic lemma we used.

\begin{proof}[Proof of \cref{lem:lemmafordoubleexponential}]
    We will assume $p_0$ is sufficiently large in terms of $k$ whenever needed.
    Let $d=\binom{k}{\ell}$ and consider the $\ell d$-dimensional vector space $\left(\mathbb{F}_p^{\ell}\right)^{\binom{[k]}{\ell}}$, where each element is a $\binom{k}{\ell}$-tuple of vectors labelled by the sets $I\in\binom{[k]}{\ell}$, with each vector in the tuple belonging to $\mathbb{F}_p^{\ell}$. (Note that this space is isomorphic to $\mathbb{F}_p^{\ell d}$.)
    For each $i\in[k]$, we will define $d$ vectors $v^{(i)}_I\in \left(\mathbb{F}_p^{\ell}\right)^{\binom{[k]}{\ell}}$, one for each $I\in\binom{[k]}{\ell}$, and the subspaces as in the statement of the lemma will be obtained by taking $V^{(i)}$ to be the span of $(v_I^{(i)})_{I\in\binom{[k]}{\ell}}$ (for some appropriately chosen vectors $v_I^{(i)}$). Thus, we now explain how we pick the vectors $v^{(i)}_I$.
    Each $v_I^{(i)}$ is of the form $(0,0,\dots,0,w_I^{(i)},0,\dots,0)$, where the only non-zero vector appears in the position labelled by $I$ and all the remaining $d-1$ coordinates are $0\in \mathbb{F}_p^{\ell}$. 
    %In other words, we have $v_I^{(i)}(J)=0$ if $J\not =I$ and $v_I^{(i)}(I)=w_I^{(i)}$.
    Moreover, the vectors $(w_I^{(i)})_{i\in[k],I\in\binom{[k]}{\ell}}$ are defined as follows. 
    
    If $I\in \binom{[k]}{\ell}\setminus \mc{I}$, then the vectors $(w_I^{(i)})_{i\in[k]}$ are chosen to be $k$ arbitrary vectors in general position in $\mathbb{F}_p^{\ell}$, i.e., any $\ell$ of them are linearly independent.
    Note that this is possible if $p\ge p_0$ is large enough. 
    
    To define the vectors $(w_I^{(i)})_{i\in[k]}$ in the case $I\in \II$, choose an arbitrary element $a\in I$, take vectors $(w_I^{(i)})_{i\in[k]\setminus\{a\}}$ arbitrarily in general position in $\mathbb{F}_p^{\ell}$, and let $w_I^{(a)}$ be an arbitrary vector contained in the linear span of $\{w_I^{(i)}:i\in I\setminus \{a\}\}$ but not in the span of $\{w_I^{(i)}:i\in I_0\}$ whenever $I_0\in \binom{[k]\setminus \{a\}}{\ell-1}$ with $I_0\not =I\setminus\{a\}$. 
    Since any two different $(\ell-1)$-dimensional subspaces have intersection of dimension at most $\ell-2$ and $p$ (the size of the field) is large enough, this is again possible. 
 %   Thus, we have that for $i_1,\dots,i_{\ell}\in[k]$ distinct, the $\ell$ vectors $w_{I}^{(i_1)},\dots,w_{I}^{(i_{\ell})}$ are linearly independent if and only if $\{i_1,\dots,i_{\ell}\}\not =I$.

    By the constructions above, we see that for $i_1,\dots,i_{\ell}\in[k]$ distinct, the $\ell$ vectors $w_{I}^{(i_1)},\dots,w_{I}^{(i_{\ell})}$ are linearly dependent if and only if $I\in\II$ and $I=\{i_1,\dots,i_{\ell}\}$. For each $i\in[k]$, take $V^{(i)}$ to be the span of $(v_I^{(i)})_{I\in\binom{[k]}{\ell}}$. 
    Using the observation above, and the fact that in each $d$-tuple $v^{(i)}_I$ only the entry labelled by $I$ is non-zero, it is easy to check that these subspaces satisfy the required conditions.
\end{proof}

\subsection{Set systems excluding one residue class}\label{section: excluding one residue}
As mentioned before, for set systems with modular restrictions, the Frankl--Wilson theorem states that if $p$ is a prime and $L\subseteq \zpz$ with $s:=|L|\leq k$ and $k\textrm{ mod }p\not \in L$, then whenever $\FF\subseteq \binom{[n]}{k}$ is such that every pair has intersection size in $L$ mod $p$, the cardinality  $|\FF|\leq \binom{n}{s}$.
Although this result is very useful for many applications, it does not tell us anything when $k\textrm{ mod } p \in L$. In this section, we will consider upper bounds in the setting where  the residue $r=k\!\!\mod p, r \in[0,p-1]$ is an element of $L$.

In the simplest case where $L$ is the singleton residue class containing $r$, we have a trivial lower bound $\binom{\floor{(n-r)/p}}{\floor{k/p}}$ for the largest family of $k$-sets over $[n]$ for which all pairwise intersections have size in $L$ (mod $p$).
For example, when $r=0$, one can partition a subset of $[n]$ of size $p\floor{n/p}$ into $\floor{n/p}$ disjoint $p$-tuples and take only the sets consisting of $k/p$ such $p$-tuples.
For this case, the Ray-Chaudhuri--Wilson theorem\footnote{The Ray-Chaudhuri--Wilson theorem states that if $\mc{F} \subseteq \binom{[n]}{k}$ is such that every pair in $\mc{F}$ has intersection size in a set $L_0\subseteq[0,k-1]$ of size $s$, then $|\mc{F}|\leq \binom ns$.} gives an upper bound of the same order of magnitude, and, in fact, Frankl and Tokushige~\cite{FT16eventown} showed that the exact value of the lower bound $\binom{\floor{(n-r)/p}}{\floor{k/p}}$ is tight as long as $n$ is sufficiently large with respect to $k$.
To the best of our knowledge, no nontrivial result is known for general $L\subsetneqq \zpz$ with $k\textrm{ mod } p\in L$.
Somewhat surprisingly, we are able to show that the answer is always at most $n^{k/p+O(p^2)}$ and at least $n^{k/p+\Omega(p)}$.
Note that to obtain the upper bound, it is enough to consider the case when $|L|=p-1$, i.e., when there is only one forbidden residue class. Thus, this problem can be seen as a modular version of the Erd\H{o}s--S\'os problem on forbidding one intersection size.
In a later section, we will also use the result below to prove upper bounds on $r(k,L,m)$ in the modular case $L\subseteq \zpz$ (see \cref{theorem: modular clique upper bound}).

For convenience, we state the lower-bound constructions below when $k$ is a multiple of $p$, but essentially the same proof gives similar constructions in the general setting. 
The upper bound holds for every $k$.
\begin{theorem}\label{theorem: excluding one residue}
    Let $p$ be a prime, $a \in [p-1]$ and $k$ a positive integer.
    Suppose $\mc{F}\subseteq \binom{[n]}{k}$ contains no $F,F' \in \mc{F}$ with $|F\cap F'|\equiv a \modp{p}$.
    Then $|\mc{F}|= O_k\left(n^{\lceil k/p\rceil+(p-1)(p-2)}\right)$.
    Furthermore, if $k$ is a multiple of $p$, the largest such family has cardinality at least $\Omega_k\left(n^{k/p+p-2}\right)$ if $a=p-1$ and at least $\Omega_k\left(n^{k/p+p-3}\right)$ if $k \ge 2p$.
\end{theorem}

We need the following general estimate, which might be of independent interest.
\begin{lemma}\label{lemma: bound via convex sequences}
     Let $k$ be a positive integer and $L \subseteq [0,k-1]$.
     Let $s$ be the largest integer such that there exists a sequence $\ell_1>\ell_2>\dots>\ell_s$ in $L$ which is convex, i.e., $2\ell_i \le \ell_{i-1}+\ell_{i+1}$ for all $i\in[2,s-1]$.
     Then, any $\mc{F}\subseteq \binom{[n]}{k}$, where $|F_1\cap F_2| \in L$ for any two distinct members $F_1,F_2 \in \mc{F}$, has cardinality $O_k(n^s)$.
\end{lemma}
\begin{proof}
    Apply~\cref{lemma: weak furedi} with $m=2$ to obtain $\FF'\subseteq\FF$ with $|\FF'|=\Omega_k(|\FF|)$ such that for every $F \in \mc{F}'$, there exists an intersection-closed set system $\mathcal{I}_F$ such that $I(F,\mc{F}')\subseteq \mathcal{I}_F\subseteq 2^{F}\setminus \{F\}$ and every element of $\mathcal{I}_F$ appears as the intersection of two distinct sets in $\FF$. 
    In particular, each member of $\mathcal{I}_F$ has size in $L$. 
    It suffices to show that 
    \begin{equation}\label{eq: star}
        \textnormal{for every } F \in \mc{F}' \textnormal{, there exists a set } X_F \subseteq F\textnormal{ with } |X_F| \le s\textnormal{ that is not contained in any member of }\II_F.
    \end{equation}
    Note that \eqref{eq: star} implies that $X_F\neq X_{F'}$ for distinct $F,F'\in\FF'$.
    Otherwise, we have $X_F=X_{F'}\subseteq F\cap F'\in I(F,\FF')\subseteq\II_F$, contradicting \eqref{eq: star}.
    Hence $X_F$ uniquely determines each $F\in\FF'$ and we may assume $|X_F|=s$ by putting in more elements from $F$.
    This gives $|\FF'|\leq\binom{n}{s}$ and thus $|\FF|=O_k(n^{s})$.

    Now, we prove \eqref{eq: star}.
    Let  $F \in \mc{F}'$ and let $A \subseteq F$ be a minimal set that is not contained in any member of $\mc{I}_F$.
    Write $A = \{a_1,\dots,a_t\}$ where $t:=|A|$.
    It suffices to show that $t \le s$.
    Observe that the minimality of $A$ implies that for each $i \in [t]$, $A\setminus \{a_i\}$ is contained in some $B_i \in \mc{I}_F$.
   We define a permutation $i_1,i_2,\dots,i_t$ of $[t]$ recursively as follows.
    Let $i_1 \in [t]$ be a choice that minimises $|B_{i_1}|$.
    For $j\in[t-1]$, if we have already defined $i_1,\dots,i_j$, then take $i_{j+1} \in [t]\setminus \{i_1,\dots,i_j\}$ that minimises $|B_{i_1}\cap B_{i_2}\cap\dots \cap B_{i_{j+1}}|$.
    
    For each $j \in [t]$, write $C_j := B_{i_1}\cap B_{i_2}\cap\dots\cap B_{i_j}$ and $\ell_j:=|C_j|$. By definition, $C_{j+1} \subset C_j$. 
    Firstly, since $\mc{I}_F$ is intersection-closed with all members having size in $L$, we know $C_j \in \mc{I}_F$ and $\ell_j \in L$ for all $j \in [k]$.
    Secondly, notice that for every $i \in [t]$, $a_i$ is contained in all of $B_1,B_2,\dots,B_{i-1},B_{i+1},B_{i+2},\dots,B_t$ but not in $B_i$. This means $a_{i_{j+1}} \in C_{j}\setminus C_{j+1}$, so $\ell_{j+1} < \ell_j$.
    Thirdly, we show $2\ell_j \leq \ell_{j-1} + \ell_{j+1}$ for all $j \in [2,t-1]$.
    Note that \[
        \begin{cases}
            \ell_{j-1}-\ell_{j}=|C_{j-1}|-|C_{j-1}\cap B_{i_j}|=|C_{j-1}\setminus B_{i_j}|, \\
            \ell_{j}-\ell_{j+1}=|C_{j}|-|C_{j}\cap B_{i_{j+1}}|=|C_j\setminus B_{i_{j+1}}|\le |C_{j-1}\setminus B_{i_{j+1}}|. 
        \end{cases}   
    \]
    But, by the choice of $i_j$ (minimising $|C_{j-1}\cap B_{i_j}|$), we have $|C_{j-1}\setminus B_{i_j}|\geq |C_{j-1}\setminus B_{i_{j+1}}|$, giving the claimed bound $2\ell_j\leq \ell_{j-1}+\ell_{j+1}$.
    Hence, by the definition of $s$, we must have $t \le s$, as desired.
\end{proof}

\begin{proof}[Proof of \cref{theorem: excluding one residue}]
    We start with the upper bound.
    Let $L= \{\ell \in [0,k-1]: \ell \not\equiv a\modp{p}\}$, so every pairwise intersection from $\FF$ belongs to $L$.
    By \cref{lemma: bound via convex sequences}, it suffices to show that the longest sequence $\ell_1>\ell_2>\dots>\ell_s$ in $L$, where $2\ell_i \le \ell_{i-1}+\ell_{i+1}$ for all $i\in[2,s-1]$, has $s \le\lceil k/p\rceil+(p-1)(p-2)$.
    
    Let $\delta_i:=\ell_i-\ell_{i+1}$ for all $i\in[s-1]$, so we have $\delta_1 \ge \delta_2 \ge \dots\ge \delta_{s-1} \ge 1$.
    We note that each $j \in [p-1]$ can only occur at most $p-2$ times as the value of some $\delta_i$.
    Indeed, suppose $\delta_i=\delta_{i+1}=\dots=\delta_{i+p-2}=j$ holds for some $i$.
    Then $\ell_i,\ell_{i+1},\dots,\ell_{i+p-1}$ give all the $p$ possible residues modulo $p$, so one of them must be $a$ modulo $p$, giving a contradiction.
    In addition, the number of values $i$ for which $\delta_i\geq p$ is at most $\floor{(k-1)/p}=k/p-1$, since the sum of these values is at most $\ell_1-\ell_s\le k-1$.
    In total, $s-1 \le (p-1)(p-2)+\floor{(k-1)/p}$, so $s \le\lceil k/p\rceil+(p-1)(p-2)$.

    We now prove the lower bound (assuming $n$ is sufficiently large whenever necessary).
    Partition $[n]$ into $X_1 \cup X_2$ with $|X_1|=\floor{n/(2p)}p$ and $|X_2|=n-|X_1|$; so $|X_1|,|X_2| = \Theta_k(n)$.
    Further, partition $X_1$ into $s:=\floor{n/(2p)}$ sets $Y_1,\dots,Y_s$ each of size $p$; so $s=\Theta_k(n)$.
    If $a=p-1$, let $\mc{F}_1$ consist of all $(k-p)$-element sets that are a union of $k/p-1$ of the sets $Y_i$.
    Hence, $|\mc{F}_1|=\Theta_k(n^{k/p-1})$ and $|F\cap F'|\equiv 0\modp{p}$ whenever $F,F'\in\FF_1$.
    Since for every $p$-element subset of $X_2$ there are at most $p|X_2|$ other $p$-element subsets that intersect it in exactly $p-1$ points, using a greedy procedure, we can obtain a family $\mc{F}_2\subseteq \binom{X_2}{p}$ such that no pair intersects in $p-1$ elements, and $|\mc{F}_2|=\Theta_k(|X_2|^{p-1})=\Theta_k(n^{p-1})$.
    Then, take $\mc{F}:=\{F_1\cup F_2: F_1 \in \mc{F}_1, F_2 \in \mc{F}_2\}$.
    Clearly, no two sets $F,\tilde{F} \in \mc{F}$ have $|F\cap \tilde{F}| \equiv a\modp{p}$, and $|\mc{F}|=\Theta_k(n^{k/p+p-2})$, as desired.

    If $k \ge 2p$, similarly, let $\mc{F}_1$ consist of all $(k-2p)$-element sets that are a union of $k/p-2$ of the sets $Y_i$.
    Hence, $|\mc{F}_1|=\Theta_k(n^{k/p-2})$ and $|F\cap F'|\equiv 0\modp{p}$ if $F,F'\in\FF_1$.
    Let $X_3\subseteq X_2$ be any subset of size $a+1$.
    Via a similar greedy argument as above, there is a family $\mc{F}_2\subseteq \binom{X_2\setminus X_3}{2p-a-1}$ such that no pair intersects in $p-1$ elements, and $|\mc{F}_2|=\Theta_k(n^{p-1})$.
    Then, take $\mc{F}:=\{F_1\cup F_2\cup X_3: F_1 \in \mc{F}_1, F_2 \in \mc{F}_2\}$.
    Clearly, no two sets $F,\tilde{F} \in \mc{F}$ have $|F\cap \tilde{F}| \equiv a\modp{p}$, and $|\mc{F}|=\Theta_k(n^{k/p+p-3})$, as desired.
\end{proof}

\subsection{Frankl--Wilson theorem revisited}\label{section: Frankl--Wilson}

In this section, we demonstrate how to use the delta-system method and \cref{lemma: weak furedi} to give a different proof of the celebrated Frankl--Wilson theorem up to a constant factor. To do this, we will need the following proposition, which shows the non-existence of certain set systems with restricted intersections. Closely related results appeared in \cite{nagele2019submodular, BBS23} and have applications to submodular minimization and extremal set theory.

\begin{proposition}\label{proposition: rank for FW}
    Let $p$ be a prime and $L\subseteq \zpz$ with $|L|=s$, and let $X$ be a ground set with $|X|$\textnormal{~mod~}$p\not \in L$. Then, there exists no intersection-closed set system $\FF$ on $X$ such that $|F|\! \!\mod p  \in L$ for all $F \in {\cal F}$, and every subset of $X$ of size at most $s$ is contained in some element of $\FF$.
\end{proposition}
\begin{proof}
Assume, for a contradiction, that such a family $\FF$ exists. Then, for all $i\leq s$, by the inclusion--exclusion principle (similarly to \cite{BBS23}),
    \begin{equation}\label{equation: inclusion-exclusion}
    \binom{|X|}{i}=\sum_{\mathcal{I}\subseteq \cal F,\mathcal{I}\not=\emptyset}(-1)^{|\mathcal{I}|+1}\binom{|\bigcap_{F\in \mathcal{I}}F|}{i}.
    \end{equation}
Consider the polynomial $f(x)=\prod_{\ell\in L}(x-\ell)$ over $\zpz$. Since $|X|$ mod $p\not \in L$, we have $f(|X|)\not =0$. Moreover, we can write $f(x)=\sum_{i=0}^{s}\lambda_i \binom{x}{i}$ for some $\lambda_i\in\zpz$. Using~\eqref{equation: inclusion-exclusion} for each $i\leq s$, and doing calculations in $\zpz$, we get
\begin{align*}
    f(|X|)=\sum_{i=0}^s\lambda_i\binom{|X|}{i}&=\sum_{i=0}^{s}\lambda_i\left[\sum_{\mathcal{I}\subseteq \cal F,\mathcal{I}\not=\emptyset}(-1)^{|\mathcal{I}|+1}\binom{\big|\bigcap_{F\in \mathcal{I}}F\big|}{i}\right]\\
    &=\sum_{\mathcal{I}\subseteq \cal F,\mathcal{I}\not=\emptyset}(-1)^{|\mathcal{I}|+1}\left[\sum_{i=0}^{s}\lambda_i\binom{\big|\bigcap_{F\in \mathcal{I}}F\big|}{i}\right]\\
    &=\sum_{\mathcal{I}\subseteq \cal F,\mathcal{I}\not=\emptyset}(-1)^{|\mathcal{I}|+1} f\Big(\big|\bigcap_{F\in \mathcal{I}}F\big|\Big).
\end{align*}
However, $\cal F$ is intersection-closed and so $\bigcap_{F\in \mathcal{I}}F\in\cal F$ for all $\mathcal{I}\subseteq \cal F, \mathcal{I}\not=\emptyset$. Thus, $|\bigcap_{F\in \mathcal{I}}F| \!\! \mod p \in L$ and $f(|\bigcap_{F\in \mathcal{I}}F|)=0$ for all such $\mathcal{I}$, and the equation above gives $f(|X|)=0$, a contradiction.   
\end{proof}

\begin{theorem}[Frankl--Wilson~\cite{FW81}]
    Let $n \geq k$ be two positive integers, $p$ be a prime, and let $L\subseteq \zpz$ with \({k\!\!\mod p}\not \in L\). Then, if $\FF\subseteq \binom{[n]}{k}$ is such that $|F_1\cap F_2| \textrm{ mod } p\in L$ for all distinct $F_1,F_2\in \FF$, then $|\FF|=O_k(n^{|L|})$.
\end{theorem}
\begin{proof}
    Apply~\cref{lemma: weak furedi} with $m=2$ to obtain $\FF'\subseteq\FF$ with $|\FF'|=\Omega_k(|\FF|)$ such that for every $F \in \mc{F}'$, there exists an intersection-closed set system $\mathcal{I}_F$ such that $I(F,\mc{F}')\subseteq \mathcal{I}_F\subseteq 2^{F}\setminus \{F\}$ and every element of $\mathcal{I}_F$ appears as the intersection of two distinct sets in $\FF$. In particular, each member of $\mathcal{I}_F$ has size in $L$~mod~$p$. Clearly, we may assume $|L|\leq k$. 
    By~\cref{proposition: rank for FW} applied for each $\II_F$, we get that for all $F\in\FF'$ there is some $X_F\subseteq F$ of size $|L|$ such that $X_F$ is not a subset of any element of $\II_F$. 
    Note that all the sets $X_F$ must be distinct (since if $F'\not =F$, then $F\cap F'\in\II_F$ and hence $X_F\not \subseteq F\cap F'$), giving $|\FF'|\leq \binom{n}{|L|}$. Thus, $|\FF|=O_k(n^{|L|})$, as claimed.
\end{proof}

\section{Modular setting}\label{sec: modular things}
In this section, we discuss forbidding $L$-cliques in the modular setting.
We will prove \cref{theorem: modular clique} and provide a general upper bound for $r(k,L,m)$ where $L\subseteq \zpz$.
We will also discuss variants of \cref{theorem: modular clique} in terms of the fractional chromatic number of the auxiliary graph $G_\mc{F}$
(whose vertices are the sets in $\mc F$ and two sets $F,F' \in \mc{F}$ are connected if and only if $|F\cap F'|\in L$), along with their application in quantum computing. 

We start with the lower bound in \cref{theorem: modular clique}.
We will deduce it from the following stronger statement, which deals with the case where $k\equiv 0$ (mod $p$) and $L=(\zpz)\setminus\{0\}$.
Surprisingly, we show that if $\mc{F}$ contains no ``large'' $L$-clique, then there is a ``large'' subfamily $\mc{F}'\subseteq \mc{F}$ with the following ``atomic structure'': there exist disjoint sets (atoms) $X_1,\dots,X_n$ such that every $F\in \mc{F}'$ is a union of $k/p$ of them. 
Arguably, such atomic structures are the most natural instances of families of sets where each intersection size is a multiple of $p$. % the easiest way to construct set systems, whose pairwise intersections all have sizes a multiple of $p$, is by taking some disjoint atoms of size $p$ each and take every set that is a union of $k/p$ of the atoms.
This structure is very strong because it implies not only that the intersection size of every two sets in $\mc{F}'$ is a multiple of $p$, but also that the intersection of \emph{any} number of sets in $\mc{F}'$ is a multiple of $p$.
In general, these two properties can behave very differently. We will discuss this more in \cref{remark: atomic structure}.

In order to obtain this atomic structure, we carefully select $p$ elements in the ground set and only keep the sets that contain either none of them or all of them. 
Then, we shrink the ground set by $p$ and repeat this process.
As we shall see, the proof does not require $p$ to be a prime, so we use letter $d$ in place of $p$.
Also, for inductive reasons, we consider a weighted version, which will also be useful for later applications for fractional chromatic numbers.

\begin{theorem}\label{thm:modularstructure}
    Let $k$, $d$, and $m$ be positive integers with $k\geq d$. 
    Assume that $\FF$ is a family of sets of size at most $k$ over some finite ground set $X$ such that the size of each element of $\FF$ is a multiple of $d$, and $\FF$ does not contain $m$ members that have pairwise intersection sizes non-zero modulo $d$ (i.e., $\FF$ has no $L$-clique of size $m$ for $L=(\mathbb{Z}/d\mathbb{Z})\setminus\{0\}$).
    Assume furthermore that there is a given non-negative weighting $w$ on $\FF$, and define a new weighting $w'$ by setting, for each $F\in\FF$, 
    \begin{equation} \label{eq: new weights}
        w'(F)=\left(\frac{1}{d\binom{k}{d-1}m}\right)^{|F|/d}w(F).
    \end{equation}
    Then, for some non-negative integer $n$, the ground set $X$ has disjoint subsets $X_1,\dots,X_n$ of size $d$ each such that if we set
    \[\FF'=\{F\in \FF: F\textnormal{ is a union of some of the sets $X_i$}\},\]
   then $\sum_{F \in \mc{F}'} w(F) \ge \sum_{F \in \mc{F}} w'(F)$.
\end{theorem}
\begin{proof}
    We prove the statement by induction on $|X|$. When $|X|<d$, the statement is trivial as $\FF$ must be $\emptyset$ or $\{\emptyset\}$ and we can take $n=0$ so that $\FF'=\FF$. Now assume $|X|\geq d$ and the statement holds for smaller values of $|X|$. Given a set of elements $Y\subseteq X$, let us write $w'_Y$ for $\sum_{F\in\FF:Y\subseteq F}w'(F)$. If $x\in X$, we will write $w'_x$ for $w'_{\{x\}}$. 
    Let $x\in X$ be such that $w'_x$ is maximal.
    \begin{claim}\label{claim:modularstructure_1}
        There exists a set $X_0\subseteq X$ of size $d$ such that $x\in X_0$ and $w'_{X_0}\geq \frac{1}{\binom{k}{d-1}m}w'_x$.
    \end{claim}
    \begin{proof}
        Consider $\mc{H}=\{F\in \mc{F}: x \in F\}$ and define an auxiliary graph $G$ on vertex set $\mc{H}$ by connecting distinct $F,F'\in\HH$ whenever $|F\cap F'| \equiv 0\modp{d}$. 
        Then $w'$ gives a weighting of the vertices of this graph, and by the assumption on $\mc{F}$, $G$ has no independent set of size $m$. Thus, there is a vertex $F_0 \in \mc{H}$ whose 
        weighted neighbourhood
        has weight $w'(F_0)+\sum_{F\text{ adjacent to }F_0} w'(F) \ge \frac{1}{m}w'_x$. Indeed, otherwise, we would be able to greedily pick $m$ sets (vertices in $G$) and delete their neighborhoods to form an independent set.
        Fix such an $F_0 \in \mc{H}$ and consider the family $\mc{H}_0$ containing $F_0$ and all neighbours of $F_0$ in the graph $G$.
        We know that $\sum_{F\in \mc{H}_0} w'(F) \ge \frac{1}{m} w'_x$ and that every set in $\mc{H}_0$ intersects $F_0$ in at least $d$ elements.
        Hence, by the pigeonhole principle, there exists a subset $X_0 \subseteq F_0$ of size $d$ such that $x\in X_0$ and $w'_{X_0}\ge\sum_{F\in \mc{H}_0: X_0\subseteq F} w'(F) \ge \frac{1}{\binom{k}{d-1}m} w'_x$, as desired.
        % Assume that no such set exists. 
        % Let $\HH=\{F\in \FF:x\in F\}$, and define an auxiliary graph $G$ on vertex set $\HH$ by connecting distinct $F,F'\in\HH$ when $|F\cap F'| \equiv 0$~mod~$d$. 
        % Since no $X_0$ as in the statement of the claim exists, for any $F\in \HH$ and any $Z\in \binom{F}{d}$ with $x\in Z$, we have $w'_Z<\frac{1}{\binom{k}{d-1}m}w'_x$.
        % In addition, since every vertex of $G$ shares $x$ in common, it follows that any two connected vertices intersect in at least $d$ elements.
        % Hence, for all $F\in \HH$, we can upper bound the combined weight of $F$ and its neighbouring vertices $F'$ by the total weight of sets containing some $Z \in \binom{F}{d}$ with $x \in Z$, i.e.,
        % $$ 
        % w'(F)+\sum_{F'\text{ connected to } F \textnormal{ in } G} w'(F')
        %     \leq \sum_{Z\in\binom{F}{d}: x\in Z}w'_Z
        %     <\binom{|F|-1}{d-1}\cdot\frac{1}{\binom{k}{d-1}m}w'_x
        %     \leq \frac{1}{m}w'_x. 
        % $$
        % Since $\sum_{F\in V(G)} w'(F)=w'_x$, it follows that we can (greedily) find $m$ distinct and pairwise non-adjacent vertices of $G$.
        % By the definition of $G$, these $m$ sets have pairwise intersection sizes which are non-zero mod~$d$, contradicting the assumption of the theorem. This finishes the proof of the claim.
    \end{proof}
    Thus, we can pick $X_0$ as in the claim above. Let 
    \begin{align*}
        \FF_0 = \{F\in \FF: F\cap X_0=\emptyset\},\quad
        \FF_1 = \{F\in \FF: X_0\subseteq F\},\quad\text{and}\quad
        \GG   = \FF_0\cup\{F \setminus X_0: F\in \FF_1\}.
    \end{align*}
    Note that $\GG$ is a collection of sets of size at most $k$ over ground set $X\setminus X_0$ and the size of every set in $\GG$ is a multiple of $d$. 
    Moreover, whenever $G_1,\dots,G_m\in \GG$, then there exist $i,j\in[m]$ distinct such that $|G_i\cap G_j|\equiv 0\textnormal{ mod }d$. 
    Indeed, by the definition of $\GG$, for each $i\in [m]$ there is some $F_i\in \FF_0\cup \FF_1$ such that $F_i=G_i$ or $F_i=G_i\cup X_0$. 
    By our assumption on $\FF$, there are $i,j\in [m]$ distinct such that $|F_i\cap F_j|\equiv 0$~mod~$d$, and then $|G_i\cap G_j|\equiv 0$~mod~$d$ as $|X_0|=d$ and $X_0\cap (G_i\cup G_j)=\emptyset$.

    Define a weight function $u$ on $\GG$ by setting $$u(G)=w(G)+w(G\cup X_0),$$ where we write $w(A)=0$ if $A\not\in\FF$. 
    By the induction hypothesis applied for the family $\GG$ and weights $u$ over the ground set $X\setminus X_0$, we know that for some non-negative integer $n$, $X\setminus X_0$ has disjoint subsets $X_1,\dots,X_n$ of size $d$ each such that if we write
    \[
    \GG'=\{G\in \GG: G\textnormal{ is a union of some of the sets $X_i$ ($i\in[n]$)}\},
    \]
    then $\sum_{G \in \mc{G}'} u(G) \geq \sum_{G \in \mc{G}} u'(G)$. 
    Here, similar to \eqref{eq: new weights}, $u'$ is defined by $u'(G)=\left(\frac{1}{d\binom{k}{d-1}m}\right)^{|G|/d} u(G)$.
    Let us write
    \[
    \FF'=\{F\in \FF: F\textnormal{ is a union of some of the sets $X_i$ ($i\in[0,n]$)}\}.
    \]
    It suffices to show that $\sum_{F \in \mc{F}'} w(F) \geq \sum_{F \in \mc{F}} w'(F)$.
    
    First, notice that 
    $$
    \sum_{F \in \mc{F}'} w(F)=\sum_{G \in \mc{G}'} u(G) \ge \sum_{G \in \mc{G}} u'(G),
    $$ 
    using the definition of $u$ and the property guaranteed by the inductive hypothesis, respectively. 
        Also, note that $\FF\setminus \FF_0= \bigcup_{y\in X_0}\{F\in\FF: y\in F\}$ and hence $$
            \sum_{F \in \mc{F}} w'(F)
            \leq \sum_{F \in \mc{F}_0} w'(F)+\sum_{y\in X_0}w'_y
            \leq \sum_{F \in \mc{F}_0} w'(F)+d\cdot w'_x
            \leq \sum_{F \in \mc{F}_0} w'(F)+d\binom{k}{d-1}m\cdot w'_{X_0}.
        $$        
        Using the inequalities above, and the definitions of $u$, $u'$, and $w'$, we have
        \begin{align*}
            \sum_{F \in \mc{F}'} w(F)
            \ge \sum_{G \in \mc{G}} u'(G) 
            &=\sum_{G\in\GG}\left(\frac{1}{d\binom{k}{d-1}m}\right)^{|G|/d}u(G)\\
            &=\sum_{F\in\FF_0}\left(\frac{1}{d\binom{k}{d-1}m}\right)^{|F|/d}w(F)+\sum_{F\in\FF_1}\left(\frac{1}{d\binom{k}{d-1}m}\right)^{(|F|-d)/d}w(F)\\
            &=\sum_{F\in \FF_0}w'(F)+\sum_{F\in \FF_1}w'(F)\cdot d\binom{k}{d-1}m\\
            &=\sum_{F \in \mc{F}_0} w'(F) + w'_{X_0}\cdot d\binom{k}{d-1}m
            \geq \sum_{F \in \mc{F}} w'(F),
        \end{align*}
        finishing the proof.
\end{proof}

We now prove the lower bounds in \cref{theorem: modular clique} with the help of \cref{thm:modularstructure}.
\begin{proof}[Proof of the lower bounds in \cref{theorem: modular clique}]
    Let $\FF$ be a family of $k$-element sets containing no $L$-clique of size $m$ for 
    \(L=(\zpz)\setminus\{a\},\)
    for some $a\in\zpz$.
    
    First assume that $a=0$. Let $k'=\lceil\frac{k}{p}\rceil p$, and let $\GG$ be the set system obtained by adding $k'-k$ new ``dummy'' elements to each set in $\FF$ in such a way that each new element is contained in precisely one set in $\GG$.
    (Note that we have enlarged our ground set by exactly $(k'-k)|\FF|$ elements.) Then, each set in $\GG$ has size $k'$ (a multiple of $p$), and $\GG$ has no $L$-clique of size $m$. 
    Applying \cref{thm:modularstructure} for $\GG$ and the uniform weighting $w(G)=1$ for all $G\in \GG$, we obtain, for some $s\geq 0$, disjoint subsets $X_1,\dots,X_s$ of the ground set such that each $X_i$ has size $p$ and
    \[\GG'=\{G\in \GG: G\textnormal{ is a union of some of the sets $X_i$}\}\]
    satisfies
    $$
    |\GG'|\geq \left(\frac{1}{p\binom{k'}{p-1}m}\right)^{k'/p}|\GG|=\Omega_k\left(m^{-\lceil k/p\rceil}|\FF|\right).
    $$
    Clearly, any two sets in $\GG'$ have intersection size a multiple of $p$. It follows that if we let $\FF'$ be the subfamily of $\FF$ corresponding to $\GG'$, obtained by removing the dummy elements from $\GG'$ (i.e., $\FF'$ contains those $F\in \FF$ for which there is some $G\in \GG'$ with $F\subseteq G$), then $|\FF'|=|\GG'|=\Omega_k(m^{-\lceil k/p\rceil}|\FF|)$, and any two sets in $\FF'$ have intersection size a multiple of $p$. This finishes the proof in the case $a=0$.

    Now consider the case $a\not=0$, so $a\in[p-1]$. (With a slight abuse of notation, we used $a$ to denote the smallest non-negative representative of the residue class $a\in\zpz$.) 
    In this case, form $\HH$ from $\FF$ by enlarging the ground set by $p-a$ new dummy elements, and adding these $p-a$ elements to each set in $\FF$. 
    Note that each set in $\HH$ has size $k''=k+(p-a)$, and each intersection size increased by exactly $p-a$ when adding the new elements, hence, $\HH$ does not contain an $L''$-clique of size $m$ for $L''=(\zpz)\setminus\{0\}$. 
    Using the $a=0$ case proved above, we find $\HH'\subseteq \HH$ of size $\Omega_k(m^{-\lceil k''/p\rceil}|\HH|)=\Omega_k(m^{-\lceil k/p\rceil-1}|\FF|)$ such that any two elements in $\HH'$ have intersection size a multiple of $p$. 
    Letting $\FF'$ be the collection of sets corresponding to $\HH'$ after removing the $p-a$ dummy elements, we see that $|\FF'|=\Omega_k(m^{-\lceil k/p\rceil-1}|\FF|)$ and any two sets in $\FF'$ have intersection size $a$~mod~$p$, finishing the proof of the lower bounds.
\end{proof}
\begin{remark}\label{remark: atomic structure}
    % Recall that Deza~\cite{Deza74} showed that any ``large'' $\{\ell\}$-clique is an $\{\ell\}$-sunflower (this a natural structure to form $\{\ell\}$-cliques).
    % In a similar spirit, \cref{thm:modularstructure} shows that any $0 \texrnormal{ mod }p$-clique contains a ``large'' subfamily that has the atomic structure.
    Let $k$ be a multiple of $p$ and $a=0$.
    Suppose $\mc{F}\subseteq \binom{[n]}{k}$ is a family where the intersection size of every two sets in $\mc{F}$ is a multiple of $p$ (this is a $0 \textnormal{ mod } p$-clique).
    \cref{thm:modularstructure} (with $m=2$) guarantees a subfamily $\mc{F}'\subseteq \mc{F}$ of cardinality $|\mc{F}'| \ge \big(2p\binom{k}{p-1}\big)^{-k/p} |\mc{F}|$ with the atomic structure, i.e., there exist disjoint atoms $X_1,\dots,X_s\subseteq [n]$ such that every set in $\mc{F}'$ is a union of $k/p$ atoms.
    In general, the argument above shows a similar (but slightly more complicated) atomic structure in any $a\textnormal{ mod } p$-clique (for any $a \in \zpz$ and $k \ge p$).
    Since the atomic structure is the most natural way to construct $a \textnormal{ mod } p$-cliques, this can be seen as a modular version of Deza's result~\cite{Deza74} which claims that any ``large'' $\{\ell\}$-clique is an $\{\ell\}$-sunflower (which is, analogously, the most natural structure forming an $\{\ell\}$-clique).

    We also note that, in a difference from the analogous $\{\ell\}$-clique setting, the ratio $\frac{|\mc{F}'|}{|\mc{F}|}$ between the size of the largest atomic subfamily and the size of the $a\textnormal{ mod } p$-clique can be small.
    Indeed, let $k$ be a multiple of $2p$ and consider an Hadamard matrix $A$ of order $4p$. Frankl and Odlyzko~\cite{FO83} constructed the following family $\mc{F}\subseteq \binom{[2k]}{k}$.
    Assume the first row of $A$ is all-ones, so all the other rows have $2p$ ones and $2p$ minus-ones.
    For every row except for the first one, construct two subsets of $[4p]$, corresponding to the ones and the minus-ones, respectively.
    Let $\mc{A}$ be the family of these $8p-2$ sets; one can show that the intersection size of two distinct sets is a multiple of $p$.
    Let $\mc{F}$ be a product system of $2k/(4p)=k/(2p)$ disjoint copies of $\mc{A}$ (obtained by taking $k/(2p)$ such systems $\mc{A}_1,\dots,\mc{A}_{k/(2p)}$ on disjoint ground sets, and taking the elements of $\FF$ to be all sets which are formed by taking a union $F_1\cup\dots\cup F_{k/(2p)}$ with $F_i\in\mc{A}_i$).
    Then, $\mc{F}$ is a $0\textnormal{ mod } p$-clique of size $(8p-2)^{k/2p}$.
    However, any family $\mc{F}'\subseteq \binom{[2k]}{k}$ with the atomic structure has cardinality at most $\binom{2k/p}{k/p}\leq 2^{2k/p}$.
    This means $|\mc{F}'|/|\mc{F}|\le\left(\frac{8}{4p-1}\right)^{k/2p}=p^{-\Omega(k/p)}$ (unless $p=2$), showing that the largest subfamily with the atomic structure must be small.
    For similar results, we refer the reader to the discussion of Theorem 2 in \cite{GST22}.
\end{remark}

The upper bounds in \cref{theorem: modular clique} basically follow from taking $\mc{F}$ to be $
\binom{[n]}{k}$ (with some modifications involving ``dummy elements'').

\begin{proof}[Proof of the upper bounds in \cref{theorem: modular clique}]
   Let $a\in\zpz$ be such that $L=(\zpz)\setminus \{a\}$.
    First assume that $a=k\!\!\mod p$.
    Consider $\mc{F} = \binom{[n]}{k}$ for some $n \ge k$. The Frankl--Wilson theorem implies that the largest $L$-clique in $\FF$ has size at most $\binom{n}{p-1}$, and therefore $\mc{F}$ contains no $L$-clique of size $m = \binom{n}{p-1}+1$. The Ray-Chaudhuri--Wilson theorem implies that the largest subfamily of $\mc{F}$ that avoids intersections of size in $L$ (i.e., with pairwise intersection sizes coming from $a,a+p,a+2p,\dots$, $k-p$) has cardinality $O_k(n^{\floor{k/p}})$. This is at most $O_k(n^{k-(p-1)\lfloor k/p\rfloor})=O_k(m^{-\floor{k/p}} |\mc{F}|)$.
    Thus, in the case $a=k\!\!\mod p$, we have $r(k,L,m)=O_k(m^{-\floor{k/p}})$ whenever $m=\binom{n}{p-1}+1$ for some $n \ge k$, and hence $r(k,L,m)=O_k(m^{-\floor{k/p}})$ for all $m$.

    Now consider the case $a\ne k\!\!\mod p$.
    Let $k' < k$ be the largest integer with $k'\mod{p}= a$; we may assume $k' \ge p$ as otherwise the bound $O_k(m^{-{\floor{k/p}+1}})$ is trivial.
    By the definition of $r(k',L,m)$, for any $\varepsilon > 0$, there exists a family $\mc{F}$ of $k'$-element sets such that it contains no $L$-clique of size $m$ and the largest subfamily of $\mc{F}$ that avoids intersections of size in $L$ has cardinality at most $(r(k',L,m)+\varepsilon)|\mc{F}|$.
    Then, we add $k-k'$ new ``dummy'' elements to each set in $\mc{F}$ to acquire a system $\mc{G}$ of $k$-element sets while maintaining the intersections between any two distinct sets.
    Hence, $\mc{G}$ contains no $L$-clique of size $m$ and the largest subfamily of $\mc{G}$ that avoids intersections of size in $L$ has cardinality at most $(r(k',L,m)+\varepsilon)|\mc{G}|$.
    This shows $r(k,L,m) \le r(k',L,m)$.
    In addition, our discussion above on the case $a=k\!\!\mod p$ implies $r(k',L,m)=O_k(m^{-\floor{k'/p}})$, so $r(k,L,m)=O_k(m^{-\floor{k'/p}})=O_k(m^{-\floor{k/p}+1})$.

    To summarize, we get that when $L=(\zpz)\setminus \{a\}$, then $r(k,L,m)=O_k(m^{-\floor{k/p}+1})$ holds for all $a \in \zpz$, and $r(k,L,m)=O_k(m^{-k/p})$ when $a= k \!\!\mod p= 0$.
\end{proof}

When $L$ corresponds to only one residue, we can show the following upper bound.
The proof is similar to the above except that we use \cref{theorem: excluding one residue} instead of the Ray-Chaudhuri--Wilson theorem.
\begin{theorem}\label{theorem: modular clique upper bound}
    Let $p$ be a prime, $k$ a positive integer and $L=\{a\} \subset \zpz$.
    Then $r(k,L,m)=O_k\big(m^{-(p-1)k/p+p^2} \big)$.
\end{theorem}
\begin{proof}
    % We may assume $k \ge p^2$ as otherwise, the bound is trivial.
    First assume that $a\neq k\textnormal{ mod }p$ and consider $\mc{F} = \binom{[m-1]}{k}$.
    The Frankl--Wilson theorem implies that $\mc{F}$ contains no $L$-clique of size $m$, and \cref{theorem: excluding one residue} implies that the largest subfamily of $\mc{F}$ that avoids intersections of size in $L$ has cardinality $O_k(m^{\lceil k/p\rceil+(p-1)(p-2)})=O_k(m^{-\floor{(p-1)k/p}+(p-1)(p-2)} \cdot |\mc{F}|)$.
    This shows that $r(k,L,m)=O_k(m^{-\floor{(p-1)k/p}+(p-1)(p-2)}) = O_k(m^{-(p-1)k/p+p^2})$ whenever $a \neq k \!\!\mod {p}$.

    Now, assume $a=k\!\!\mod {p}$. Then  $k'=k-1$ has $a\not = k'\!\!\mod p$. 
    Similarly to the proof of the upper bounds in \cref{theorem: modular clique}, we have 
    \(
        r(k,L,m) \le r(k',L,m) 
        = O_k\big (m^{-\floor{(p-1)k'/p}+(p-1)(p-2)} \big)
        % = O_k\left(m^{-k'+\lceil k'/p\rceil+(p-1)(p-2)}\right) 
        = O_k\big(m^{-(p-1)k/p+p^2}\big).
    \)
\end{proof}
We note that almost the same argument shows that $r(k,L,m) = O_k\left(m^{-\frac{(p-1)k}{p|L|} + p^2}\right)$ for every $L\subsetneqq \mb{Z}/p\mb{Z}$. Indeed, the only difference from the above proof is that in the first paragraph we need to consider the case ${k\!\!\mod p\not \in L}$, and the application of the Frankl--Wilson theorem gives an upper bound of $m^{|L|}$ (instead of $m$) on the size of the largest $L$-clique. Then, in the second paragraph, we take $k'\leq k$ to be maximal with $k'\!\!\mod p\not \in L$.

\subsection{Fractional chromatic number of \texorpdfstring{$G_{\mc{F}}$}{G F} and quantum computing} \label{section: fractional chromatic}
Recall that given $L\subseteq \zpz$ and a set system $\mc{F}$ of $k$-element sets, the associated graph $G_{\mc{F}}$ has vertex set $\mc{F}$ and two vertices $F,F' \in \mc{F}$ are connected precisely when $|F\cap F'| \in L$.
In \cref{theorem: modular clique}, we showed if $k$ is even, $L$ is the set of odd integers, and $\mc{F}$ contains no $L$-clique of size $m$, then the independence number of $G_{\mc{F}}$ satisfies $\alpha(G_{\mc{F}}) = \Omega_k(m^{-k/2} |\mc{F}|)$.
As mentioned in \cref{sec:intro_fermionic}, the same bound $m^{k/2}$ also works for the so-called fractional chromatic number of $G_\mc{F}$, which is defined as follows.

\begin{definition}\label{def:fractional_chrome}
    Let $G$ be a graph.
    A \emph{fractional colouring} of size $s\ge1$ of $G$ is a probability distribution over all the independent sets of $G$ such that for each vertex $v$, if we sample an independent set $I$ from this distribution, then $\mathbb{P}(v\in I)\ge1/s$. 
    The {\em fractional chromatic number} $\chi_f(G)$ is defined to be the smallest real number $s \in [1,\infty)$ such that $G$ admits a fractional colouring of size $s$.
\end{definition}

We note that there are many other equivalent definitions of $\chi_f(G)$. Writing $\mc{I}(G)$ for the set of all independent sets of $G$, the definition above can easily be restated in terms of a linear program. Then, using the strong duality of linear programming, we see that $\chi_f$ is the optimal value of the program
\begin{equation}\label{eq: dual fractional}
    \text{maximize} \sum_{v \in V(G)} w_v \quad \text{s.t. } \sum_{v \in I} w_v \le 1\quad\text{for all $I\in \II(G)$, $w_v\geq 0$ for all $v\in V(G)$.}
\end{equation}

It is easy to see that $|V(G)|/\alpha(G) \le \chi_f(G) \le \chi(G)$ for any graph $G$, where $\chi(G)$ denotes the usual (non-fractional) chromatic number.

As mentioned above, our results can be strengthened to give bounds on the fractional chromatic number of $G_\FF$. In particular, we get the following result, which, as we will explain later, has applications in quantum computing.

\begin{theorem}\label{theorem: modular clique fractional chromatic number}
    Let $k$ be a positive even integer and let $L$ be the set of all odd integers.
    Suppose that $\mc{F}$ is a family of $k$-element sets that contains no $L$-clique of size $m$.
    Then, the associated graph $G_\mc{F}$ has fractional chromatic number at most $O_k(m^{k/2})$.
\end{theorem}
\begin{proof}
    The same proof as for \cref{theorem: modular clique} shows that for any non-negative weighting $(w(F))_{F \in \mc{F}}$, there exists a subfamily $\mc{F}'\subseteq \mc{F}$ with 
    $$
        \sum_{F \in \mc{F}'} w(F) \ge \Omega_k(m^{-k/2}) \cdot \sum_{F \in \mc{F}} w(F)
    $$
    such that $\mc{F}'$ is an independent set of $G_{\mc F}$.
    Using the equivalent definition of the fractional chromatic number given by \eqref{eq: dual fractional}, it follows that the fractional chromatic number of $G_{\mc{F}}$ is $O_k(m^{k/2})$. 
\end{proof}

In the same way as above, any bound on $r(k,L,m)$ or $r_{\textnormal{sf}}(k,L,m)$ in this paper also works for the fractional chromatic number of the associated graph $G_{\mc{F}}$. Interestingly, this is no longer true if one replaces ``fractional chromatic number'' by ``chromatic number''.
\begin{theorem}
    Let $k\geq 4$ be an even integer, and let $L$ be the set of all odd integers. 
    Then, for infinitely many positive integers $N$, there exists a family $\mc{F}$ of $N$ sets of size $k$ such that $\mc{F}$ contains no $L$-clique of size $2^{k/2}+1$, yet the chromatic number of the associated graph $G_{\mc{F}}$ is $\Omega_k(\log N)$.
\end{theorem}
\begin{proof}
    Let $n > k/2$ be a large integer. Consider a set $V$ of $3n$ distinct points, labelled as $a_i,x_i,y_i$ for $i\in[n]$.    
    Let $\mc{F}$ be the family of all $k$-element sets $F$ such that for each $i \in [n]$, $F\cap \{a_i,x_i,y_i\}$ is either empty, $\{a_i,x_i\}$ or $\{a_i,y_i\}$.
    It is easy to see that $|\FF|=2^{k/2}\binom{n}{k/2}$.

    We claim that $\mc{F}$ contains no $L$-clique of size $2^{k/2}+1$. Indeed, suppose that $F_1,\dots,F_r \in \mc{F}$ form an $L$-clique, we will show that $r \le 2^{k/2}$.
        For each $i \in [r]$, define $A_i:=\{j \in [n]: \{a_j,x_j\}\subseteq F_i\}$ and $B_i:=\{j\in[n]: \{a_j,y_j\}\subseteq F_i\}$.
        We know that $A_i\cap B_i=\emptyset$ and $|A_i|+|B_i|=k/2$ for all $i \in [r]$.
        Moreover, for all $i,i'\in [r]$ distinct, we have $|F_i\cap F_{i'}|\equiv |A_i\cap B_{i'}|+|A_{i'}\cap B_i|\modp{2}$. 
        In particular, since $|F_i\cap F_{i'}|$ is odd, we have $|A_i\cap B_{i'}|+|A_{i'}\cap B_{i}| \ge 1$ for all distinct $i,i' \in [r]$.
        Tuza~\cite[Theorem 3]{Tuza87} showed that under these conditions (namely, if $A_1,\dots,A_r,B_1,\dots,B_r$ are finite sets such that $|A_{i}\cap B_{i'}|+|A_{i'}\cap B_i|$ is zero if and only if $i=i'$), we have $\sum_{i=1}^r \left(\frac{1}{2}\right)^{|A_i|+|B_i|}\leq 1$. It follows that $r\leq 2^{k/2}$, as claimed.
      %  Then, $r\le 2^{k/2}$ is implied by a result of Tuza~\cite{Tuza87} on set pair inequalities.

    To lower bound the chromatic number of $G_\mc{F}$, suppose that $\FF=\FF_1\cup\dots\cup\FF_\ell$ is a partition of $\FF$ such that every $\mc{F}_i$ avoids intersections of size in $L$.
    (This corresponds one-to-one to a proper colouring of $G_{\mc{F}}$).
    It is enough to show that $\ell = \Omega_k(\log n)$.

    For every $j \in [\ell]$, we say that a pair $a_ix_i$ or $a_iy_i$, is {\em $j$-heavy} if it is contained in at least $n^{(k-3)/2}$ sets in $\FF_j$, and {\em $j$-light} otherwise.
    We decompose $\FF_j$ into $\FF_j^{\text{heavy}} \cup \FF_j^{\text{light}}$, where $\FF_j^{\text{heavy}}$ consists of all $F \in \FF_j$ that is formed by $k/2$ $j$-heavy pairs and $\FF_j^{\text{light}}$ consists of all $F \in \FF_j$ that contains at least one $j$-light pair.
    Furthermore, we define $\FF^{\text{heavy}}:=\bigcup_j \FF_j^{\text{heavy}}$ and $\FF^{\text{light}}:=\bigcup_j \FF_j^{\text{light}}$ 
    Observe that $|\FF_j^{\text{light}}| \le 2n\cdot n^{(k-3)/2} = O_k(n^{-1/2} |\FF|)$ for each $j \in [\ell]$, so 
    $$
    |\FF^{\text{light}}| = O_k(\ell n^{-1/2} |\FF|).
    $$
    
    To lower bound $|\FF^{\text{light}}|=|\FF \setminus \FF^{\text{heavy}}|$, we notice that if $n$ is large enough, then for every $i \in [n]$ and every $j \in [\ell]$, $a_ix_i$ and $a_iy_i$ cannot be both $j$-heavy.
    Indeed, suppose that $a_ix_i$ and $a_iy_i$ are both $j$-heavy, and fix any $F \in \FF_j$ that contains $a_ix_i$.
    For every $F' \in \FF_j$ that contains $a_iy_i$, we know $a_i \in F\cap F'$, so it must be that $|F \cap F'| \ge 2$.
    In other words, there exists $i'\ne i$ such that $a_{i'} \in F\cap F'$.
    Hence $F'$ has at most $(k/2-1)\cdot 2 \cdot 2^{k/2-2}\binom{n-2}{k/2-2}=o_k(n^{(k-3)/2})$ possibilities, contradicting that $a_iy_i$ is $j$-heavy.
    This means, for every $i \in [n]$ and every $j \in [\ell]$, the triple $a_ix_iy_i$ has three possible {\em $j$-types}: (1) $a_ix_i$ and $a_iy_i$ are both $j$-light, (2) $a_ix_i$ is $j$-heavy and $a_iy_i$ is $j$-light, (3) $a_ix_i$ is $j$-light and $a_iy_i$ is $j$-heavy.
    By the pigeonhole principle, there is a set $T$ of at least $n/3^\ell$ triples such that for every $j \in [\ell]$, these triples have the same $j$-type.
    
    Take any $k/2$ of the triples in $T$, say $(a_ix_iy_i)_{i\in I}$ ($|I|=k/2$).
    Note that there are $2^{k/2}$ sets in $\FF$ formed by elements from $(a_ix_iy_i)_{i\in I}$.
    On the other hand, at most two sets among these are in $\FF^{\text{heavy}}$: these two possibilities are $\{a_i,x_i:i\in I\}$ and $\{a_i,y_i:i\in I\}$ as $(a_ix_iy_i)_{i\in I}$ have the same $j$-type for all $j$.
    Applying this for all choices of $I$ (ranging over $k/2$-element subsets of $T$), we get 
    $$
    |\FF^{\text{light}}|=|\FF\setminus \FF^{\text{heavy}}|\geq \binom{n/3^\ell}{k/2} (2^{k/2}-2) \ge \binom{n/3^\ell}{k/2}=\Omega_k\left(3^{-\ell\cdot k/2}|\FF|\right).
    $$
    But we have also seen that $|\mc{F}^{\textnormal{light}}| = O_k(\ell n^{-1/2}|\mc{F}|)$.
    This proves $\ell=\Omega_k(\log n)$ by rearranging.
\end{proof}
% \subsection{Quantum computing}\label{section:quantum_computing}
As we explained in \cref{sec:intro_fermionic}, the upper bounds on the fractional chromatic number in the modular setting find applications in quantum computing for shadow tomography tasks; we now give more details on this. 
 King, Gosset, Kothari, and Babbush, in \cite[Section IV]{king2025triply}, reduced the shadow tomography problem for local fermionic observables to a graph theoretic question, as follows.

\begin{question}\label{qs:fermionic}
    Let $k\ge1$ and let $L$ be the set of all odd integers.
    Suppose that a set system $\mc F\subseteq\binom{[2n]}{2k}$ contains no $L$-clique of size $m$.
    Given $\mc F$ as input, how efficiently can we construct a fractional colouring of the associated graph $G_{\mc F}$ and what is the size of this fractional colouring?
\end{question}

\cref{theorem: modular clique fractional chromatic number} immediately implies the existence of a fractional colouring of size $O_k(m^k)$. Note that we proved this theorem via \cref{thm:modularstructure} whose proof can easily be turned into an efficient algorithm. Using such an algorithm, we will show that \cref{theorem: modular clique fractional chromatic number} can also be made algorithmic. 
Before that, we remark that \cite{king2025triply} designed efficient protocols for fermionic shadow tomography assuming a satisfactory answer to Question \ref{qs:fermionic}.

\begin{theorem}[{\cite[Lemma 6]{king2025triply}}]\label{thm:fermionic_KGKB}
    Let $\varepsilon\in(0,1]$ and $m=4/\varepsilon^2$.
    Suppose Question \ref{qs:fermionic} has a fractional colouring of size $s$ that can be classically sampled from with running time $T$.
    Then there is a shadow tomography algorithm for $k$-local fermionic observables on $n$ qubits that performs only two-copy Clifford measurements, uses classical running time $\mathrm{poly}(n^k,T,1/\varepsilon)$, and
    requires 
    $$
        O\left(\left(\frac1{\varepsilon^2}+s\right)\cdot\frac{k\log(en/k)}{\varepsilon^2}\right)
    $$
    total copies of the unknown quantum states.
\end{theorem}

Note that the classical running time $\mathrm{poly}(n^k)$ in \cref{thm:fermionic_KGKB} (and \cref{lem:qs:fermionic} below) is inevitable, as the number of $k$-local fermionic observables is $\binom{2n}{2k}$ and the algorithm, even outputting the expectation value for each one of them, requires running time $\binom{2n}{2k}=\mathrm{poly}(n^k)$. The major term in \cref{thm:fermionic_KGKB} is the number of copies of the quantum states, which faces limit on the near-term quantum devices.

To address Question \ref{qs:fermionic}, the results in this paper together with a standard reduction using the \emph{multiplicative weight update method}~\cite{arora2012multiplicative} give the following algorithmic result. 

\begin{lemma}\label{lem:qs:fermionic}
    In Question \ref{qs:fermionic}, we have a fractional colouring of size $s=O_k(m^k)$ that can be classically sampled from with running time $\mathrm{poly}(n^k)$.
\end{lemma}

Details on how to obtain such an algorithm are given in \cref{section: appendix quantum}.
Note that combining \cref{lem:qs:fermionic} and \cref{thm:fermionic_KGKB}, we immediately get \cref{thm:fermionic_new}. Indeed, substituting $s=O_k(m^k)=O_k((4/\varepsilon^2)^k)$ and $T=\mathrm{poly}(n^k)$ in \cref{thm:fermionic_KGKB}, we get a triply efficient shadow tomography algorithm with classical running time $\mathrm{poly}(n^k,1/\varepsilon)$ which uses $O_k(\frac{1}{\varepsilon^{2k}}\cdot \frac{\log n}{\varepsilon^2})=O_k((1/\varepsilon)^{2k+2}\log n)$ total copies of the unknown quantum states.

\section{Concluding remarks and open questions} \label{sec:concluding}

In this paper we obtain an efficient delta-system lemma (\cref{lemma: weak furedi}) with a better quantitative dependency on both $k$ and $m$ than F\"uredi's delta-system lemma (\cref{lemma: furedi}).
However, our result only guarantees the sunflowers in the original family $\mc{F}$ rather than in $\mc{F}'$.
As mentioned in \cref{sec: intro furedi}, the stronger conclusion of the sunflowers being contained in $\FF'$ is not necessary in several well-known applications, but can still be important for some problems like the Tur\'an number of expansions of trees \cite{MV16}. Therefore,
it would be interesting to show that every $\mc{F}\subseteq \binom{[n]}{k}$ contains a subfamily $\mc{F}'\subseteq \mc{F}$ with $|\mc{F}'| = \Omega_k(m^{-k}|\mc{F}|)$ that satisfies \ref{item:Furedimain}, i.e., for every distinct $F_1,F_2 \in \mc{F}'$, their intersection $F_1\cap F_2$ is the kernel of a sunflower in $\mc{F}'$ with $m$ petals.
Note that when $m \ge k+1$, this implies \ref{item:Furediintersection-closed}, i.e., that $I(F,\mc{F}')$ is intersection-closed.

In this paper we primarily focused on the dependency on $m$, and in many of our results, the dependency on the uniformity $k$ is $2^{O(k^2)}$.
This comes from our main technique -- colour certificates -- used to prove the lower bounds on $r_\textnormal{sf}(k,L,m)$.
Even in \cref{theorem: modular clique}, we still have a dependency of $k^{O(k)}$ when the prime $p$ is fixed. It would be nice to improve the dependencies on $k$ in all the results in this paper. We note that dependencies on the uniformity are often very interesting, for example, in the sunflower conjecture of Erdős and Rado -- as mentioned before, the recent breakthrough of Alweiss, Lovett, Zhang and the last author~\cite{ALWZ21} on that problem has found many applications. Perhaps some of their techniques can also be used in our setting.

In \cref{theorem: modular clique}, we essentially determined $r(k,L,m) = m^{-k/p+\Theta(1)}$ when $L \subseteq \mb{Z}/p\mb{Z}$ has size $p-1$.
In \cref{theorem: modular clique upper bound}, we gave an upper bound $r(k,L,m)=O_k(m^{-(p-1)k/p+p^2})$ when $L\subseteq \mb{Z}/p\mb{Z}$ has size $1$; and $r(k,L,m) = O_k\left(m^{-\frac{(p-1)k}{p|L|} + p^2}\right)$ in general. It would be very interesting to determine the exact exponents in these results.

\vspace{0.3cm}
\noindent
{\bf Acknowledgements:}
We would like to thank Robbie King, David Gosset, Robin Kothari, and Ryan Babbush for clarifying questions about \cite{king2025triply}. Additionally, we are grateful to Péter Frankl, Zoltán F\"uredi, Dhruv Mubayi, and Jacques Verstra\"ete for their insights about the delta-system method, and Ehud Freidgut and Peter Keevash for answering questions about the junta method. 
We also thank Yuval Wigderson for the stimulating discussions. Part of the work was done while BS and KW were visiting Bernoulli Center for Fundamental Studies at EPFL in Aug 2024.

\bibliographystyle{abbrv} 
\bibliography{ref}

\appendix

\section{Proof of Theorem~\ref{theorem: forbidding sunflowers with one but zero}}
\label{sec: appendix}

In this section, we prove \cref{theorem: forbidding sunflowers with one but zero} by extending the construction we used in the case $L=\{1\}$.
For notational convenience, we will show  
%\begin{equation} \label{eq: main in appendix}
 $   r_{\textnormal{sf}}(k,L,m + 1) \le m^{-(k-1)}$
%\end{equation}
for all $L \subseteq [k-1]$ with $1 \in L$.

Similarly to the proof for $L=\{1\}$, we will define our family $\mc{F}$ based on a graph $G$ whose vertex set is partitioned into $k$ parts.
But for general $L$, the graph between the $k$ parts will look like a tree rather than a chain (as in \cref{lemma: digraph}), see \cref{fig: tree structure} for an example. 
\begin{figure}[ht]
    \centering
    \begin{tikzpicture}[scale = 1.2]
        \pgfmathsetmacro{\dx}{2.5};
        \pgfmathsetmacro{\dy}{1.5};
        \pgfmathsetmacro{\r}{0.5};

        \coordinate (1) at (0, 0);
        \coordinate (2) at (\dx, 0);
        \coordinate (21) at (\dx, \dy);
        \coordinate (3) at (\dx * 2, 0);
        \coordinate (4) at (\dx * 3, 0);
        \coordinate (41) at (\dx * 3, \dy);
        \coordinate (42) at (\dx * 3, \dy * 2);
        \coordinate (5) at (\dx * 4, 0);

        \draw (1) circle (\r);
        \draw (2) circle (\r);
        \draw (21) circle (\r);
        \draw (3) circle (\r);
        \draw (4) circle (\r);
        \draw (41) circle (\r);
        \draw (42) circle (\r);
        \draw (5) circle (\r);

        \draw (1) node {$V_{(1,1)}$};
        \draw (2) node {$V_{(2,1)}$};
        \draw (21) node {$V_{(2,2)}$};
        \draw (3) node {$V_{(3,1)}$};
        \draw (4) node {$V_{(4,1)}$};
        \draw (41) node {$V_{(4,2)}$};
        \draw (42) node {$V_{(4,3)}$};
        \draw (5) node {$V_{(5,1)}$};

        \draw ($(1) + ({0}, {-1.5*\r})$)  node {$h_1=1$};        
        \draw ($(2) + ({0}, {-1.5*\r})$)  node {$h_2=2$};        
        \draw ($(3) + ({0}, {-1.5*\r})$)  node {$h_3=1$};        
        \draw ($(4) + ({0}, {-1.5*\r})$)  node {$h_4=3$};        
        \draw ($(5) + ({0}, {-1.5*\r})$)  node {$h_5=1$};

        \draw[very thick] ($(1) + ({\r}, {0})$) -- ($(2) + ({-\r}, {0})$);
        \draw[very thick] ($(2) + ({\r}, {0})$) -- ($(3) + ({-\r}, {0})$);
        \draw[very thick] ($(3) + ({\r}, {0})$) -- ($(4) + ({-\r}, {0})$);
        \draw[very thick] ($(4) + ({\r}, {0})$) -- ($(5) + ({-\r}, {0})$);
        
        \draw[very thick] ($(2) + ({0}, {\r})$) -- ($(21) + ({0}, {-\r})$);
        \draw[very thick] ($(4) + ({0}, {\r})$) -- ($(41) + ({0}, {-\r})$);
        \draw[very thick] ($(41) + ({0}, {\r})$) -- ($(42) + ({0}, {-\r})$);
    \end{tikzpicture}
    \caption{An illustration of $G$ when $k=8$ and $L=\{1,2,3,6\}$ (hence $t=5$). The eight parts of $G$ are depicted by circles and there are edges between two parts if and only if there is a solid line between the two circles.}
    \label{fig: tree structure} 
\end{figure}
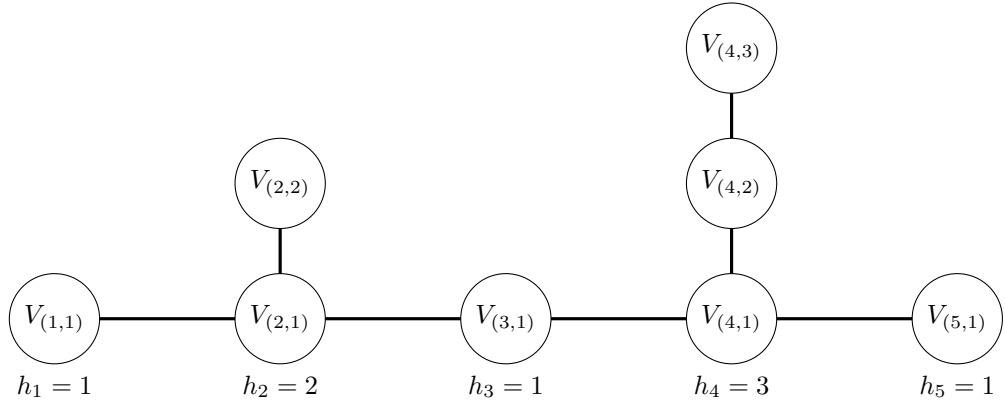

To define $G$, it will be convenient to consider a decomposition of $L$, given by the lemma below.
In this decomposition of $L$, one should think of $t=|[0,k]\setminus L|$ as the number of parts $V_{(1,1)}, \dots, V_{(t,1)}$ forming a horizontal chain as in~\cref{fig: tree structure}, and $h_i$ as the number of parts $V_{(i,1)},\dots,V_{(i,h_i)}$ forming a vertical chain at the $i$th layer. Moreover, these parameters are always chosen in such a way that for each $\ell \in [k-1]$, we have $\ell\not \in L$ if and only if $\ell=\sum_{i=i_0}^t h_i$ for some $i_0 \in [t-1]$.

\begin{lemma}\label{lemma:decomposeL}
    Let $k \ge 2$ and $L\subseteq [k-1]$ with $1\in L$. 
    Then there exist positive integers $t\geq 2$ and $h_i$ for each $i\in[t]$ such that $h_t=1$, $\sum_{i=1}^t h_i=k$, and
    \begin{align*}
        [k]\setminus L&=\Big\{\ell: \ell=\sum_{i=i_0}^t h_i\textnormal{ for some }i_0 \in [t-1]\Big\},\\
        L&=\{1\}\cup\bigcup_{i_0=1}^{t-1}  \Big\{\ell:\sum_{i=i_0+1}^t h_i< \ell<\sum_{i=i_0}^t h_i\Big\}.
    \end{align*}
\end{lemma}

\begin{proof}
    Let $t:=k+1-|L|$ and write $[0,k]\setminus L=\{a_1,\dots,a_{t}\}$ with $k=a_1>a_2>\dots>a_{t}=0$. Then let $h_t:=1$, $h_{t-1}:=a_{t-1}-1$, and $h_i:=a_{i}-a_{i+1}$ for all $i\in[t-2]$. It is easy to see that these satisfy the required conditions.
\end{proof}

Let $t$ and $h_1,\dots,h_t$ be as in \cref{lemma:decomposeL}. 
Let $P:=\{(a,b):a\in[t],b\in[h_a]\}$.
As we shall see later, the $k$ parts of $G$ will correspond to the $k$ points of $P$ and we will index them by $V_{(a,b)}$, where $(a,b) \in P$.
The $k$ vertex sets of $G$ basically form a tree $\mc{T}$ with vertex set $P$ where two vertices $(a,b)$ and $(a',b')$ are connected if $a=a',b=b'\pm 1$ or $a=a'\pm 1,b=b'$.
In addition, every member of $\mc{F}$ will be a tree in $G$ that contains precisely one vertex from each $V_{(a,b)}$.

To illustrate the construction, let us first re-consider the case $L=\{1\}$, whose proof is in \cref{sec: forbidding sunflowers}.
In this case, $t=k$ and $h_1=\dots=h_k=1$.
Hence $P=\{(1,1),(2,1),\dots,(k,1)\}$ and $\mc{T}$ is a chain from $(1,1)$ to $(k,1)$.
Let us recall a sketch proof with $V_1,\dots, V_k$ replaced by $V_{(1,1)},\dots,V_{(k,1)}$.
Recall that $\mc{F}$ contains the $k$-tuples of vertices of all possible chains (in $G$) from $V_{(1,1)}$ to $V_{(k,1)}$ and we make use of the following properties.
\begin{enumerate}[(1)]
    \item For every two chains in $\mc{F}$, the parts of $G$ that the chains intersect on form an interval, i.e., are of form $V_{(i,1)},V_{(i+1,1)},\dots,V_{(j,1)}$ for some $1 \le i < j \le k$. \label{item: recall interval}
    \item Every vertex in $V_{(a,1)}$ has $m$ neighbours in $V_{(a+1,1)}$.
    \label{item: recall c to the right}
    \item Every vertex in $V_{(a,1)}$ has $n$ neighbours in $V_{(a-1,1)}$ for all $a \in [k-1]$ while every vertex in $V_{(k,1)}$ has $m$ neighbours in $V_{(k-1,1)}$. \label{item: recall n to the left}
\end{enumerate}
These facts imply that the
kernel of any sunflower with $m+1$ petals corresponds to parts $V_{(a,1)},\dots,V_{(k,1)}$, for some $a \in [k-1]$ (an interval on the right of length at least $2$).
In particular, there is no sunflower with $m+1$ petals whose kernel has size 1.
In addition, suppose $\mc{F}'\subseteq \mc{F}$ avoids intersections of size 1.
We classified the edges of $G$ into heavy edges and light edges, based on how many sets (chains) in $\mc{F}'$ contain the pair of vertices of this edge.
On the one hand, there are ``few'' sets containing at least one light edge.
On the other hand, every vertex $u \in V_{(a,1)}$ has at most one heavy edge $uv$ to $V_{(a+1,1)}$.
Indeed, otherwise, in order to avoid intersections of size 1, we have to fix one of the $n$ neighbours of $u$ in $V_{(a-1,1)}$, but then, there are not enough sets to make $uv$ or $uv'$ heavy.
Hence, there are at most $|V_{(1,1)}|$ sets which contain only heavy edges.
In total, we see that $\mc{F}'$ is much smaller than $\mc{F}$.

Now, we sketch how to generalize the above construction.
For simplicity, for the purposes of this informal discussion, let us assume that $h_1,\dots,h_t \in \{1,2\}$.
Recall that we want to take $\mc{F}$ to be trees in $G$ that contain precisely one vertex from each $V_{(a,b)}$.
An easy way to think about the sets in $\mc{F}$ is by fixing any vertex in any set $V_{(a,b)}$ and build our set in $\FF$ by using rules like \ref{item: recall c to the right} and \ref{item: recall n to the left} to gradually determine vertices in other parts of $G$.
We will keep \ref{item: recall c to the right} and \ref{item: recall n to the left}, and as a substitute for \ref{item: recall interval}, we want that for any two sets in $\mc{F}$, the parts they intersect on is connected in $\mc{T}$ (hence give a subtree of $\mc{T}$).
Moreover, we need the following rule relating $V_{(a,1)}$ and $V_{(a,2)}$.
\begin{enumerate}[(1),resume]
    \item When building an element of $\FF$, given a vertex $u \in V_{(a,1)}$ and $v \in V_{(a+1,1)}$, there are $m$ options to choose the vertex in $V_{(a,2)}$ (whenever $h_a = 2$); furthermore, the $m$ options for $u,v$ are disjoint from those for $u,v'$ for any $v' \not = v$. \label{item: new rule up}
\end{enumerate}
The reason to have the above less natural rule is that we want to make sure that every vertex in $V_{(a,1)}$ is incident to at most one heavy edge to $V_{(a+1,1)}$ (if $\mc{F}'\subseteq \mc{F}$ avoids intersections of size in $L$).
Indeed, if two sets $F,F'$ in $\mc{F}'$ meet on $V_{(a,1)}$ but not on $V_{(a+1,1)}$, \ref{item: new rule up} guarantees they also do not meet on $V_{(a,2)}$.
In order to satisfy $|F\cap F'| \neq 1$, they must meet on $V_{(a-1,1)}$.
But then, the same argument as in the case $L=\{1\}$ will show that every vertex in $V_{(a,1)}$ is incident to at most one heavy edge to $V_{(a+1,1)}$ for all $a \in [t-1]$.
With some extra work, we can also show that every vertex in $V_{(a,1)}$ is incident to at most one heavy edge to $V_{(a,2)}$ with some suitable notation of heavy edges.
Hence, the number of sets containing only heavy edges is at most $V_{(1,1)}$, like in the case $L=\{1\}$.
Moreover, for sunflowers with $m+1$ petals, we know the parts corresponding to the kernel must form a subtree in $\mc{T}$.
Using \ref{item: recall c to the right} and \ref{item: new rule up} (and an additional property we guarantee, namely that every vertex in $V_{(a,2)}$ has $m$ neighbours in $V_{(a,1)}$ in $G$), we eventually see that the parts corresponding to the kernel must be $\{V_{(a,b)}: a \ge a_0\}$ for some $a_0 \in [t]$.
So either its size is $1$ (i.e., $a_0=t$), which is not possible due to \ref{item: recall n to the left}, or its size is not in $L$.
This demonstrates that any sunflower with $m+1$ petals has kernel size not in $L$, as desired.

To formally prove \cref{theorem: forbidding sunflowers with one but zero}, we need a few more rules than above, to deal with the cases when $h_a \ge 3$ (although these rules are similar in spirit).
See \cref{fig: G} for the properties of $G$ and $\mc{F}$ in the same example as in \cref{fig: tree structure}.
In particular, the new rules \ref{prop: backup}, \ref{prop: branchup}, \ref{prop: up}, and \ref{prop: branchup2} below are depicted in pink.
\begin{figure}[ht]
    \centering
    \begin{tikzpicture}[scale = 1.2]
        \pgfmathsetmacro{\dx}{2.55};
        \pgfmathsetmacro{\dy}{1.57};
        \pgfmathsetmacro{\r}{0.5};
        \pgfmathsetmacro{\h}{0.23};
        \pgfmathsetmacro{\a}{10};
        \pgfmathsetmacro{\b}{35};
        \pgfmathsetmacro{\c}{20};

        \coordinate (1) at (0, 0);
        \coordinate (2) at (\dx, 0);
        \coordinate (21) at (\dx, \dy);
        \coordinate (3) at (\dx * 2, 0);
        \coordinate (4) at (\dx * 3, 0);
        \coordinate (41) at (\dx * 3, \dy);
        \coordinate (42) at (\dx * 3, \dy * 2);
        \coordinate (5) at (\dx * 4, 0);

        \coordinate (RU) at ({\r * cos(\a)}, {\r * sin(\a)});
        \coordinate (LU) at ({-\r * cos(\a)}, {\r * sin(\a)});
        \coordinate (RD) at ({\r * cos(\a)}, {-\r * sin(\a)});
        \coordinate (LD) at ({-\r * cos(\a)}, {-\r * sin(\a)});
        \coordinate (UL) at ({-\r * sin(\a)}, {\r * cos(\a)});
        \coordinate (UR) at ({\r * sin(\a)}, {\r * cos(\a)});
        \coordinate (DL) at ({-\r * sin(\a)}, {-\r * cos(\a)});
        \coordinate (DR) at ({\r * sin(\a)}, {-\r * cos(\a)});
        \coordinate (RRUU) at ({\r * cos(\b)}, {\r * sin(\b)});
        \coordinate (LLUU) at ({-\r * cos(\b)}, {\r * sin(\b)});
        \coordinate (RRDD) at ({\r * cos(\b)}, {-\r * sin(\b)});
        \coordinate (LLDD) at ({-\r * cos(\b)}, {-\r * sin(\b)});
        \coordinate (UULL) at ({-\r * sin(\c)}, {\r * cos(\c)});
        \coordinate (UURR) at ({\r * sin(\c)}, {\r * cos(\c)});
        \coordinate (DDLL) at ({-\r * sin(\c)}, {-\r * cos(\c)});
        \coordinate (DDRR) at ({\r * sin(\c)}, {-\r * cos(\c)});

        \draw (1) circle (\r);
        \draw (2) circle (\r);
        \draw (21) circle (\r);
        \draw (3) circle (\r);
        \draw (4) circle (\r);
        \draw (41) circle (\r);
        \draw (42) circle (\r);
        \draw (5) circle (\r);

        \draw (1) node {$V_{(1,1)}$};
        \draw (2) node {$V_{(2,1)}$};
        \draw (21) node {$V_{(2,2)}$};
        \draw (3) node {$V_{(3,1)}$};
        \draw (4) node {$V_{(4,1)}$};
        \draw (41) node {$V_{(4,2)}$};
        \draw (42) node {$V_{(4,3)}$};
        \draw (5) node {$V_{(5,1)}$};

        \draw ($(1) + ({0}, {-1.5*\r})$)  node {$h_1=1$};        
        \draw ($(2) + ({0}, {-1.5*\r})$)  node {$h_2=2$};        
        \draw ($(3) + ({0}, {-1.5*\r})$)  node {$h_3=1$};        
        \draw ($(4) + ({0}, {-1.5*\r})$)  node {$h_4=3$};        
        \draw ($(5) + ({0}, {-1.5*\r})$)  node {$h_5=1$};

        \draw[->,thick] ($(1) + (RU)$) -- ($(2) + (LU)$);
        \draw (\dx/2, \h) node {$m$};
        \draw[->,thick] ($(2) + (LD)$) -- ($(1) + (RD)$);
        \draw (\dx/2, -\h) node {$n$};
        
        \draw[->,thick] ($(2) + (RU)$) -- ($(3) + (LU)$);
        \draw (\dx*3/2, \h) node {$m$};
        \draw[->,thick] ($(3) + (LD)$) -- ($(2) + (RD)$);
        \draw (\dx*3/2, -\h) node {$n$};
        
        \draw[->,thick] ($(3) + (RU)$) -- ($(4) + (LU)$);
        \draw (\dx*5/2, \h) node {$m$};
        \draw[->,thick] ($(4) + (LD)$) -- ($(3) + (RD)$);
        \draw (\dx*5/2, -\h) node {$n$};
        
        \draw[->,thick] ($(4) + (RU)$) -- ($(5) + (LU)$);
        \draw (\dx*7/2, \h) node {$m$};
        \draw[->,thick] ($(5) + (LD)$) -- ($(4) + (RD)$);
        \draw (\dx*7/2, -\h) node {$m$};
        
        \draw[->, thick] ($(21) + (DL)$) -- ($(2) + (UL)$);
        \draw (\dx-\h*1.2, \dy*0.5) node {$m$};
        \draw[RubineRed,thick,dashed] plot [smooth, tension=1] coordinates { ($(2)+(RRUU)$) (\dx*1.5, \h*2.3)  ($(3)+(LLUU)$) };
        \draw[RubineRed,thick,dashed, ->] plot [smooth, tension=1] coordinates {(\dx*1.5, \h*2.3) (\dx*1.4,\dy*0.85) ($(21)+(\r,0)$) };
        \draw[RubineRed] (\dx*1.4+0.3,\dy*0.85-0.1) node {$m$};
        
        \draw[->, thick] ($(41) + (DR)$) -- ($(4) + (UR)$);
        \draw (3*\dx+\h*1.2, \dy*0.5) node {$m$};
        \draw[RubineRed,thick,dashed] plot [smooth, tension=1] coordinates { ($(4)+(RRUU)$) (\dx*3.5, \h*2.3)  ($(5)+(LLUU)$) };
        \draw[RubineRed,thick,dashed, ->] plot [smooth, tension=1] coordinates {(\dx*3.5, \h*2.3) (\dx*3.4,\dy*0.85) ($(41)+(\r,0)$) };
        \draw[RubineRed] (\dx*3.4+0.3,\dy*0.85-0.1) node {$m$};

        \draw[->, thick] ($(42) + (0,-\r)$) -- ($(41) + (0,\r)$);
        \draw (3*\dx+\h*1.2, \dy*1.5) node {$m$};
        \draw[RubineRed,thick,dashed] plot [smooth, tension=1] coordinates { ($(4)+(UULL)$) (\dx*3-0.25, \dy*0.5)  ($(41)+(DDLL)$) };
        \draw[RubineRed,thick,dashed, ->] plot [smooth, tension=1] coordinates {(\dx*3-0.25, \dy*0.5) (\dx*2.7,\dy*1.1) ($(42)+(LLDD)$) };
        \draw[RubineRed] (\dx*2.62,\dy*1.1) node {$m$};
    \end{tikzpicture}
    \caption{Continuation of the illustration when $k=8$ and $L=\{1,2,3,6\}$. The solid black arrows indicate the number of neighbours in the incoming part of any vertex in the outgoing part. The dashed pink arrows indicate the number of options ($W_{u,v}$) in the incoming part given any two vertices in the two outgoing parts.}
    \label{fig: G} 
\end{figure}
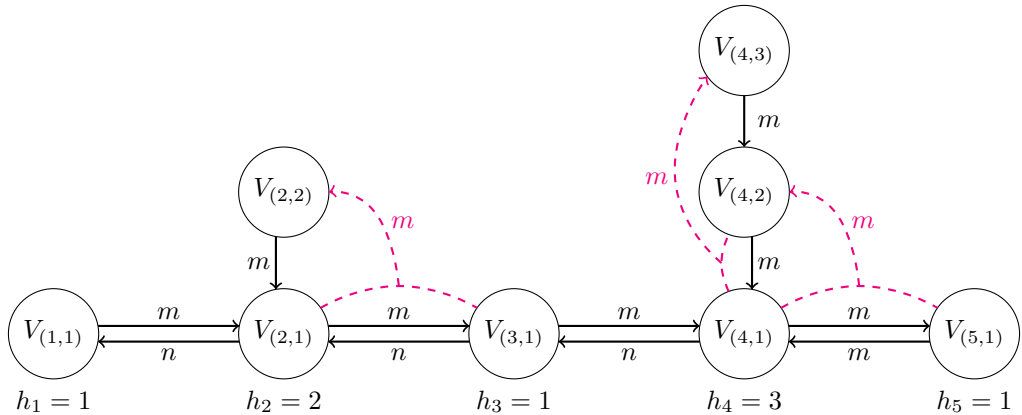
\begin{lemma} \label{lemma: properties of the sets}
    Let $n,m \ge 1$ and $k \ge 2$ be integers.
    Let $L\subseteq [k-1]$ with $1 \in L$, and let $t,h_1,\dots,t_t$ be as in \cref{lemma:decomposeL}.
    Set $P:=\{(a,b):a\in[t],b\in[h_a])$. 
    Then, there exist $k$ disjoints sets $V_{p}$ for $p \in P$ and a $k$-partite set system $\mc{F}$ with parts $(V_{p})_{p\in P}$ such that the following holds.

    Write $\mc{T}$ for the auxiliary tree on the vertex set $P$ where two vertices $(a,b)$ and $(a',b')$ are connected if $a=a',b=b'\pm 1$ or $b=b',a=a'\pm 1$.
    Define $G$ to be the graph with vertex set $\bigcup_{p \in P} V_{p}$ such that two vertices $u\in V_p, v \in V_q$ are connected precisely when $pq$ is an edge in $\mc{T}$ and some set in $\mc{F}$ contains both $u$ and $v$. Then, the following properties are satisfied.
    \begin{enumerate}\setcounter{enumi}{-1}
        \item $|V_{(1,1)}|=n^{t-2}m^{k-t+1}$ and $|\mc{F}|=n^{t-2}m^{2k-t}$. \label{prop: counts}
        \item Whenever $F,F' \in \mc{F}$, then $\{p\in P: F\cap F'\cap V_p\neq\emptyset\}$ induces a connected (or empty) subgraph of $\mc{T}$.\label{prop: connected}
        \item For all $a\in[t-1]$, each vertex in $V_{(a,1)}$ has exactly $m$ neighbours in $V_{(a+1,1)}$.\label{prop: outdegree horizontal}
        \item For all $a\in[2,t-1]$, each vertex in $V_{(a,1)}$ has exactly $n$ neighbours in $V_{(a-1,1)}$. Moreover, each vertex in $V_{(t,1)}$ has exactly $m$ neighbours in $V_{(t-1,1)}$. \label{prop: indegree horizontal}
        \item Whenever $(a,b)\in P$ with $b\geq 2$, then each vertex in $V_{(a,b)}$ has exactly $m$ neighbours in $V_{(a,b-1)}$%, and each vertex in $V_{(i,j-1)}$ has exactly $m^2$ neighbours in $V_{i,j}$
        .\label{prop: indegree vertical}
        \item If $a \in [t-1]$ with $h_a \ge 2$, $u \in V_{(a,1)}$, and $v \in V_{(a+1,1)}$, there exists a set $W_{u,v}\subseteq V_{(a,2)}$ of size $m$ such that every $F\in \mc{F}$ that contains $u,v$ must contain one of the vertices in $W_{u,v}$.\label{prop: backup}
        \item Let $a \in [t-1]$ with $h_a \ge 2$. If $F,F' \in \mc{F}$ satisfies $F\cap V_{(a,1)}=F' \cap V_{(a,1)}$ and $F \cap V_{(a+1,1)} \neq F'\cap V_{(a+1,1)}$, then $F\cap V_{(a,2)} \neq F' \cap V_{(a,2)}$. \label{prop: branchup}
        \item If $a \in [t-1]$, $b \in [3,h_a]$, $u \in V_{(a,b-1)}$, and $v \in V_{(a,b-2)}$, there exists a set $W_{u,v}\subseteq V_{(a,b)}$ of size $m$ such that every $F\in \mc{F}$ that contains $u,v$ must contain one of the vertices in $W_{u,v}$. \label{prop: up}
        \item Let $a \in [t-1]$ and $b \in [3,h_a]$. If $F,F' \in \mc{F}$ satisfies $F \cap V_{(a,b-1)}=F'\cap V_{(a,b-1)}$ and $F\cap V_{(a,b-2)} \neq F' \cap V_{(a,b-2)}$, then $F \cap V_{(a,b)} \neq F' \cap V_{(a,b)}$.\label{prop: branchup2}
    \end{enumerate}
\end{lemma}
We note that Properties~\ref{prop: connected}, \ref{prop: outdegree horizontal}, and \ref{prop: indegree horizontal} correspond to \ref{item: recall interval}, \ref{item: recall c to the right}, and \ref{item: recall n to the left}, respectively; Property~\ref{prop: indegree vertical} is the analogue of Property~\ref{item: recall c to the right} for $V_{(a,b)}$ with $b \ge 2$; Properties~\ref{prop: backup} and \ref{prop: branchup} correspond to \ref{item: new rule up}; and Properties~\ref{prop: up} and \ref{prop: branchup2} give the analogue of Property~\ref{item: new rule up} for the case $h_a \ge 3$.
We first prove \cref{theorem: forbidding sunflowers with one but zero} assuming \cref{lemma: properties of the sets} holds.
\begin{proof}[Proof of \cref{theorem: forbidding sunflowers with one but zero}]
    As discussed at the beginning of the appendix, we show $ r_{\textnormal{sf}}(k,L,m + 1) \le m^{-(k-1)}$.
    Let $t,h_1,\dots,h_t$ be as in \cref{lemma:decomposeL} and recall that 
        \begin{align*}
        [k]\setminus L=\Big\{\ell: \ell=\sum_{i=i_0}^t h_i\textnormal{ for some }i_0 \in [t-1]\Big\}\quad\text{and}\quad
        L=\{1\}\cup\bigcup_{i_0=1}^{t-1}  \Big\{\ell:\sum_{i=i_0+1}^t h_i< \ell<\sum_{i=i_0}^t h_i\Big\}.
    \end{align*}
    % Write $P:=\{(a,b):a\in[t],b\in[h_a]\}$ and $\mc{T}$ to be the tree on the vertex set $P$ where two vertices $(a,b)$ and $(a',b')$ are connected if $a=a',b=b'\pm 1$ or $a=a'\pm 1,b=b'$.
    Let $n$ be a sufficiently large integer in terms of $m$ and $k$, and let $(V_p)_{p\in P}$, $\mc{T}$, $\mc{F}$, and $G$ be as in \cref{lemma: properties of the sets}.
    Here, $P=\{(a,b):a\in[t],b\in[h_a]\}$.

    We claim that $\mc{F}$ has no $L$-sunflower with $m+1$ petals. 
    Suppose for contradiction that there exists a sunflower with $m+1$ petals whose kernel has size in $L$.
    Write $A$ for the projection of the kernel to $P$ (i.e., $A$ is the set of $p\in P$ such that the kernel intersects $V_p$). 
    We know $A$ is non-empty as $0 \notin L$, and $A$ is connected in $\mathcal{T}$ by Property~\ref{prop: connected}.
    % By Property~\ref{prop: connected}, $A$ is connected in the auxiliary graph $\mathcal{T}$.
    In particular, the set $\{a: (a,1) \in A\}$ must form an interval.
    Moreover, by Property~\ref{prop: outdegree horizontal}, $\{a:(a,1)\in A\}$ must be empty or of the form $[a_0,t]$ for some $a_0\in[t]$, and by Property~\ref{prop: indegree vertical}, for each $a$, the set $\{b:(a,b)\in A\}$ is empty or of the form $[1,b_a]$ for some $b_a\in[h_a]$. 
    Furthermore, by Property~\ref{prop: backup}, if $h_a\geq 2$ and $(a,1),(a+1,1)\in A$, then we must have $(a,2)\in A$, and by Property~\ref{prop: up}, if $b\in[3,h_a]$ and $(a,b-1),(a,b-2)\in A$, then $(a,b)\in A$. 
    It follows that $A$ is of the form $\{(a,b)\in P:a\geq a_0\}$ for some $a_0 \in [t]$, and hence $|A|=\sum_{i=a_0}^t h_i$. 
    This means $|A|\not \in L$, unless $a_0=t$ and $A=\{(t,1)\}$. 
    But in this case Property~\ref{prop: indegree horizontal} guarantees this sunflower can have at most $m$ petals, which is a contradiction.
    This proves the claim.
    
    Now, assume the subfamily $\mc{F}'\subseteq \mc{F}$ avoids intersections of size in $L$. 
    Our goal is to show 
    $$
    |\mc{F}'| \le \left(m^{-(k-1)}+o(1)\right)|\mc{F}|,
    $$
    where $o(1)$ stands for some function tending to zero as $n\to\infty$ (with $k,m$ fixed).
    Similarly to the proof when $L=\{1\}$, we want to classify edges in $G$ into heavy and light edges, and show that ``very few'' sets in $\mc{F}'$ contain at least one light edge, and every vertex $u$ in $G$ has at most one heavy edge ``to the right'' as well as at most one ``upwards''.
    
    We first classify edges ``to the right''.
    For each $a\in [t-1]$ and edge $uv$ of $G$ with $u \in V_{(a,1)}$ and $v \in V_{(a+1,1)}$, the edge $uv$ is \emph{heavy} if at least $n^{a-3/2}$ sets in $\mc{F}'$ contains $uv$, and \emph{light} otherwise. 
    We have the following claim.
    \begin{claim} \label{claim: horizontal heavy}
        Suppose $n$ is sufficiently large.
        For every $a \in [t-1]$ and every vertex $u \in V_{(a,1)}$, there is at most one $v \in V_{(a+1,1)}$ such that $uv$ is a heavy edge.
    \end{claim}
    \begin{proof}
        Assume that $u\in V_{(a,1)}$ and $v,v'\in V_{(a+1,1)}$ are such that $uv,uv'$ are both heavy. 
        Fix any set $F \in \mc{F}'$ containing $uv$ (so that $F\cap V_{(a,1)}=\{u\},F \cap V_{(a+1,1)} = \{v\}$). 
        Note that whenever $F'\in \mc{F}$ contains $uv'$, we must have $F\cap V_{(a,1)}=F'\cap V_{(a,1)}$ and $F\cap V_{(a+1,1)}\not =F'\cap V_{(a+1,1)}$. 
        Since $1\in L$, the sets $F,F'$ cannot intersect in exactly $1$ vertex. 
        By Property~\ref{prop: branchup}, we have $F\cap V_{(a,2)}\neq F'\cap V_{(a,2)}$.
        Using Property~\ref{prop: connected}, we see that we must have $a\ge 2$ and $F \cap V_{(a-1,1)}=F' \cap V_{(a-1,1)}$. 
        But the number of sets containing $u,v'$ and a specified vertex from $V_{(a-1,1)}$ is at most $n^{a-2}m^k$.
        Indeed, there are $n^{a-2}$ choices for the vertices in $V_{(a-2,1)},V_{(a-3,1)},\dots,V_{(1,1)}$ by repeated application of Property~\ref{prop: indegree horizontal}.
        After this, we look at the vertices in the sets $V_{(a+1,1)},V_{(a+2,1)},\dots,V_{(t,1)}$, and then those in sets of form $V_{(i,2)}$, then those in sets of form $V_{(i,3)}$, etc. Using Properties \ref{prop: outdegree horizontal}, \ref{prop: backup}, and \ref{prop: up}, each such a step gives at most $m$ choices, so in total they contribute at most a factor of $m^k$. So we get the claimed upper bound $n^{a-2}m^k$ of sets in $\FF'$ containing $uv'$, contradicting the assumption that $uv'$ is a heavy edge when $n$ is sufficiently large.
    \end{proof}

    We are now left to classify edges ``upwards'', i.e., edges between $V_{(a,b)}$ and $V_{(a,b+1)}$, for some $a \in [t-1],b\in [h_a-1]$.
    Note that there are $k-t$ such pairs $(a,b)$, so $k-t$ different ``types'' of edges left to classify.
    We will classify these $k-t$ types of edges recursively according to the ordering on these pairs $(a,b)$ where pairs $(a,b)$ with larger $a$-coordinates are listed first, and if two pairs have the same $a$-coordinate, then the one with smaller $b$-coordinate is listed earlier.
    Let this ordering be $(a_1,b_1),\dots,(a_{k-t},b_{k-t})$. (Thus, for example, in the setting depicted in \cref{fig: tree structure}, this ordering is given by $(4,1)$, $(4,2)$, $(2,1)$.)
    
    Assume that $s \in [k-t]$, and for all $r \in [s-1]$ we have already classified edges between $V_{(a_r,b_r)}$ and $V_{(a_r,b_r+1)}$. 
    Then, we say that a set $F \in \mc{F}'$ is {\em $(s-1)$-completely-heavy} if between all the parts where the edges of $G$ have already classified, $F$ gives a heavy edge in $G$. (In other words, this says that between $V_{(a,1)}$ and $V_{(a+1,1)}$ for all $a\in [t-1]$, as well as between $V_{(a_r,b_r)}$ and $V_{(a_r,b_r+1)}$ for all $r \in [s-1]$, the set $F$ must contain a heavy edge.)
    Now, given any edge $uv$ of $G$ with $u\in V_{(a_s,b_s)}$, $v\in V_{(a_s,b_s+1)}$, we say that $uv$ is \emph{heavy} if there are at least $n^{a_s-3/2}$ sets in $\mc{F}'$ that are $(s-1)$-completely-heavy and contain both $u$ and $v$.
    Similarly to Claim \ref{claim: horizontal heavy}, we have the following.
    \begin{claim} \label{claim: vertical heavy}
        Suppose $n$ is sufficiently large.
        For every $s \in [k-t]$, the heavy edges between $V_{(a_s,b_s)}$ and $V_{(a_s,b_s+1)}$ form a matching.
    \end{claim}
    \begin{proof}
        We will prove the claim by an induction on $s$.
        Let us assume that we already showed that the heavy edges between $V_{(a_r,b_r)}$ and $V_{(a_r,b_r+1)}$ form a matching for all $r \in [s-1]$.
        For convenience, write $a=a_s$ and $b=b_s$; we know that $a \in [t-1]$ and $b \in [h_a-1]$.
        First, observe that for every $v \in V_{(a,b+1)}$, there can be at most one $u \in V_{(a,b)}$ such that $uv$ is a heavy edge.
        Indeed, suppose $uv,u'v$ are distinct heavy edges with $u,u' \in V_{(a,b)}$.
        Consider any $F \in \mc{F}'$ containing $u,v$ and any $F' \in \mc{F}'$ containing $u',v$.
        Properties \ref{prop: connected} and \ref{prop: branchup2} imply $|F\cap F'|=1$, contradicting the fact that $\mc{F}$ avoids intersections of size in $L$.

        Second, let $u \in V_{(a,b)}$ and suppose for contradiction that $uv,uv'$ are distinct heavy edges with $v,v' \in V_{(a,b+1)}$.
        Consider any $(s-1)$-completely-heavy set $F \in \mc{F}'$ containing $u,v$ and any $(s-1)$-completely-heavy set $F' \in \mc{F}'$ containing $u,v'$.
        The definition of being $(s-1)$-completely-heavy implies that both $F$ and $F'$ contain only heavy edges between the parts $V_{(a',1)}$ and $V_{(a'+1,1)}$ for all $a' \in [t-1]$, and between the parts $V_{(a_r,b_r)}$ and $V_{(a_r,b_r+1)}$ for all $r \in [s-1]$.
        Using Claim~\ref{claim: horizontal heavy} and the induction hypothesis (as well as the definition of our ordering of the pairs $(a,b)$), we see that $F$ and $F'$ agree on every part $V_{(a',b')}$ with $a' > a$ or $a'=a,b'\le b$, i.e., $I:=\{p \in P: F\cap F'\cap V_p\neq \emptyset\}$ satisfies $I\supseteq \{(a',b') \in P: a' > a \text{ or } a'=a, b'\le b\}$.
        Notice that the number of such $(a',b')$ is equal to $\sum_{i=a+1}^t h_i + b \in L$.
        Thus, $F$ and $F'$ must intersect in some other part as well.
        Using Property \ref{prop: connected}, we must have $a \ge 2$ and $F\cap V_{(a-1,1)}=F\cap V'_{(a-1,1)}$. 
        Using a similar argument as in Claim \ref{claim: horizontal heavy}, we see that the number of $(s-1)$-completely heavy sets containing $u,v'$ and a specified vertex from $V_{(a-1,1)}$ is at most $n^{a-2}m^k$.
        This contradicts that $uv'$ is an heavy edge when $n$ is sufficiently large.
        Therefore, there is at most one heavy edges $uv$ with $v \in V_{(a,b+1)}$, and the claim follows.
    \end{proof}

    We now complete the proof by counting sets in $\mc{F}'$.
    By Properties~\ref{prop: outdegree horizontal}, \ref{prop: indegree horizontal}, and then \ref{prop: counts}, we know that for all $a\in[t-1]$, $|V_{(a,1)}|=|V_{(1,1)}|(m/n)^{a-1}=n^{t-a-1}m^{k-t+a}$.
    Notice that the number of sets in $\mc{F}'$ containing some light edge between $V_{(a,1)}$ and $V_{(a+1,1)}$ (for some $a \in [t-1]$) is at most $$
        \sum_{a=1}^{t-1} e_G(V_{(a,1)},V_{(a+1,1)})\cdot n^{a-3/2}
        =\sum_{a=1}^{t-1} m\cdot |V_{(a,1)}| \cdot n^{a-3/2}
        =\sum_{a=1}^{t-1} n^{t-a-1} m^{k-t+a+1} \cdot n^{a-3/2}=o\left(m^{-(k-1)}|\mc{F}|\right),
    $$
    where we used $|\FF|=n^{t-2}m^{2k-t}$ (by Property~\ref{prop: counts}) in the last step.
    
    For all $a \in [t-1]$ and $b\in[h_a-1]$, by Properties \ref{prop: outdegree horizontal} and~\ref{prop: backup} (if $b=1$) or Properties~\ref{prop: indegree vertical}~and~\ref{prop: up} (if $b>1$), each vertex in $V_{(a,b)}$ has at most $m^2$ $G$-neighbours in $V_{(a,b+1)}$. Hence, using Property~\ref{prop: indegree vertical}, we have $|V_{(a,b)}| \le |V_{(a,1)}|\cdot m^{b-1}=n^{t-a-1}m^{k-t+a+b-1}$ for all $a \in [t-1],b\in[h_a]$.
    By a similar calculation to the one above, the number of sets in $\mc{F}'$ for which all the corresponding edges between parts $V_{(a,1)}$ and $V_{(a+1,1)}$ are heavy but an edge between some pair $V_{(a,b)}$, $V_{(a,b+1)}$ is light, is $o(m^{-(k-1)}|\mc{F}|)$. Indeed, if $s\in[k-t]$, then the number of sets in $\FF'$ which are $(s-1)$-completely heavy but not $s$-completely heavy is easily seen to be at most
    $$
    |e_G(V_{(a_s,b_s)},V_{(a_s,b_s+1)})|n^{a_s-3/2}\leq n^{t-a_s-1}m^{k-t+a_s+b_s+1}n^{a_s-3/2}=o\left(m^{-(k-1)}|\mc{F}|\right),
    $$
    giving the count mentioned above.
    Thus, $\mc{F}'$ has $|\mc{F}'|-o(m^{-(k-1)}|\mc{F}|)$ sets $F$ such that between any two adjacent parts $V_{a,b},V_{a',b'}$ (in $\mc{T}$), the two corresponding vertices in $F$ form a heavy edge of $G$.
    However, using Claim \ref{claim: horizontal heavy} and Claim \ref{claim: vertical heavy}, we see that the number of such sets is at most $|V_{(1,1)}|=m^{-(k-1)}|\mc{F}|$. 
    Hence, $|\mc{F}'|\leq (1+o(1))m^{-(k-1)}|\mc{F}|$, finishing the proof.
\end{proof}

It remains to prove \cref{lemma: properties of the sets}, i.e., construct $(V_{p})_{p \in P}$ and $\mc{F}$. 
The construction is quite technical, but it may be helpful for the reader to compare it with the construction used in the proof of \cref{lemma: digraph}, which deals with the case $L=\{1\}$ (i.e., $t=k$, $h_1=\dots=h_k=1$).
\begin{proof}[Proof of \cref{lemma: properties of the sets}]
    Let $\mathcal{X}$ be the collection of sequences $$X=((x_i,y_i)_{i\in [t-1]},(x_{i,j},y_{i,j},z_{i,j})_{i\in[t-1],j\in[2,h_i]})$$ satisfying
\begin{itemize}
    \item  $x_{t-1},y_{t-1}\in[m]$;
    \item  $x_i\in[n]$ and $y_i\in[m]$ for all $i\in[t-2]$;
%    \item $x_{i,2},y_{i,2},z_{i,2}\in[m]$ for all $i\in[t-2]$ with $h_i\geq 2$;
    \item $x_{i,j},y_{i,j},z_{i,j}\in[m]$ for all $i\in[t-1]$ and $j\in[2,h_i]$;%[3,h_i]$;
    \item $y_{i,2}=y_{i}$ whenever $i\in[t-1]$ with $h_i\geq 2$;
    \item $y_{i,j+1}=x_{i,j}$ whenever $i\in [t-1]$ and $j\in[2,h_i-1]$.
\end{itemize}

Every $X \in \mc{X}$ will correspond to a set in our family $\mc{F}$ via the following functions, which will be the $k$ elements in the set.
For each $X \in \mc{X}$ and $p=(a,b) \in P$, let $f_p=f^X_p$ be the function with domain $P$ such that $f_{(a,b)}(t,1) =(a,b)$ and

\begin{align*}
    \textnormal{for } i \in \{1,\dots,a-1\}, \quad\quad\quad f_{(a,b)}(i,j) &=
    \begin{cases}
        y_i & \textnormal{if } j=1, \\
        (x_{i,j},y_{i,j}) & \textnormal{if } j \ge 2;
    \end{cases}  \\[3pt]
    \textnormal{for } i = a, \quad\quad\quad f_{(a,b)}(i,j) &=
    \begin{cases}
        x_i & \textnormal{if } j=1, \\
        (y_{i,j},z_{i,j}) & \textnormal{if } 2 \le j \le b, \\
        x_{i,j} & \textnormal{if } j=b+1, \\
        (x_{i,j}, y_{i,j}) & \textnormal{if } j \ge b + 2;
    \end{cases} \\[3pt]
    \textnormal{for } i\in \{a+1,\dots,t-1\}, \quad\quad\quad f_{(a,b)}(i,j) &=
    \begin{cases}
        x_i & \textnormal{if } j=1, \\
        x_{i,j} & \textnormal{if } j=2, \\
        (x_{i,j}, y_{i,j}) & \textnormal{if } j \ge 3.
    \end{cases}
\end{align*}

Let $F_X=\{f_p^X:p\in P\}$ for each $X\in \mathcal{X}$; this is a set of $|P|=k$ functions.
For each $p\in P$, we define $V_{p}=\{f_p^X:X\in \mathcal{X}\}$, and we let $\mc{F} = \{F_X:X\in \mathcal{X}\}$. 
It is easy to see that $F_X \neq F_{X'}$ for distinct $X,X' \in \mc{X}$, and $\mc{F} $ is $k$-partite with parts $(V_p)_{p\in P}$.
Before checking the properties in \cref{lemma: properties of the sets}, we need the following two auxiliary properties.
These two properties can be verified using the definition of $(f_{p})_{p \in P}$.
For example, to check Property \ref{propi}, one can see that $f_{(a,1)}(i,j)=f_{(a+1,1)}(i,j)$ whenever $i \in [1,a-1]\cup[a+1,t-1]$, so it suffices to look at $i=a$.
The rest is then easy to see.
\begin{enumerate}[(i)]
    \item Whenever $a\in[t-1]$, then $f_{(a,1)}$ and $f_{(a+1,1)}$ agree at every point apart from $(t,1)$, except possibly $(a,1)$ and $(a,2)$ (if $h_a\geq 2$), for which values we have $f_{(a,1)}(a,1)=x_a$, $f_{(a+1,1)}(a,1)=y_a$, $f_{(a,1)}(a,2)=x_{a,2}$, $f_{(a+1,1)}(a,2)=(x_{a,2},y_{a,2})=(x_{a,2},y_{a})$.\label{propi}
    \item Whenever $a\in[t-1]$ and $b\in[2,h_a]$, then $f_{(a,b)}$ and $f_{(a,b-1)}$ agree on every point apart from $(t,1)$, except possibly $(a,b)$ and $(a,b+1)$ (if $h_a\geq b+1)$, for which values we have $f_{(a,b)}(a,b)=(y_{a,b},z_{a,b})$, $f_{(a,b-1)}(a,b)=x_{a,b}$, $f_{(a,b)}(a,b+1)=x_{a,b+1}$, $f_{(a,b-1)}(a,b+1)=(x_{a,b+1},y_{a,b+1})=(x_{a,b+1},x_{a,b})$.\label{propii}
\end{enumerate}

Now, we (briefly) verify the properties required by the lemma.
\begin{description}
    \item[For Property \ref{prop: counts}:] Note that $$
        f_{(1,1)}(i,j) =
        \begin{cases}
            x_i & \textnormal{if } j=1, \\
            x_{i,2} & \textnormal{if } j=2, \\
            (x_{i,j}, y_{i,j})=(x_{i,j},x_{i,j-1}) & \textnormal{if } j > 2.
        \end{cases}
    $$ 
        This means $f_{(1,1)}$ is uniquely determined by all $x_i$s and $x_{i,j}$s, so $|V_{(1,1)}|=m\cdot n^{t-2}\cdot m^{k-t}=n^{t-2}m^{k-t+1}$.
        In addition, $F_X\neq F_{X'}$ whenever $X\neq X'$, so a similar calculation shows $|\mc{F}|=m^2 \cdot n^{t-2} \cdot m^{t-2} \cdot m^{2(k-t)}=n^{t-2}m^{2k-t}$.
    \item[For Property \ref{prop: connected}:] If $f^X_{(a,1)}=f^Y_{(a,1)}$ but $f^X_{(a+1,1)}\not =f^Y_{(a+1,1)}$, then \ref{propi} shows $X$ and $Y$ disagree on the value of $y_a$ and hence $f^X_{a',b'}\not =f^Y_{a',b'}$ whenever $a'\geq a+1$. 
        Similarly, if $f^X_{(a,b)}=f^Y_{(a,b)}$ but $f^X_{(a,b-1)}\not =f^Y_{(a,b-1)}$, then $X$ and $Y$ disagree on the value of $x_{a,b}$ and hence $f^X_{a',b'}\not =f^Y_{a',b'}$ whenever $a'\not = a$ and also whenever $a'=a,b'\leq b-1$. 
    \item[For Property \ref{prop: outdegree horizontal}]: This follows from \ref{propi} for the $m$ choices of $y_a$.
    \item[For Property \ref{prop: indegree horizontal}]: This follows from \ref{propi} for the $n$ choices of $x_{a-1}$ when $a \in [t-2]$ and the $m$ choices of $x_{a-1}$ when $a=t-1$.
    \item[For Property \ref{prop: indegree vertical}]: This follows from \ref{propii} for the $m$ choices of $x_{a,b}$.
    \item[For Property \ref{prop: backup}]: Given $f_{(a,1)},f_{(a+1,1)}$, we know the value of $y_a=y_{a,2}$ (using \ref{propi}). Then, there are $m$ choices for $z_{a,2}$, and hence for $f_{(a,2)}$ by \ref{propii} with $b=2$.
    \item[For Property \ref{prop: branchup}]: If $F_X$ and $F_Y$ give the same $f_{(a,1)}$ but different $f_{(a+1,1)}$, then $X$ and $Y$ disagree on the value of $y_a=y_{a,2}$ (using \ref{propi}). Hence, they give different vertices $f_{(a,2)}$ (using \ref{propii} for $b=2$).
    \item[For Property \ref{prop: up}]: Given $f_{(a,b-1)},f_{(a,b-2)}$, we know the value of $x_{a,b-1}=y_{a,b}$ (by \ref{propii}). Then, there are $m$ choices for $z_{a,b}$ and hence for $f_{(a,b)}$ by \ref{propii}.
    \item[For Property \ref{prop: branchup2}]: If $F_X$ and $F_Y$ give the same $f_{(a,b-1)}$ but different $f_{(a,b-2)}$, then $X$ and $Y$ disagree on the value of $x_{a,b-1}=y_{a,b}$ (using \ref{propii}). Hence, they give different vertices $f_{(a,b)}$ (using \ref{propii} again).
    \qedhere
\end{description}
\end{proof}

\section{Proof of Lemma~\ref{lem:qs:fermionic}}
\label{section: appendix quantum}

Finally, we explain how to obtain the algorithm we need for the applications in quantum computing.

\begin{proof}[Proof of Lemma~\ref{lem:qs:fermionic}]
    By \cref{thm:modularstructure}, there exists some $\chi=\min\{O_k(m^k),|\FF|\}$ such that for any non-negative weighting $w \in \mb{R}_{\ge 0}^{\mc{F}}$ on $\mc{F}$, there is a subfamily $\mc{F}'$ that avoids intersections of size in $L$ and whose weight satisfies $\sum_{F \in \mc{F}'} w(F) \ge \frac{1}{\chi}\sum_{F \in \mc{F}} w(F)$; this $\mc{F}'$ corresponds to an independent set of the associated graph $G_\mc{F}$.
    Moreover, the proof of \cref{thm:modularstructure} guarantees that we can find this $\mc{F}'$ in running time $\mathrm{poly}(n,|\mc{F}|)=\mathrm{poly}(n^k)$.
    Indeed, in the proof, we iteratively shrink the ground set by two each time and keep track of the weighting along the way.
    The key step of this iteration is to locate two elements in the ground set that accumulate sufficiently large weights (see Claim~\ref{claim:modularstructure_1} and surrounding arguments), which runs efficiently in time $\mathrm{poly}(n,|\mc F|)\le\mathrm{poly}(n^k)$.
    Thus, in total, the algorithm has running time $\mathrm{poly}(n^k)$.
    
    As we have seen in \cref{theorem: modular clique fractional chromatic number}, the above implies that the fractional chromatic number of $G_\mc{F}$ is at most $\chi$. By increasing the value of $\chi$ if necessary, we may assume that $\chi\geq 2$.
    In order to efficiently find a ``small'' fractional colouring, it suffices to find $T=20\chi^2\ln(|\mc{F}|)=\mathrm{poly}(n^k)$ independent sets $\mc{I}^{(1)},\dots,\mc{I}^{(T)}$ of $G_\mc{F}$ in running time $\mathrm{poly}(n^k)$ such that if we uniformly sample $i \in [T]$, the random independent set $\mc{I}:=\mc{I}^{(i)}$ satisfies $\Pr[F \in \mc{I}] = \Omega(\chi^{-1})$ for all $F \in \mc{F}$.
    Then, the distribution of $\mc{I}$ is a fractional colouring of $G_\mc{F}$ of size $O(\chi)=O_k(m^k)$.

    To this end, we apply the multiplicative weight update method as follows; see \cite[Section 3.3]{arora2012multiplicative} for the analysis in more general settings.
    \begin{enumerate}
        \item Initialize $w^{(1)}\in\mb{R}_{\ge0}^{\mc F}$ by setting $w^{(1)}_F=\frac{1}{|\mc{F}|}$ for all $F\in\mc F$.
        \item For each step $t=1,\dots,T$, find an independent set $\mc{I}^{(t)}\subseteq \mc{F}$ of $G_\mc{F}$ satisfying \begin{equation} \label{eq: indset with large weight}
            \sum_{F \in \mc{I}^{(t)}} w^{(t)}(F) \ge \frac{1}{\chi}\sum_{F \in \mc{F}} w^{(t)}(F),
        \end{equation} 
        and let $w^{(t+1)} \in \mb{R}^{\mc{F}}_{\ge 0}$ be the weighting given by $$
            w^{(t+1)}(F) = \begin{cases}
                w^{(t)}(F) \cdot \big(1+\frac{1}{4\chi^2} \big) & \text{ if } F \notin \mc{I}^{(t)},\\
                w^{(t)}(F) \cdot \big(1-\frac{1}{4\chi}+\frac{1}{4\chi^2} \big) & \text{ if } F \in \mc{I}^{(t)}.
            \end{cases}
        $$
    \end{enumerate}

    As discussed above, we can find the independent set $\mc{I}^{(t)}$ in running time $\mathrm{poly}(n^k)$.
    So it suffices to show that if we uniformly sample $i \in [T]$, then $\mathbb{P}[F \in \mc{I}^{(i)}] = \Omega(\chi^{-1})$ for all $F \in \mc{F}$.
    In other words, for each $F \in \mc{F}$, at least $\Omega(T/\chi)$ of the independent sets among $\mc{I}^{(1)},\dots,\mc{I}^{(T)}$ (with multiplicity) contain $F$.

    We first observe that $w^{(T+1)}(F) \le 1$ holds for all $F \in \mc{F}$.
    To see this, for any $t \in [T]$,
    \begin{align*}
        \sum_{F \in \mc{F}} w^{(t+1)}(F)
        &= \sum_{F \in \mc{I}^{(t)}} w^{(t+1)}(F) + \sum_{F \notin \mc{I}^{(t)}} w^{(t+1)}(F)\\
        &= \sum_{F \in \mc{F}} w^{(t)}(F) \cdot \left(1 + \frac{1}{4\chi^2} \right) - \frac{1}{4\chi}\sum_{F \in \mc{I}^{(t)}} w^{(t)}(F)
        \le \sum_{F \in \mc{F}} w^{(t)}(F),
    \end{align*}
    where we used \eqref{eq: indset with large weight} in the last inequality.
    Hence, $\sum_{F \in \mc{F}} w^{(T+1)}(F) \le\cdots\le\sum_{F \in \mc{F}} w^{(1)}(F)=1$.
    In particular, $w^{(T+1)}(F) \le 1$ for all $F \in \mc{F}$ since the weights are clearly all non-negative.

    Now, suppose $a$ of the independent sets among $\mc{I}^{(1)},\dots,\mc{I}^{(T)}$ contain $F$.
    Then, using that $1+x \ge e^{x/2}$ and $1-x \ge e^{-2x}$ for all $x \in [0,1/2]$, we know $$
        \begin{aligned}
            w^{(T+1)}(F) 
            = w^{(1)}(F)\cdot \Big(1+\frac{1}{4\chi^2}\Big)^{T-a}\Big(1-\frac{1}{4\chi}+\frac{1}{4\chi^2}\Big)^a
            \geq \frac{1}{|\mc{F}|} \cdot \Big(1+\frac{1}{4\chi^2}\Big)^{T-a} \Big(1-\frac{1}{2\chi}\Big)^{2a}
            \ge \frac{1}{|\mc{F}|} \cdot e^{\frac{T-a}{8\chi^2} - \frac{2a}{\chi}}.
        \end{aligned}
    $$
    We know that $w^{(T+1)}(F) \le 1$, so $\frac{T-a}{8\chi^2}-\frac{2a}{\chi} \le \ln|\mc{F}|$.
    Recall $T=20\chi^2\ln|\FF|$.
    Solving this inequality, we see that 
    $$
    a \ge\frac{12\chi^2\ln|\mc{F}|}{16\chi+1} \ge \frac{2\chi\ln|\mc{F}|}{3}=\Omega(T/\chi).
    $$
    Therefore every $F$ is contained in at least $\Omega(T/\chi)$ independent sets among $\mc{I}^{(1)}, \dots, \mc{I}^{(T)}$ (with multiplicity).
\end{proof}

\end{document}